\documentclass{amsart}
\usepackage{amsmath}

\usepackage{subcaption}

\usepackage{here}
\usepackage[foot]{amsaddr}
\usepackage{graphicx,url}
\usepackage{fullpage}
\usepackage{color}
\usepackage{dsfont}
\usepackage{tikz}
\usepackage{float,here}
\usetikzlibrary{decorations.pathmorphing,patterns,scopes,intersections,calc,snakes}
\everymath{\displaystyle}
\def\R{\mathbb{R}}

\def\N{\mathbb{N}}

\newtheorem{Theorem}{Theorem}

\newtheorem{defi}{Definition}

\newtheorem{Lemma}{Lemma}

\newtheorem{Remark}{Remark}

\numberwithin{equation}{section}

\RequirePackage[bookmarks,%
                colorlinks,%
                urlcolor=blue,%
                citecolor=blue,%
                linkcolor=blue,%
                hyperfigures,%
                pdfcreator=LaTeX,%
                breaklinks=true,%
                pdfpagelayout=SinglePage,%
                bookmarksopen=true,%
                bookmarksopenlevel=2]{hyperref}

\title[Numerical study of a wave equation with internal localized Kelvin-Voigt damping]{Numerical analysis of a finite volume method for a $1$-d wave equation with non smooth wave speed and localized Kelvin-Voigt damping}

\author{St\'ephane Gerbi$^1$}
\address{$^1$Laboratoire de Math\'ematiques UMR 5127 CNRS \& Universit\'e de Savoie Mont Blanc, Campus scientifique, 73376 Le Bourget du Lac Cedex, France}
\email{stephane.gerbi@univ-smb.fr}

\author{Rayan Nasser$^{2,3,4}$}
\address{$^2$Lebanese University, Faculty of sciences 1, Khawarizmi Laboratory of  Mathematics and Applications-KALMA, Hadath-Beirut, Lebanon
}
\address{$^3$PDE's Theoretical and Numerical Studies, Department of Mathematics and Physics, Lebanese International University (LIU), Beirut, Lebanon}

\address{$^4$PDE's Theoretical and Numerical Studies, Department of Mathematics and Physics, The International University of Beirut (BIU), Beirut, Lebanon}
\email{rayan.nasser@liu.edu.lb}

\author{Ali Wehbe$^2$}

\email{ali.wehbe@ul.edu.lb}

\usepackage[english]{babel}

\begin{document}

\pagenumbering{arabic}
\begin{abstract}
In this paper, we study the numerical solution of an elastic/viscoelastic wave equation with non smooth \textcolor{black}{wave speed and internal localized}  distributed Kelvin-Voigt damping acting faraway from the boundary. Our method is based on the Finite Volume Method (FVM) and we are interested in deriving the stability estimates and the convergence of the numerical solution to the continuous one. Numerical experiments are performed to confirm the theoretical study on the decay rate of the
solution to the null one when a localized damping acts.
\end{abstract}

\maketitle

\tableofcontents

\section{Introduction}
\noindent Numerical simulation is a computation that implements a mathematical model for a physical system. It is required to study the behavior of systems representing complicated mathematical models in order to provide good approximations to the analytical solutions. Numerical modeling uses mathematical models to describe the physical conditions necessary especially for engineers. Still, some of their equations, mostly partial differential equations, are somehow impossible to solve directly. With numerical models, the approximate solutions can be obtained; for example, by Finite Difference Method (FDM), Finite Element Method (FEM) or Finite Volume Method (FVM), with each method having its own merits and limitations.  Nevertheless, several results are interpreted within the framework of an engineering process concerning the numerical experiments carried out in these models. However, as mathematicians, numerical studies help us not only in confirming the theoretical results, but also in predicting the solutions of some conjectures. In fact, there are many numerical studies in the literature related to FDM, FEM and FVM. For example, Larsson et al. applied FEM on a strongly damped wave equation (see \cite{Larsson1991}). In \cite{zuazua2003}, Zuazua et al. considered a problem that models the damped vibrations of a string with fixed ends. Assuming that the damping is localized on a subinterval, FDM  with artificial viscosity was applied to show  the exponential decay of the discrete energy and to obtain the convergence of the scheme. However, the error estimates were obtained later on by Rincon et al. after applying a spatial FEM (see \cite{Rincon2013}). Moreover, Gao et al., in \cite{Gao2007}, presented an unconditionally stable finite difference scheme to find the solution of a one-dimensional linear hyperbolic equation with global damping. Also, Wang et al. used high order FDM to solve the wave equation in the second order form in two space dimensions (see \cite{Wang2016}). Furthermore, Xu et al., in \cite{Xin2014}, considered the Euler-Bernoulli beam equation with local Kelvin-Voigt damping \textcolor{black}{acting} via nonsmooth coefficient. Using FEM followed by the control parameterization method, they aim to design a control input numerically distributed locally on a subinterval such that the total energy of the beam and the control on a given time period is minimal. As to FVM, which is based on an integral formulation, it is very popular in solving linear hyperbolic equations. For instance, LeVeque, in \cite{Leveque}, used  FVM for hyperbolic equations. Also, considering that wave equations are important hyperbolic equations that arise in acoustics, numerical studies were carried on wave-based modeling of room acoustics using FDM and FVM (see \cite{HamiltonB}). In addition, in \cite{SATO1994}, explicit numerical simulation has been developed for time dependent viscoelastic flow problems using a combination between FVM  and FEM. Zhang et al., in \cite{WZhang2016}, illustrated a new spectral FVM for a $2$-d and $3$-d elastic wave equations with external sources on unstructured meshes. Zhang et al. also presented a new efficient FVM for $3$-d elastic wave simulation on unstructured tetrahedral meshes (see \cite{ZHANG2017}). We also mention Rie\v{c}anov\'{a} et al. \cite{riecanova2018}, where the authors obtained the stability estimates and convergence of the numerical scheme of a wave equation with Dirichlet boundary condition on a rectangular domain. Their  method was based on FVM in space together with the average of $n + 1$ and $n-1$ time step diffusion. However, in the one-dimensional case, some authors ensured the decay of energy and gave examples that verify its asymptotic behavior by implementing numerical schemes (see \cite{Alves-Rivera-Sepulveda-Villagran-Garay-2013,Alves-Rivera-Sepulveda-Villagran-2014,Raposo-Bastos-Avila-2011,Maryati-Rivera-Rambaud-2018}). Thus, to our knowledge, it seems that there are no results in the literature on the convergence and stability estimates, concerning the transmission problem of  elastic/viscoelastic systems  where there is a discontinuity at the interface, based especially on FVM. 

\noindent The main objective of this research work is to fill this gap. More precisely, we study the numerical solution of the transmission problem of a wave equation with localized Kelvin-Voigt damping acting faraway from the boundary via nonsmooth coefficient \textcolor{black}{as described in Figure \ref{figure-material}
\begin{figure}[H]
\centering{\includegraphics[width=0.75\textwidth]{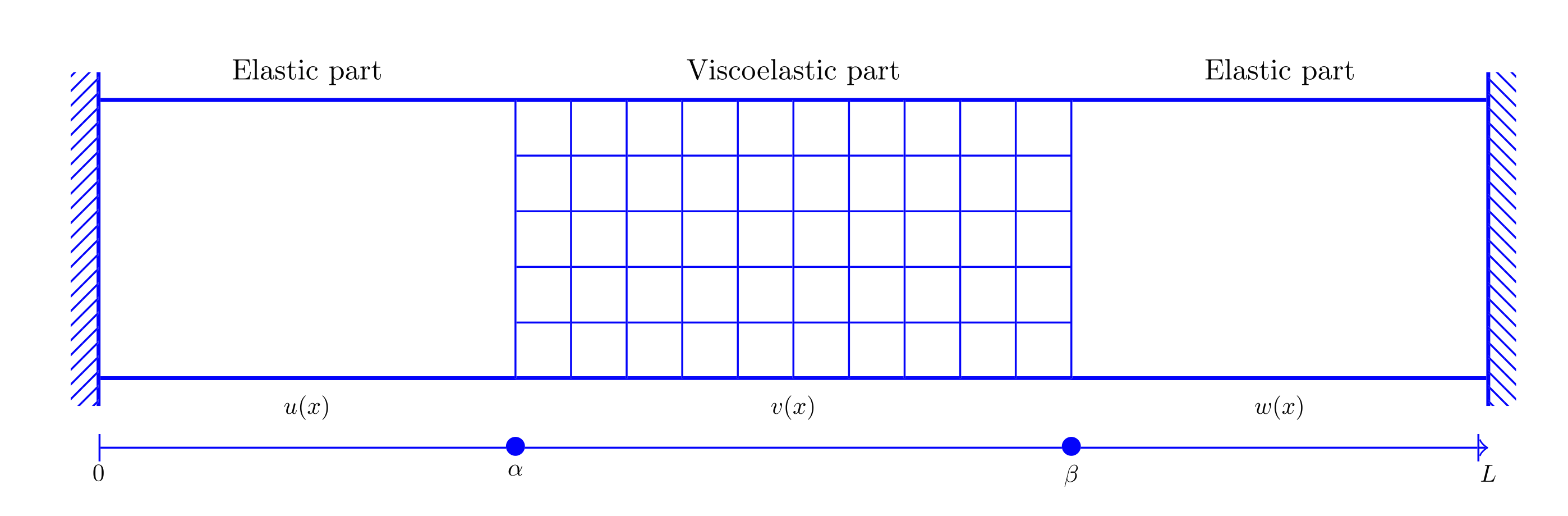}}
\caption{Partial viscoelastic material}
\label{figure-material}
\end{figure}
}

Consider the following transmission problem of a wave equation with localized Kelvin-Voigt type damping
 \begin{eqnarray}
\rho_1 u_{tt} - \kappa_1 u_{xx} = 0, &(x,t) \in  (0,\alpha) \times (0, + \infty), \label{Eqqq(2.1)}\\
\rho_2 v_{tt} - \kappa_2 v_{xx} - \varrho v_{xxt} = 0, &(x,t) \in (\alpha, \beta ) \times (0, + \infty),\label{Eqqq(2.2)}\\
\rho_3 w_{tt} - \kappa_3 w_{xx} = 0, &(x,t) \in  ( \beta, L )  \times (0, + \infty). \label{Eqqq(2.3)}
\end{eqnarray}
Here $\rho_1, \ \rho_2, \ \rho_3, \kappa_1,\ \kappa_2$ and $\kappa_3$ are strictly positive constant numbers representing the density and metric coefficients of Equations \eqref{Eqqq(2.1)}, \eqref{Eqqq(2.2)} and \eqref{Eqqq(2.3)} respectively and the damping constant $\varrho$ is strictly positive. In fact, we divide Equations \eqref{Eqqq(2.1)}-\eqref{Eqqq(2.3)} by $\rho_1, \ \rho_2$ and $\rho_3$ respectively to obtain
 \begin{eqnarray}
u_{tt} - C_1^2 u_{xx} = 0, &(x,t) \in  (0,\alpha) \times (0, + \infty), \label{Eq(2.1)}\\
v_{tt} - C_2^2 v_{xx} - \delta v_{xxt} = 0, &(x,t) \in (\alpha, \beta ) \times (0, + \infty),\label{Eq(2.2)}\\
 w_{tt} - C_3^2 w_{xx} = 0, &(x,t) \in  ( \beta, L )  \times (0, + \infty). \label{Eq(2.3)}
\end{eqnarray}
System \eqref{Eq(2.1)}-\eqref{Eq(2.3)} is subjected to the Dirichlet boundary conditions
\begin{equation}\label{Eq(2.4)}
u(0,t) = w(L,t) = 0, \quad  t\in (0,+\infty)
\end{equation}
and the following initial conditions
\begin{equation} \label{initial-condition-u-v-w}
\left\{
\begin{array}{ll}
u(x, 0)=\varphi(x), \quad & x\in  (0,\alpha), \\
v(x, 0)=\eta(x), \quad &x\in  (\alpha,\beta),\\
w(x, 0 )=\gamma(x), \quad &x\in   (\beta,L) \ ,
\end{array}
\right.
\end{equation}
\begin{equation} \label{initial-condition-ut-vt-wt}
\left\{
\begin{array}{ll}
u_t(x, 0)=\psi(x), \quad & x\in  (0,\alpha), \\
v_t(x, 0)=\zeta(x), \quad &x\in  (\alpha,\beta),\\
w_t(x, 0 )=\theta(x), \quad &x\in   (\beta,L) \ .
\end{array}
\right.
\end{equation}
We set $\Phi = (\varphi, \eta, \gamma)$ and $\Psi = (\psi, \zeta, \theta)$; $C_1 = \sqrt{\frac{\kappa_1}{\rho_1}}$, $C_2 = \sqrt{\frac{\kappa_2}{\rho_2}}$ and $C_3 = \sqrt{\frac{\kappa_3}{\rho_3}}$ represent the speed of the wave propagation of  Equations \eqref{Eq(2.1)}-\eqref{Eq(2.3)} respectively. Also, $\delta = \frac{\varrho}{\rho_2}$. \\[0.1in]

The transmission conditions are given by
\begin{equation}\label{Eq(2.5)}
\left\{
\begin{array}{ll}
\displaystyle{
u(\alpha,t ) = v(\alpha,t), \quad v (\beta , t) = w (\beta , t)}, &t\in  (0, + \infty), \\
\displaystyle{C_2^2 v_x (\alpha , t) + \delta v_{xt} (\alpha,t ) = C_1^2 u_x (\alpha ,t)}, &t\in  (0, + \infty),\\
\displaystyle{C_2^2 v_x (\beta , t) + \delta v_{xt} (\beta,t ) = C_3^2 w_x (\beta ,t),} &t\in  (0, + \infty).
\end{array}
\right.
\end{equation}
\begin{defi}[\textbf{Weak solution}] \label{weak-solution} 
Let us define the functional spaces:
\begin{equation*}
\begin{array}{ll}
\displaystyle{ \mathbb{H}^1 =  H^1 (0,\alpha) \times H^1 (\alpha, \beta) \times H^1 (\beta, L),  }\\ 
\displaystyle{\mathbb{L}^2 =  L^2 (0,\alpha) \times L^2 (\alpha, \beta) \times L^2(\beta,L)},\\
\displaystyle{  \mathbb{H}^1_L = \{(u,v,w) \in \mathbb{H}^1\ | \ u (0) = w(L) = 0, \ u(\alpha) = v (\alpha), \ v(\beta) = w(\beta) \}.}
\end{array}
\end{equation*}
We say that $U = (u,v,w)$ is a weak solution of the System 
\eqref{Eq(2.1)}-\eqref{Eq(2.5)} 
if for all $T>0$, the following conditions hold:
\begin{itemize}
\item[1.] $U = (u,v,w) \in L^2([0,T];\mathbb{H}^1_L)$ and $U_t=(u_t,v_t,w_t) \in L^2([0,T]; \mathbb{L}^2)$,
\item[2.] $(u (.,0), v(.,0), w(.,0)) = (\varphi(.), \eta(.), \gamma(.))$,
\item[3.] \begin{equation}\label{continuous weak variational problem}
\begin{array}{lll}
-  \int_{0}^{T} \int_{0}^{\alpha} u_t p_t dx dt + C_1^2 \int_{0}^{T} \int_{0}^{\alpha} u_x p_x dx dt -  \int_{0}^{\alpha} \psi(x)p(x,0) dx  \\
-  \int_{0}^{T} \int_{\alpha}^{\beta} v_t q_t dx dt + C_2^2 \int_{0}^{T} \int_{\alpha}^{\beta} v_x q_x dx dt + \delta  \int_{0}^{T} \int_{\alpha}^{\beta} v_{tx}q_x dx dt -  \int_{\alpha}^{\beta} \zeta(x)q(x,0) dx   \\
-  \int_{0}^{T} \int_{\beta}^{L} w_t z_t dx dt + C_3^2 \int_{0}^{T} \int_{\beta}^{L} w_x z_x dx dt -  \int_{\beta}^{L} \theta(x)z(x,0) dx   = 0
\end{array}
\end{equation}
for all $(p,q,z) \in L^2([0,T];\mathbb{H}^1_L)$ and $(p_t,q_t,z_t) \in L^2([0,T]; \mathbb{L}^2)$ such that 
\[
p(T,x)=0 \,,\, q(T,x)=0 \mbox{ and } z(T,x) =0 \quad .
\]
\end{itemize}
\end{defi}
We study our problem under the following hypothesis $\eqref{H'}$:
\begin{equation} \label{H'}\tag{H$^{\prime}$}
\Phi \in \mathbb{H}^1_L, \quad \Psi \in \mathbb{H}^1.
\end{equation}
 \textcolor{black}{The energy of the solutions of  the System \eqref{Eq(2.1)}-\eqref{Eq(2.5)} is defined by 
\begin{equation}\label{energy}
E(t) = \dfrac{1}{2} \int_0^\alpha\vert{u_t}\vert^2 + \dfrac{C_1^2}{2} \int_0^\alpha\vert{u_x}\vert^2 +
\dfrac{1}{2} \int_\alpha^\beta\vert{v_t}\vert^2 + \dfrac{C_2^2}{2} \int_\alpha^\beta\vert{v_x}\vert^2 +
\dfrac{1}{2} \int_\beta^L\vert{w_t}\vert^2 + \dfrac{C_3^2}{2} \int_\beta^L\vert{w_x}\vert^2
\end{equation}
Let $U = (u,v,w) \in L^2([0,T];\mathbb{H}^1_L)$ and $U_t=(u_t,v_t,w_t) \in L^2([0,T]; \mathbb{L}^2)$ be a weak solution of the system
\eqref{Eq(2.1)}-\eqref{Eq(2.5)}. Then we have the following dissipativity estimation:
\begin{equation}\label{dissipative}
	\forall t > 0 \,,\, \dfrac{d E(t)}{d t} = -\delta \int_\alpha^\beta\vert{v_{xt}}\vert^2
\end{equation}
The theoretical study of such systems was done in previous literature (see \cite{GhNAWE, Nasser-Noun-Wehbe-2020, Ghader-Nasser-Wehbe-delay, Ammari-Hassine-Robbiano-2019, Rivera-Villagran-Sepulveda-2018, Ammari-Liu-Shel-2019}). 
In \cite{GhNAWE}, the authors studied the asymptotic behavior of the energy of the continuous case of system \eqref{Eq(2.1)}-\eqref{Eq(2.5)} and obtained an optimal energy decay rate of type $t^{-4}$. In this article, our objective is to study this system numerically. To this aim, in Section \ref{Section Mesh}, we discretize our system using an admissible mesh and set the explicit numerical scheme in 
Subsection \ref{SE} and the semi-implicit numerical scheme with an average of $n+1$ and $n-1$ time step diffusion in Subsection \ref{SI} using FVM in space.
However, in Subsections \ref{SEE} and \ref{S(IE)}, we design a discrete energy that dissipates when the control is acting \textit{i.e.} $\delta \neq 0$, and is conserved during its
absence. 
Sections \ref{Stability section} and \ref{Convergence section} are devoted to stability estimates and convergence of the discrete solution to the continuous one.
Finally, in Section \ref{Numerical Experiments}, we give some numerical examples which demonstrate the theoretical results obtained in \cite{GhNAWE}.\\[0.1in]
We also mention that in \cite{GhNAWE} two of the authors have proved the existence and uniqueness of a strong solution $U = (u,v,w) \in L^2([0,T];\mathbb{H}^1_L)$ and $U_t=(u_t,v_t,w_t) \in L^2([0,T]; \mathbb{L}^2)$. So the existence and uniqueness of the weak solution is thus verified.
}
\section{Construction of the discretization of the problem} \label{Section Mesh}
\noindent In this section, we will firstly define the admissible mesh and the discrete unknowns according to the Finite Volume Method. 

Then we will construct an explicit in time Finite Volume discretization. We will build a discrete energy and will prove that this discrete energy is decreasing in time.

Finally we will construct an implicit in time Finite Volume discretization. Again, we will build a discrete energy and will prove that this discrete energy is decreasing in time.

\subsection{Admissible one-dimensional mesh and defintition of the discrete unknowns} An admissible mesh $\mathcal{T}$ of the interval $(0,L)$ is given by a family $\{ K_i: \ i \in \{1, \cdots, N_{\max} \} ; ~\ N_{\max} \in \N^* \}$ of control volumes such that $K_i = ( x_{i - \frac{1}{2} } , x_{i+  \frac{1}{2}} )$ and a family $( x_i)_{i = 0, \cdots, N_{\max}+1 }$ assumed to be the center of $(K_i)_{i = 1, \cdots, N_{\max} }$ such that
\begin{equation*}
 0 = x_0 = x_{\frac{1}{2}}<  x_1 < x_{\frac{3}{2}} < x_2 < \cdots < x_{i-\frac{1}{2}} < x_i < x_{i+ \frac{1}{2}} < \cdots < x_{N_{\mbox{max}}} < x_{N_{\mbox{max}} + \frac{1}{2}} = x_{N_{\mbox{max}} + 1}  = L.  
 \end{equation*}
We discretize the intervals $[0,\alpha]$, $[\alpha, \beta]$ and $[\beta,L]$ into $N_{\alpha}$, $N$ and $N_{\beta}$ internal points respectively such that $N_{\alpha}, N, N_{\beta} \in \N^*$ and $N_{\max} = N_{\alpha} + N + N_{\beta}$. To be clear, let $h_{\alpha} = \alpha/ N_{\alpha}, \ h = (\beta - \alpha)/N $ and $h_{\beta} = (L - \beta)/ N_ \beta$ and discretize as the following: \\[0.1in]
We discretize $[0, \alpha]$ such that: 
\begin{itemize}
\item For $i = 0, \ldots,  N_{\alpha}$, $x_{i + \frac{1}{2}} = i h_{\alpha},$ which yields, $x_{\frac{1}{2}} = 0$ and $x_{N_{\alpha}+ \frac{1}{2}} = \alpha.$
\item For $i = 1, \ldots, N_{\alpha}$, $x_i = \bigg( i - \frac{1}{2} \bigg) h_{\alpha}$, which yields, $x_1 = \frac{h_{\alpha}}{2}$ and $x_{N_{\alpha}} = \alpha - \frac{h_{\alpha}}{2}$.
\end{itemize}
\noindent We discretize $[\alpha, \beta]$ such that: 
\begin{itemize}
\item For $i = N_{\alpha}+1, \ldots, N_{\alpha}+N$, $x_{i + \frac{1}{2}} = \alpha + (i - N_{\alpha})h,$ which yields, $x_{N_{\alpha}+\frac{3}{2}} =  \alpha + h$ and $x_{N_{\alpha}+ N+ \frac{1}{2}} = \beta$.
\item For $i = N_{\alpha}+1, \ldots, N_{\alpha}+N$, $x_i = \alpha +  \bigg( i - N_{\alpha} -\frac{1}{2}\bigg)h$, which yields, $x_{N_{\alpha}+1} = \alpha + \frac{h}{2}$ and $x_{N_{\alpha}+N} = \beta - \frac{h}{2}$.
\end{itemize}
\noindent We discretize $[\beta,L]$ such that: 
\begin{itemize}
\item For $i = N_{\alpha}+N+1, \ldots, N_{\max}$, $x_{i + \frac{1}{2}} = \beta+ (i - N_{\alpha}-N)h_{\beta},$ which yields, $x_{N_{\alpha}+N+\frac{3}{2}} =  \beta + h_{\beta}$ and $x_{N_{\max} + \frac{1}{2}} =L$.
\item For $i = N_{\alpha}+N+1, \ldots, N_{\max}$, $x_i = \beta+ \bigg(i - N_{\alpha}-N - \frac{1}{2} \bigg)h_{\beta}$, which yields, $x_{N_{\alpha}+N+1} = \beta + \frac{h_{\beta}}{2}$ and $x_{N_{\max}} = L - \frac{h_{\beta}}{2}$.
\end{itemize}
\noindent Now, we set
\begin{itemize}
\item $ h_{i} = x_{i + \frac{1}{2}} - x_{i - \frac{1}{2}}; \ i = 1 , \ldots, N_{\max},$ 
\item $h_{i+ \frac{1}{2}} = x_{i+1} - x_i; \ i = 0 , \ldots, N_{\max}$,
\item $h_{i- \frac{1}{2}} = x_{i} - x_{i-1}; \ i = 1 , \ldots, N_{\max}+1$,
\item $h_i^{+} = x_{i+\frac{1}{2}} - x_i; \ i = 1 , \ldots, N_{\max}$,
\item $h_i^{-} = x_i - x_{i-\frac{1}{2}}; \ i = 1 , \ldots, N_{\max}$,
\item size$(\mathcal{T}) = \max \{ h_i, \ i =1 , \ldots, N_{\max} \} $.
\end{itemize}


\begin{figure}[!tbp]
	\includegraphics[width=\textwidth]{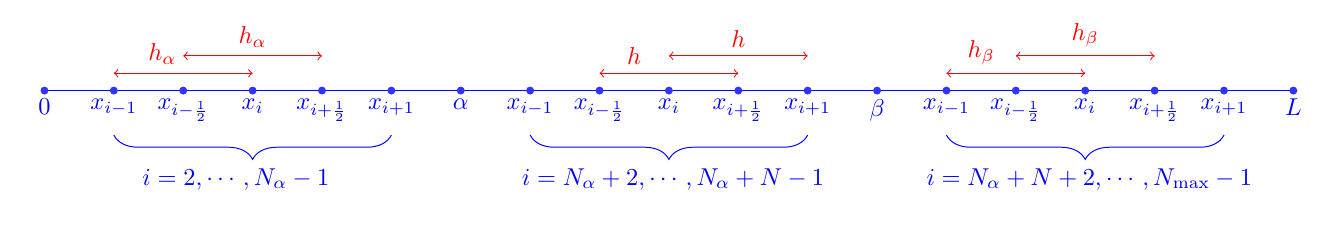}
		\caption{A model representing the admissible one-dimensional mesh}\label{fig1}
\end{figure}

\noindent The interval $(0,L)$ is discretized using the admissible mesh provided above and we define the space discretization mesh as :
\begin{align*}
	h_i = \left\{\begin{array}{ll}
		h_\alpha, & i = 1, \ldots, N_{\alpha},\\
		h, & i = N_{\alpha}+1, \ldots, N_{\alpha}+N,\\
		h_\beta, & i = N_{\alpha}+N+1, \ldots, N_{\max}.
	\end{array} \right. \ .
\end{align*}

The time discretization is performed with a constant time step {\small{$\Delta t = \frac{T}{\mathcal{N}}$}} with $t_{n+1} - t_n = \Delta t$ for all $n=0, 1, \cdots, \mathcal{N}$ so that for a function $v$ defined on $[0,L] \times [0,T]$, $v^n$ represents the value of the function at time $t_n$

We denote the first and second discrete time derivatives following a backward difference and a second-order central difference at time $t_n$ as:
\begin{equation*}
\partial^1 v^n = \frac{v^{n} - v^{n-1}}{\Delta t} \quad \textrm{and} \quad  \partial^2 v^n = \frac{v^{n+1} - 2v^{n} + v^{n-1}}{\Delta t^2} \,,\,
\end{equation*}
and the centered time discretization is denoted by:
\begin{equation}
\partial^{1/2} v^n = \frac{v^{n+1} - v^{n-1}}{2 \Delta t} \label{centered-time} \ .
\end{equation}
Now, for the time step $\Delta t$, we set the following condition 
\begin{equation} \label{TA}\tag{TA}
\textrm{There exists} \ \tau_0 \in [0,T] \ \textrm{such that} \ \Delta t \leq \tau_0.
\end{equation}

We designate the discrete unknowns by $ \{  u_{i}^n; ~\ i = 1, \cdots,  N_{\alpha}  , ~\ n \in \N \}$, $\{  v_{i}^n; ~\ i = N_{\alpha} + 1 , \cdots,  N_{\alpha} + N, ~\ n \in \N \}$, $\{  w_{i}^n; ~\ i = N_{\alpha} + N + 1 , \cdots, N_{\max}, ~\  n \in \N \}$ where they stand for approximation of the mean values of $u, \ v , \ w$ over the control volumes $K_i$ respectively at time $t_n$ defined by:
\begin{eqnarray}
u_i(t) &=& \frac{1}{h_i} \int_{x_i-\frac{1}{2}}^{x_i+ \frac{1}{2}} u(x,t) dx,  \quad i = 1,  \ldots, N_{\alpha}\,,\, t \in [0,T]\label{u} \\ 
v_i(t) &=& \frac{1}{h_i} \int_{x_i-\frac{1}{2}}^{x_i+ \frac{1}{2}} v(x,t) dx,  \quad i =  N_{\alpha}+1, \ldots, N_{\alpha}+N  \,,\, t \in [0,T]\label{v} \\
w_i(t) &=& \frac{1}{h_i} \int_{x_i-\frac{1}{2}}^{x_i+ \frac{1}{2}} w(x,t) dx,  \quad i =  N_{\alpha}+N+1, \ldots, N_{\max}  \,,\, t \in [0,T]\label{w} \\
u_i^n &=& u_i(t_n) \quad i = 1,  \ldots, N_{\alpha}, \nonumber \\ 
v_i^n &=& v_i(t_n) \quad i =  N_{\alpha}+1, \ldots, N_{\alpha}+N,  \nonumber \\
w_i^n &=& w_i(t_n) \quad i =  N_{\alpha}+N+1, \ldots, N_{\max}.  \nonumber
\end{eqnarray}

We next define :  
\begin{align*}
U_i^n = \left\{\begin{array}{ll}
u_i^n, & i = 1, \ldots, N_{\alpha},\\
v_i^n, & i = N_{\alpha}+1, \ldots, N_{\alpha}+N,\\
w_i^n, & i = N_{\alpha}+N+1, \ldots, N_{\max}.
\end{array}\right.  \; .
\end{align*}
It is completed by the Dirichlet boundary conditions : $u_0^n = U_0^n = 0 \,,\, w_{N_{\max} +1}^n = U_{N_{\max} +1}^n = 0$.

The reconstructed solution is defined by  $U_{\mathcal{T},\Delta t}$ constant over the control volumes $K_i$ at time $t \in [t_n,t_{n+1})$ as:
\begin{equation}\label{UTN}
\textrm{ for } t \in [t_n,t_{n+1}) \,,\, \textrm{ for } x \in K_i \,,\, U_{\mathcal{T},\Delta t}(x,t) = U_i^n 
\end{equation}
For a function of only one space variable, we denote $U_{\mathcal{T}}$ the reconstructed solution constant over the control volumes $K_i$: 
\begin{equation}\label{UT}
\textrm{ for } x \in K_i \,,\, U_{\mathcal{T}}(x) = U_i 
\end{equation}

The initial solution $U_{\mathcal{T}}^0$ is constructed according to the initial condition \eqref{initial-condition-u-v-w}:
\begin{equation}\label{initialSOLUTION}
\left\{
\begin{array}{ll}
u_i^0 = \varphi_i \quad \textrm{where} ~\ \varphi_i = \frac{1}{h_i} \int_{K_i} \varphi(x) dx, \quad & \textrm{for} \ i = 1, \cdots, N_{\alpha}, \\
v_i^0 = \eta_i \quad \textrm{where} ~\ \eta_i = \frac{1}{h_i} \int_{K_i} \eta(x) dx, \quad & \textrm{for} \ i = N_{\alpha}+1, \cdots, N_{\alpha}+N, \\
w_i^0 = \gamma_i \quad \textrm{where} ~\ \varphi_i = \frac{1}{h_i} \int_{K_i} \gamma(x) dx , \quad & \textrm{for} \ i = N_{\alpha}+N+1, \cdots, N_{\textrm{max}} \\
\Phi_i = (\varphi_i, \eta_i, \gamma_i) .
\end{array}
\right.
\end{equation}
We set $\Phi_{\mathcal{T}} = (\Phi_i )_{i=1,\ldots,N_{\max}}$ so that the initial solution is defined by:
\begin{equation}\label{U0}
U^0 = \Phi_{\mathcal{T}}
\end{equation}

In order to discretize the initial condition \eqref{initial-condition-ut-vt-wt}, we define a ``ghost'' time-boundary (i.e. $t_{-1} = - \Delta t$) and
we use the second-order central difference formula \eqref{centered-time} at time $t=0$. Indeed as we have :
\[
\dfrac{\partial \displaystyle \int_{K_{i}} U(x,t)dx} {\partial t} = \int_{K_{i}} \dfrac{\partial U(x,t)}{\partial t} dx
\]
we define:
\begin{eqnarray}
u_i^{-1} &\textrm{for}& i = 1 , \cdots, N_{\alpha}, \nonumber \\
v_i^{-1}  &\textrm{for}& i = N_{\alpha}+1 , \cdots, N_{\alpha}+N, \nonumber \\
w_i^{-1}&\textrm{for}& i = N_{\alpha}+N+1  \cdots, N_{\textrm{max}}. \nonumber
\end{eqnarray}
and we set $\Psi_i = (\psi_i, \zeta_i, \theta_i)$ such that :
\begin{equation*}
\begin{array}{lll}
 \psi_i = \frac{1}{h_i} \int_{K_i} \psi(x) dx, \quad &\textrm{for} \ i = 1, \cdots, N_{\alpha}, \\
 \zeta_i = \frac{1}{h_i} \int_{K_i} \zeta(x) dx, \quad & \textrm{for} \ i = N_{\alpha}+1, \cdots, N_{\alpha}+N, \\
 \theta_i = \frac{1}{h_i} \int_{K_i} \theta(x) dx , \quad & \textrm{for} \ i = N_{\alpha}+N+1, \cdots, N_{\max}.
\end{array}
\end{equation*}
Therefore, using the above notations, we can approximate the initial conditions of $u_t$, $v_t$ and $w_t$ using centered time approximation:
\[
\Psi(x) = \dfrac{\partial U}{\partial t} (x,0) \simeq \dfrac{U(x,\Delta t) - U(x,-\Delta t)}{ 2 \Delta t}
\]
to get 
\begin{equation*}
\begin{array}{lll}
u_i^{-1} = u_i^1 -2 \Delta t \psi_i , \quad & \textrm{for} \ i = 1, \cdots, N_{\alpha}, \\
v_i^{-1} =  v_i^1 -2 \Delta t \zeta_i, \quad & \textrm{for} \ i = N_{\alpha}+1, \cdots, N_{\alpha}+N, \\
w_i^{-1} = w_i^1 -2 \Delta t \theta_i,  \quad & \textrm{for} \ i = N_{\alpha}+N+1, \cdots, N_{\max} \\
\Psi_i = (\psi_i, \zeta_i, \theta_i) .
\end{array}
\end{equation*}
Consequently, we set $\Psi_{\mathcal{T}} = (\Psi_i )_{i=1,\ldots,N_{\max}}$ so that the solution at the ``ghost'' time-boundary is defined by:
\begin{equation}\label{artificialSOLUTION}
U^{-1} = U^1 -2 \Delta t \Psi_{\mathcal{T}} \ .
\end{equation}

As the speed $C$ is defined by
\begin{equation*}
C=\left\{\begin{array}{lllll}
C_1,& x\in (0,\alpha),\\
C_2,& x\in (\alpha,\beta),\\
C_3, & x\in (\beta, L)
\end{array}\right. 
\end{equation*}
we define the discrete speeds as:
\begin{align*}
c_i = C_1 \quad & \textrm{for} \ i = 1, \cdots, N_{\alpha}, \\
c_i = C_2 \quad & \textrm{for} \ i = N_{\alpha}+1, \cdots, N_{\alpha}+N, \\
c_i = C_3 \quad & \textrm{for} \ i = N_{\alpha}+N+1, \cdots, N_{\max}.
\end{align*}
Now, we will construct the numerical schemes.
\subsection{Construction of the explicit discretized problem.} \label{SE}
In this part, we will construct the numerical scheme to find the solution of System \eqref{Eq(2.1)}-\eqref{Eq(2.5)} using the Finite Volume Method (FVM) in space and finite difference method in time. The principle of FVM is to integrate the equations \eqref{Eq(2.1)}-\eqref{Eq(2.2)}-\eqref{Eq(2.3)}  over the controle volume $K_i$. A flux at the boundaries of the control volume will come up typically $- c_i^2 U_x(x_{i+ \frac{1}{2}},t_n)$.
Thus a reasonable choice for the approximation of the continuous flux $- c_i^2 U_x(x_{i+ \frac{1}{2}},t_n)$ is the differential quotient : 
\begin{equation*} 
F_{i+{\frac{1}{2}}}^n =\left\{\begin{array}{ll}
- C_1^2 \frac{u_1^n}{h_{\frac{1}{2}}}, & i = 0, \\[0.25cm]
- c_i^2 \frac{U_{i+1}^n - U_i^n}{h_{i+\frac{1}{2}}}, & i = 1, \ldots, N_{\max}-1, \\[0.25cm]
C_3^2 \frac{w_{N_{\max}}^n}{h_{N_{\max}+\frac{1}{2}}}, & i = N_{\max}.
\end{array}\right.    
\end{equation*}
As this numerical flux depends only on the value of the unknown $U^n$ at time $t_n$, we decide to call it \textbf{explicit numerical flux} and the discretized problem is called 
\textbf{explicit discretized problem} although as it is mentionned in Remark \ref{remark3}, the discretized problem is not explicit because of the damping term  
$\delta v_{xxt}$ in equation \eqref{Eq(2.2)}.

We denote $\ell_{i+ \frac{1}{2}} = \frac{c_i^2}{h_{i+ \frac{1}{2}}}$ and we obtain :
\begin{align}\label{li}
	\ell_{i+ \frac{1}{2}}= \left\{\begin{array}{ll}
		\frac{C_1^2}{h_{i+ \frac{1}{2}}}=\frac{C_1^2}{h_\alpha} , & i=1,\ldots N_{\alpha}-1,\\[0.25cm]
		\frac{C_2^2}{h_{i+ \frac{1}{2}}}= \frac{C_2^2}{h} , & i = N_{\alpha}+1, \ldots, N_{\alpha}+N-1,\\[0.25cm]
		\frac{C_3^2}{h_{i+ \frac{1}{2}}}= \frac{C_3^2}{h_\beta} , & i = N_{\alpha}+N+1, \ldots, N_{\max}.
	\end{array} \right. \ .
\end{align}
Later in this paper, we will adapt the formula for the numerical fluxes at the inner boundary points \linebreak  $x_{N_\alpha+ \frac{1}{2}}=\alpha$ and $ x_{N_\beta+ \frac{1}{2}}=\beta$ and thus we will
properly define the value of $\ell_{N_\alpha+ \frac{1}{2}}$ and $\ell_{N_\alpha+ N+\frac{1}{2}}$.

\noindent \textbf{Discretization of Equation \eqref{Eq(2.1)}.} 

In this part, we will detail every computations in order to avoid the redaction of the computations for the next parts.

$\Box$ For $i = 1, \cdots, N_{\alpha} - 1$.

\begin{align*} 
  \int_{K_i} u_{tt}(x,t_n) dx - C_1^2  \int_{K_i} u_{xx}(x,t_n) dx &= 0,  \\
  \int_{x_{i-\frac{1}{2}}}^{x_{i+\frac{1}{2}}} u_{tt}(x,t_n) dx - C_1^2  \int_{x_{i-\frac{1}{2}}}^{x_{i+\frac{1}{2}}} u_{xx}(x,t_n) dx &= 0. 
\end{align*}
Let us again remark that
\begin{align*}
\dfrac{\partial^2 \displaystyle \int_{K_{i}} U(x,t)dx} {\partial t^2} &= \int_{K_{i}} \dfrac{\partial^2 U(x,t)}{\partial t^2} dx \mbox{ and }\\
\int_{x_{i-\frac{1}{2}}}^{x_{i+\frac{1}{2}}} u_{xx}(x,t_n) dx &= u_x(x_{i+\frac{1}{2}},t_n) - u_x(x_{i-\frac{1}{2}},t_n) \ .
\end{align*}
So we firstly obtain:
\begin{equation*}
\frac{1}{h_i}  \int_{x_{i-\frac{1}{2}}}^{x_{i+\frac{1}{2}}} u_{tt} (x,t_n) dx  = \frac{1}{h_i} \bigg(  \int_{x_{i-\frac{1}{2}}}^{x_{i+\frac{1}{2}}} u (x,t_n) dx  \bigg)_{tt} =  [u_i(t_n)]_{tt}.
\end{equation*}
A direct calculation gives
\begin{equation*}
\begin{array}{lll}
h_i [u_i(t_n)]_{tt} - C_1^2 [ u_x(x_{i+\frac{1}{2}},t_n) - u_x(x_{i-\frac{1}{2}},t_n)] = 0.
\end{array}
\end{equation*}
Using second-order central time discretization at time $t_n$ and the fluxes approximation defined before, we obtain:
\begin{equation*}\label{Eq-discrete-1}
  h_{i} \frac{u_i^{n+1}- 2 u_i^n + u_i^{n-1}}{\Delta t^2} - \big[ \ell_{i+ \frac{1}{2}} ( u^{n}_{i + 1} - u^{n}_{i }) - \ell_{i- \frac{1}{2}} (u^{n}_{i } - u^{n}_{i - 1 }) \big] = 0.
\end{equation*}
$\Box$ For $i = N_{\alpha}$ that is $x_{N_{\alpha}+\frac{1}{2}} = \alpha$.
\begin{align*} 
  \int_{x_{N_{\alpha}-\frac{1}{2}}}^{x_{N_{\alpha}+\frac{1}{2}}} u_{tt} dx - C_1^2  \int_{x_{N_{\alpha}-\frac{1}{2}}}^{x_{N_{\alpha}+\frac{1}{2}}} u_{xx} dx &= 0, \\
h_{N_{\alpha}} [u_{N_{\alpha}}(t_n)]_{tt} - C_1^2 [ u_x(x_{N_{\alpha}+\frac{1}{2}},t_n) - u_x(x_{N_{\alpha}-\frac{1}{2}},t_n)] &= 0.
\end{align*}
Let $H_{i+\frac{1}{2}}$ represent the approximation of the flux $ c_{i}^2 u_x(x_{i+\frac{1}{2}},t_n)$ for $i = N_{\alpha}$. Using the finite difference principle, and noticing that 
on the cell $K_i$, the unknown is $u$ and on the cell $K_{i+1}$, the unknown is $v$, we have
\begin{equation*}
\begin{array}{lll}
H_{i+\frac{1}{2}} = c_i^2 \frac{u_{i+\frac{1}{2}}^n - u_i^n}{h_i^+}    \quad & \textrm{over} \ K_{i}; \ i = N_{\alpha}, \\
H_{i+\frac{1}{2}} = c_{i+1}^2 \frac{{v^n_{i+1}} - v^n_{{i+\frac{1}{2}}}}{h_{i+1}^-} \quad & \textrm{over} \ K_{i+1}; \ i = N_{\alpha}
\end{array}
\end{equation*}
where $u_{i+\frac{1}{2}}^n$ is unknown and represents an approximation of $u(\alpha,t_n)$ on the left side	and $v_{i+\frac{1}{2}}^n$ is unknown and represents an approximation of $u(\alpha,t_n)$ 
on the right side. Since $x_{N_{\alpha}+\frac{1}{2}} = \alpha$, the transmission condition $u(\alpha,t) =  v(\alpha,t)$ leads to $v^n_{i+\frac{1}{2}} = u^n_{i+\frac{1}{2}}$.
		
Since $x_i$ is the center of $K_i$ for all $i$, then substituting $i = N_{\alpha}$ gives
\begin{equation*}
\begin{array}{lll}
H_{i+\frac{1}{2}} = c_i^2 \frac{u_{i+\frac{1}{2}}^n - u_i^n}{\frac{h_{\alpha}}{2}}  \quad & \textrm{over} \ K_{i}; \ i = N_{\alpha}, \\
H_{i+\frac{1}{2}}  = c_{i+1}^2 \frac{{v^n_{i+1}} - u^n_{{i+\frac{1}{2}}}}{\frac{h}{2}}  \quad & \textrm{over} \ K_{i+1}; \ i = N_{\alpha}
\end{array}
\end{equation*}
The above approximations should be equal to guarantee the conservation of the flux (see \cite{ERG2000}), and thus we get by a direct calculation 
\begin{equation}\label{conservativity of flux}
H_{i+\frac{1}{2}} = \frac{c_{i}^2 c_{i+1}^2}{ \frac{h}{2} c^2_{i}  + \frac{h_{\alpha}}{2} c_{i+1}^2  } (v_{i+1}^n - u_i^n) \quad \textrm{for} \ i = N_{\alpha}.
\end{equation} 
Given that for $i = N_{\alpha}$, $c_{i+1} = C_2$ and $c_i = C_1$, \textbf{we set the value of} $\boldsymbol{\ell_{N_{\alpha}+\frac{1}{2}}}$ as 
$\boldsymbol{\ell_{N_{\alpha}+\frac{1}{2}}= \frac{2 C_1^2 C_2^2}{C_1^2 h +  C_2^2 h_{\alpha}}}$ to obtain :
\begin{equation}\label{eq-interface-alpha}
C_1^2 u_x(x_{N_{\alpha}+\frac{1}{2}}, t_n) \simeq \ell_{N_{\alpha}+\frac{1}{2}} (v^n_{N_{\alpha}+1} - u^n_{N_{\alpha}}) \ .
\end{equation}
\begin{Remark}
Let us remark that we have : 
\begin{equation}\label{C1-C2}
\min(C_1^2,C_2^2) \leq \ell_{N_{\alpha}+\frac{1}{2}} \times h_{N_{\alpha}+\frac{1}{2}}\leq \max(C_1^2,C_2^2) \ .
\end{equation}
We will use these inequalities to obtain an estimation of the discrete $\mathbb{H}^1$ semi-norm (see lemma \ref{H1estimate})
\end{Remark}
Now, using the second-order central difference in time and forward difference discretization in space, we obtain
\begin{align*}\label{Eq-discrete-2}
  h_{N_{\alpha}}  \frac{u_{N_{\alpha}}^{n+1}- 2 u_{N_{\alpha}}^n + u_{N_{\alpha}}^{n-1}}{\Delta t^2}  -\big [ \ell_{N_{\alpha}+\frac{1}{2}} (v^{n}_{N_{\alpha} + 1} - u^{n}_{N_{\alpha}}) - \ell_{N_{\alpha}-\frac{1}{2}} (u^{n}_{N_{\alpha} } - u^{n}_{N_{\alpha} -1}) \big] &= 0.
\end{align*} 
\textbf{Discretization of Equation \eqref{Eq(2.2)}.} \\
$\Box$ For $i = N_{\alpha} + 1$:
\begin{align*} 
  \int_{x_{N_{\alpha}+\frac{1}{2}}}^{x_{N_{\alpha}+\frac{3}{2}}} v_{tt} dx - C_2^2  \int_{x_{N_{\alpha}+\frac{1}{2}}}^{x_{N_{\alpha}+\frac{3}{2}}} v_{xx} dx - \delta \int_{x_{N_{\alpha}+\frac{1}{2}}}^{x_{N_{\alpha}+\frac{3}{2}}} v_{xxt} dx &= 0, \\
h_{N_{\alpha}+1} [v_{N_{\alpha}+1}(t_n)]_{tt} - C_2^2 [ v_x(x_{N_{\alpha}+\frac{3}{2}},t_n) - v_x(x_{N_{\alpha}+\frac{1}{2}},t_n)]  - \delta [ v_{xt}(x_{N_{\alpha}+ \frac{3}{2}},t_n) - v_{xt}(x_{N_{\alpha} + \frac{1}{2}},t_n)]&= 0.
\end{align*}
First of all, since the position $x_{N_{\alpha}+\frac{1}{2}}$ represents the point $\alpha$, then we use the transmission condition \eqref{Eq(2.5)} at point $\alpha$ and thus
\begin{equation*}
C_2^2 v_x(x_{N_{\alpha}+\frac{1}{2}},t_n) + \delta  v_{xt}(x_{N_{\alpha} + \frac{1}{2}},t_n) = C_1^2 u_x(x_{N_{\alpha}+\frac{1}{2}},t_n).
\end{equation*}
Consequently, we obtain
\begin{equation*}
h_{N_{\alpha}+1} [v_{N_{\alpha}+1}(t_n)]_{tt} + C_1^2 u_x(x_{N_{\alpha}+\frac{1}{2}},t_n)- C_2^2  v_x(x_{N_{\alpha}+\frac{3}{2}},t_n) - \delta v_{xt}(x_{N_{\alpha}+ \frac{3}{2}},t_n) = 0.
\end{equation*}
For the first term of the above equation we apply the second-order central difference in time at $t_n$ and for the third and fourth terms we use spatial forward difference, where the central difference in time is applied only to the fourth term. However, the second term is treated for $i = N_{\alpha}$ similarly as \eqref{conservativity of flux}. Therefore, we obtain
\begin{align}\label{Eq-discrete-3}
h_{N_{\alpha}+1} \frac{v_{N_{\alpha}+1}^{n+1}- 2 v_{N_{\alpha}+1}^n + v_{N_{\alpha}+1}^{n-1}}{\Delta t^2} -\big [ \ell_{N_{\alpha}+\frac{3}{2}} (v^{n}_{N_{\alpha} + 2} - v^{n}_{N_{\alpha}+1}) - \ell_{N_{\alpha}+\frac{1}{2}} (v^{n}_{N_{\alpha}+1 } - u^{n}_{N_{\alpha} }) \big] \\
  = \frac{\delta}{2 \Delta t} \bigg[  \frac{{v^{n+1}_{N_{\alpha}+2} - v^{n+1}_{N_{\alpha}+1}}}{h_{N_{\alpha}+1}} - \frac{{v^{n-1}_{N_{\alpha}+2} - v^{n-1}_{N_{\alpha}+1}}}{h_{N_{\alpha}+1}} \bigg].  \nonumber
\end{align}
$\Box$ For $i = N_{\alpha} + 2, \cdots, N_{\alpha} + N -1 $:

\begin{align*} 
  \int_{x_{i}-\frac{1}{2}}^{x_{i}+\frac{1}{2}} v_{tt} dx - C_2^2  \int_{x_{i}-\frac{1}{2}}^{x_{i}+\frac{1}{2}} v_{xx} dx - \delta \int_{x_{i}-\frac{1}{2}}^{x_{i}+\frac{1}{2}} v_{xxt} dx &= 0,\\  \nonumber
 h_i [v_i(t_n)]_{tt} - C_2^2 [ v_x(x_{i+\frac{1}{2}},t_n) - v_x(x_{i-\frac{1}{2}},t_n)]  - \delta [ v_{xt}(x_{i+\frac{1}{2}},t_n) - v_{xt}(x_{i-\frac{1}{2}},t_n)] &= 0.
\end{align*}

Similarly, applying a second-order central difference to the second time derivative and a forward difference in space to the other terms combined with centered time discretization for the last two terms, we obtain 
\begin{equation*}\label{Eq-discrete-4}
\begin{array}{lll}
  h_i \frac{v_i^{n+1}- 2 v_i^n + v_i^{n-1}}{\Delta t^2} - \big[ \ell_{i+ \frac{1}{2}} ( v^{n}_{i + 1} - v^{n}_{i }) - \ell_{i- \frac{1}{2}} (v^{n}_{i } - v^{n}_{i - 1 }) \big] \\
 = \frac{\delta}{2 \Delta t} \bigg [   \frac{v^{n+1}_{i +1 } - v^{n+1}_{i }}{h_i} -   \frac{v^{n-1}_{i +1 } - v^{n-1}_{i }}{h_i} -  \frac{v^{n+1}_{i } - v^{n+1}_{i - 1 }}{h_i} +  \frac{v^{n-1}_{i } - v^{n-1}_{i - 1 }}{h_i} \bigg].
\end{array}
\end{equation*}
$\Box$ For $i = N_{\alpha} + N$:
\begin{equation*}
\begin{array}{lll} 
  \int_{x_{N_{\alpha}+N-\frac{1}{2}}}^{x_{N_{\alpha}+N+\frac{1}{2}}} v_{tt} dx - C_2^2  \int_{x_{N_{\alpha}+N-\frac{1}{2}}}^{x_{N_{\alpha}+N+\frac{1}{2}}} v_{xx} dx - \delta \int_{x_{N_{\alpha}+N-\frac{1}{2}}}^{x_{N_{\alpha}+N+\frac{1}{2}}} v_{xxt} dx = 0, \\
 h_{N_{\alpha}+N} [v_{N_{\alpha}+N}(t_n)]_{tt} - C_2^2 [ v_x(x_{N_{\alpha}+N+\frac{1}{2}},t_n) - v_x(x_{N_{\alpha}+N-\frac{1}{2}},t_n)] - \delta [ v_{xt}(x_{N_{\alpha}+ N+ \frac{1}{2}},t_n) - v_{xt}(x_{N_{\alpha} + N - \frac{1}{2}},t_n)]= 0.
\end{array}
\end{equation*}
Noting that the position $x_{N_{\alpha}+N+\frac{1}{2}}$ represents the point $\beta$, then we use the transmission condition \eqref{Eq(2.5)} at point $\beta$ and thus
\begin{equation*}
- C_2^2 v_x(x_{N_{\alpha}+N+\frac{1}{2}},t_n) - \delta  v_{xt}(x_{N_{\alpha} + N+ \frac{1}{2}},t_n) = - C_3^2 w_x(x_{N_{\alpha}+N+\frac{1}{2}},t_n).
\end{equation*}
Consequently, we obtain
\begin{equation*}
 h_{N_{\alpha}+N} [v_{N_{\alpha}+N}(t_n)]_{tt} - C_3^2 w_x(x_{N_{\alpha}+N+\frac{1}{2}},t_n) + C_2^2  v_x(x_{N_{\alpha}+N-\frac{1}{2}},t_n)  + \delta  v_{xt}(x_{N_{\alpha} + N - \frac{1}{2}},t_n)   = 0.
\end{equation*}

Similar to the way used for $i = N_{\alpha}+1$ (see \eqref{Eq-discrete-3}), we apply the second-order central difference in time for the first term of the above equation. With respect to the third and fourth terms, we use spatial forward difference, where the central difference in time is applied only to the fourth term. 
Concerning the second term,  we denote the approximation of $c_i^2 w_x(x_{i+ \frac{1}{2}},t_n)$ by $J_{i+ \frac{1}{2}}$ for $i = N_{\alpha}+N$. 
We argue in the same way as in \cite{ERG2000} to obtain $J_{i+ \frac{1}{2}}$by the same way to obtain the value of $H_{i+ \frac{1}{2}}$ (see \eqref{conservativity of flux}). 
Given that for $i = N_{\alpha}+N$, $c_{i+1} = C_3$ and $c_i = C_2$,   \textbf{we set the value of} $\boldsymbol{\ell_{N_{\alpha}+N+\frac{1}{2}}}$ as 
$\boldsymbol{\ell_{N_{\alpha}+N+\frac{1}{2}}= \frac{2 C_2^2 C_3^2}{C_3^2 h +  C_2^2 h_{\beta}}}$ to obtain :
 
\begin{equation}\label{eq-interface-beta}
C_3^2 w_x(x_{N_{\alpha}+N+\frac{1}{2}}, t_n) \simeq \ell_{N_{\alpha}+N+\frac{1}{2}} (w^n_{N_{\alpha}+N+1} - v^n_{N_{\alpha}+N}) \  .
\end{equation}
\begin{Remark}
Let us remark that we have : 
\begin{equation}\label{C2-C3}
\min(C_2^2,C_3^2)\leq \ell_{N_{\alpha}+N+\frac{1}{2}} \times h_{N_{\alpha}+N+\frac{1}{2}} \leq \max(C_2^2,C_3^2)
\end{equation}
We will use these inequalities to obtain an estimation of the discrete $\mathbb{H}^1$ semi-norm (see lemma \ref{H1estimate})
\end{Remark}
Therefore, we obtain
\begin{align}\label{Eq-discrete-5}
  h_{N_{\alpha}+N} \frac{v_{N_{\alpha}+N}^{n+1}- 2 v_{N_{\alpha}+N}^n + v_{N_{\alpha}+N}^{n-1}}{\Delta t^2} -\big [ \ell_{N_{\alpha}+N+\frac{1}{2}} (w^{n}_{N_{\alpha}+N + 1} - v^{n}_{N_{\alpha}+N}) - \ell_{N_{\alpha}+N-\frac{1}{2}} (v^{n}_{N_{\alpha}+N } - v^{n}_{N_{\alpha} +N-1}) \big] \\
=- \frac{\delta}{2 \Delta t} \bigg[  \frac{{v^{n+1}_{N_{\alpha}+N} - v^{n+1}_{N_{\alpha}+N-1}}
}{h_{N_{\alpha}+N}} - \frac{{v^{n-1}_{N_{\alpha}+N} - v^{n-1}_{N_{\alpha}+N-1}}}{h_{N_{\alpha}+N}} \bigg].\nonumber
\end{align}
\textbf{Discretization of Equation \eqref{Eq(2.3)}.} \\
$\Box$ For $i = N_{\alpha} + N + 1$:
\begin{align*} 
  \int_{x_{N_{\alpha}+N+\frac{3}{2}}}^{x_{N_{\alpha}+N+\frac{1}{2}}} w_{tt} dx - C_3^2  \int_{x_{N_{\alpha}+N+\frac{1}{2}}}^{x_{N_{\alpha}+N+\frac{3}{2}}} w_{xx} dx &= 0, \\
 h_{N_{\alpha} + N + 1} [w_{N_{\alpha} + N + 1}(t_n)]_{tt} - C_3^2 [ w_x(x_{N_{\alpha}+N+\frac{3}{2}},t_n) - w_x(x_{N_{\alpha}+N+\frac{1}{2}},t_n)] &= 0.
\end{align*}
Note that the last term of the above equation is treated in the previous part (for $i = N_{\alpha}+N$). Again, using the second-order central difference in time and forward difference discretization in space, we obtain
\begin{align*}
h_{N_{\alpha}+N+1}  \frac{w_{N_{\alpha}+N+1}^{n+1}- 2 w_{N_{\alpha}+N+1}^n + w_{N_{\alpha}+N+1}^{n-1}}{\Delta t^2}  -  \big [ \ell_{N_{\alpha}+N+{\frac{3}{2}}} (w^{n}_{N_{\alpha}+N+2 } - w^{n}_{N_{\alpha} +N+1}) - \ell_{N_{\alpha}+N+{\frac{1}{2}}}  (w^{n}_{N_{\alpha}+N + 1} - v^{n}_{N_{\alpha}+N})  \big] &= 0.
\end{align*} 
$\Box$ For $i =N_{\alpha} + N +2  , \cdots, N_{\max}$: \\[0.1in]
Similar to the treatment of the first equation (for $i = 1, \cdots, N_{\alpha}-1$, see \eqref{Eq-discrete-1}), we do the same thing to obtain
\begin{equation*}\label{Eq-discrete-7}
  h_{i} \frac{w_i^{n+1}- 2 w_i^n + w_i^{n-1}}{\Delta t^2} - \big[ \ell_{i+ \frac{1}{2}} ( w^{n}_{i + 1} - w^{n}_{i }) - \ell_{i- \frac{1}{2}} (w^{n}_{i } - w^{n}_{i - 1 }) \big] = 0.
\end{equation*}
Finally the explicit discrete problem can be written in the compact form, for $i = 1,\ldots,N_{\max}$
\begin{equation} \label{Discrete Explicit}
h_i \frac{U_i^{n+1}-2U_i^n + U_i^{n-1}}{\Delta t^2} - [\ell_{i+{\frac{1}{2}}} (U_{i+1}^n - U_i^n) - \ell_{i-{\frac{1}{2}}} (U_{i}^n - U_{i-1}^n)]
= rhs_i
\end{equation}
where $rhs_i$ is defined by: for
\begin{align*}
&i = 1, \ldots, N_{\alpha}\,,\,rhs_i = 0 \\
&i = N_{\alpha}+1\,,\,rhs_i =\frac{\delta}{2 \Delta t} \bigg[  \frac{{U^{n+1}_{N_{\alpha}+2} - U^{n+1}_{N_{\alpha}+1}}}{h_{N_{\alpha}+1}} - \frac{{U^{n-1}_{N_{\alpha}+2} - U^{n-1}_{N_{\alpha}+1}}}{h_{N_{\alpha}+1}} \bigg],\\
&i = N_{\alpha}+2, \ldots, N_{\alpha}+N-1\,,\,rhs_i =\frac{\delta}{2 \Delta t} \bigg [   \frac{U^{n+1}_{i +1 } - U^{n+1}_{i }}{h_i} -   \frac{U^{n-1}_{i +1 } - U^{n-1}_{i }}{h_i} -  \frac{U^{n+1}_{i } - U^{n+1}_{i - 1 }}{h_i} +  \frac{U^{n-1}_{i } - U^{n-1}_{i - 1 }}{h_i} \bigg],  \\
&i = N_{\alpha}+N\,,\,rhs_i =-\frac{\delta}{2 \Delta t} \bigg[  \frac{{U^{n+1}_{N_{\alpha}+N} - U^{n+1}_{N_{\alpha}+N-1}}
}{h_{N_{\alpha}+N}} - \frac{{U^{n-1}_{N_{\alpha}+N} - U^{n-1}_{N_{\alpha}+N-1}}}{h_{N_{\alpha}+N}} \bigg],  \\
&i = N_{\alpha}+N+1, \ldots, N_{\max}\,,\,rhs_i = 0 .
\end{align*}
Consequently, the discrete explicit problem which represents an approximation of System \eqref{Eq(2.1)}-\eqref{Eq(2.5)} can be written as:
\begin{equation} \label{Discrete Explicit-Matrix}
\mbox{for } n = 1,\ldots,\mathcal{N} \,,\, \bigg(\mathcal{M} +  \frac{\delta \Delta t}{h}\mathcal{A} \bigg)U^{n+1} =\bigg(2 \mathcal{M} + \Delta t^2 \mathcal{B} \bigg)U^n - \bigg(\mathcal{M} -\frac{\delta \Delta t}{h}\mathcal{A} \bigg) U^{n-1},
\end{equation}
where  
$$ U =\begin{pmatrix}
u_{i = 1 , \cdots, N_{\alpha} }\\
v_{i= N_{\alpha}+1, \cdots N_{\alpha}+N}\\
w_{i= N_{\alpha}+N +1, \cdots, N_{\max}}
\end{pmatrix} , \quad  
 \mathcal{A} = \begin{pmatrix}
0 & 0 & 0\\
0 & A_{22} & 0\\
0 & 0 & 0
\end{pmatrix}   , \quad \mathcal{B} = \begin{pmatrix}
B_{11} & B_{12} & 0\\
B_{21} & B_{22} & B_{23}\\
0 & B_{32} & B_{33}
\end{pmatrix}$$
\noindent and $\mathcal{M}$ is a diagonal matrix of size $N_{\max}$ such that $M_{11}=\textrm{diag}(h_i)_{i = 1, \ldots, N_{\alpha}}$, $M_{22} = \textrm{diag}(h_i)_{i= N_{\alpha}+1, \ldots, N_{\alpha}+N}$ and $M_{33} = \textrm{diag}(h_i)_{i= N_{\alpha}+N+1, \ldots, N_{\max}}$. 
{\small{\begin{equation*}
\begin{array}{ll}
B_{11} = \begin{pmatrix}
-\ell_{3/2}-\ell_{1/2} & \ell_{3/2} & 0 & \cdots    & \cdots & 0\\
\ell_{3/2} &  -\ell_{5/2}-\ell_{3/2} & \ell_{5/2} & 0 & \cdots & 0\\
0 & \ddots & \cdots & \ddots \\
0 & \cdots& \cdots & \cdots& \ell_{N_{\alpha}-1/2} & -\ell_{N_{\alpha}+1/2} - \ell_{N_{\alpha}-1/2}
\end{pmatrix}, \\
B_{22} = \begin{pmatrix}
-\ell_{N_{\alpha}+3/2}- \ell_{N_{\alpha}+1/2} &  \ell_{N_{\alpha}+3/2} & 0    & \cdots & 0\\ 
 \ell_{N_{\alpha}+3/2} & -\ell_{N_{\alpha}+5/2} - \ell_{N_{\alpha}+3/2}&  \ell_{N_{\alpha}+5/2} & 0 & \cdots \\
0 & \ddots & \cdots & \ddots \\
0 & \cdots& \cdots &  \ell_{N_{\alpha}+N-1/2} &  -\ell_{N_{\alpha}+N+1/2} - \ell_{N_{\alpha}+N-1/2}
\end{pmatrix}, \\
B_{33} = \begin{pmatrix}
 -\ell_{N_{\alpha}+N+3/2} - \ell_{N_{\alpha}+N+1/2} &  \ell_{N_{\alpha}+N+3/2} & 0 & \cdots    & \cdots & 0\\
\ell_{N_{\alpha}+N+3/2} & - \ell_{N_{\alpha}+N+5/2} - \ell_{N_{\alpha}+N+3/2}&  \ell_{N_{\alpha}+N+5/2} & 0 & \cdots & 0 \\
0 & \ddots & \cdots & \ddots \\
0 & \cdots& \cdots & \cdots&  \ell_{N_{\max}- 1/2} & - \ell_{N_{\max}+ 1/2} - \ell_{N_{\max}- 1/2}
\end{pmatrix},
\end{array}
\end{equation*} 
\begin{equation*}
\begin{array}{ll}
B_{21} = \begin{pmatrix}
0 & \cdots & 0 &  \ell_{N_{\alpha}+1/2}\\
\ddots &\cdots& \ddots & 0 \\
\ddots &\cdots& \ddots & \vdots
\end{pmatrix},  \ B_{12} = \begin{pmatrix}
0 & \cdots & 0 & \cdots \\
\ddots &\cdots& \ddots & \vdots \\
 \ell_{N_{\alpha}+1/2} &\cdots& \ddots & \vdots
\end{pmatrix}, 
B_{32} = \begin{pmatrix}
0 & \cdots & 0 &  \ell_{N_{\alpha}+N+1/2}\\
\ddots &\cdots& \ddots & 0 \\
\ddots &\cdots& \ddots & \vdots
\end{pmatrix}, 
\\
 \ B_{23} = \begin{pmatrix}
0 & \cdots & 0 & \cdots \\
\ddots &\cdots& \ddots & \vdots \\
 \ell_{N_{\alpha}+N+1/2} &\cdots& \ddots & \vdots
\end{pmatrix},  A_{22} = \begin{pmatrix}
1/2 & - 1/2 & 0    & \cdots & 0\\ 
- 1/2 & 1& -1/2 & 0 & \cdots \\
0 & \ddots & \ddots & \ddots \\
0 & \cdots& \ddots &  - 1/2 & 1/2
\end{pmatrix}.
\end{array}
\end{equation*}}}
Writing the numerical scheme for $n=0$ and using the value of $U^{-1}$ defined by \eqref{artificialSOLUTION} we get :
\begin{equation} \label{Discrete-Explicit-Matrix1}
2 \mathcal{M} U^{1} =\bigg(2 \mathcal{M} + \Delta t^2 \mathcal{B} \bigg)U^0 + 2 \Delta t \bigg(\mathcal{M} -\frac{\delta \Delta t}{h}\mathcal{A} \bigg) \Psi_{\mathcal{T}}  
\end{equation}
\begin{Remark} \label{remark3}~
\begin{itemize}
\item The matrix $\bigg(\mathcal{M} +  \frac{\delta \Delta t}{h}\mathcal{A} \bigg)$ is
irreducible diagonal dominant. 
So this matrix is also invertible (see \cite{Golub-Meurant}). Therefore, equations \eqref{Discrete Explicit-Matrix} and \eqref{Discrete-Explicit-Matrix1} are well-posed.
\item Let us also mention that due to the Kelvin-Voigt dissipation term, the numerical scheme \eqref{Discrete Explicit-Matrix} is not fully explicit. 
\end{itemize}
\end{Remark}
\bigskip

\begin{center}\textbf{\textit{Practical implementation}}\end{center}
\hspace*{-0.1cm}\rule[-1mm]{15cm}{0.05cm}\\
\textbf{Initialisation}\\
$U^0 = \Phi_{\mathcal{T}}$\\
\hspace*{-0.1cm}\rule[-1mm]{15cm}{0.05cm}\\
\textbf{Computation of $\boldsymbol{U^1}$}\\
$2 \mathcal{M} U^{1} =\bigg(2 \mathcal{M} + \Delta t^2 \mathcal{B} \bigg)U^0 
+ 2 \Delta t \bigg(\mathcal{M} -\frac{\delta \Delta t}{h}\mathcal{A} \bigg) \Psi_{\mathcal{T}}$\\
\hspace*{-0.1cm}\rule[-1mm]{15cm}{0.05cm}\\
\textbf{Compute the inverse of the matrix} $\boldsymbol{\bigg(\mathcal{M} +  \frac{\delta \Delta t}{h}\mathcal{A} \bigg)}$.\\
$C =\bigg(\mathcal{M} +  \frac{\delta \Delta t}{h}\mathcal{A} \bigg)^{-1}$.\\
\hspace*{-0.1cm}\rule[-1mm]{15cm}{0.05cm}\\
\textbf{Computation of $\boldsymbol{U^{n+1}}$ for $\boldsymbol{n=1,\ldots,\mathcal{N}}$}\\
for $n=1 \ldots, \mathcal{N}$ \\
\hspace*{1cm} $RHS = \bigg(2 \mathcal{M} + \Delta t^2 \mathcal{B} \bigg)U^n - \bigg(\mathcal{M} -\frac{\delta \Delta t}{h}\mathcal{A} \bigg) U^{n-1}$ \\
\hspace*{1cm} $U^{n+1} = C \times RHS$\\
endfor\\
\hspace*{-0.1cm}\rule[-1mm]{15cm}{0.05cm}\\

By discrete Fourier analysis, the numerical scheme \eqref{Discrete Explicit-Matrix} is stable if and only if the following Courant-Friedrichs-Lewy; {\it{i.e.}} CFL condition holds:
\begin{equation*}
\begin{array}{lll}
c_i^2 \Delta t^2 \leq (h_i)^2
\end{array}
\end{equation*} 
which is equivalent to
\begin{equation} \tag{CFL condition}
\Delta t \leq \min (c_i^{-1}) \Delta x; \quad \textrm{with} \quad \Delta x = \min \{h_i; \ i=1,\ldots, N_{\max} \}. 
\end{equation}
The number $\min (c_i^{-1})$ stands for the CFL number.

\subsection{Dissipation of the discrete energy for the explicit discretized problem} \label{SEE}
In this subsection, we plan to design a discrete energy that is conserved when $\delta = 0$ and dissipates when $\delta >0$. 
For this aim, \textcolor{black}{we copy the definition of the continuous energy \eqref{energy} and we define:}
\begin{itemize}
\item[$\bullet$] the discrete kinetic energy for $U$ as: 
\[ E_k(U^{n}) = \frac{1}{2} \sum_{i=1}^{N_{\max}} h_i \bigg(\frac{U_{i}^{n+1}-U_{i}^n}{\Delta t}\bigg)^2
\]
\item[$\bullet$] the discrete potential energy for $U$ as: 
\[E_p(U^{n}) = \frac{1}{2}\sum_{i=0}^{N_{\max}} \ell_{i+{\frac{1}{2}}} (U_{i+1}^{n+1} - U_{i}^{n+1})(U_{i+1}^{n}- U_i^{n}) \ .
\]
\end{itemize}
The total discrete energy is then defined as
\begin{equation}\label{discrete-energy}
\mathcal{E}^{n} = E_k(U^{n}) + E_p(U^{n}).
\end{equation}
Now, we want to prove that the above stated goals for the discrete energy \eqref{discrete-energy} are fulfilled. For this purpose, we aim to prove the following theorem:
\begin{Theorem} \label{theorem explicit dissipation}
The total discrete energy defined by \eqref{discrete-energy} of the explicit numerical scheme \eqref{Discrete Explicit} satisfies the following
 \textcolor{black}{dissipativity estimation}
\begin{equation*}
\mathcal{E}^{n} - \mathcal{E}^ {n-1} = -  \delta \Delta t \textcolor{black}{h} \sum_{i=N_{\alpha}+1}^{N_{\alpha}+N -1}\bigg( \frac{U^{n+1}_{i +1 } - U^{n+1}_{i } - U^{n-1}_{i +1 } + U^{n-1}_{i }}{2 \Delta t \ h} \bigg)^2.
\end{equation*}
\end{Theorem}

\begin{proof}
 \textcolor{black}{Let us first mention that this estimation is the discrete version of the dissipativity estimation of the continuous solution \eqref{dissipative}.}\\
The proof of Theorem \ref{theorem explicit dissipation} is similar to the continuous case where we multiply the discrete problem by the approximation of $U_t$. To obtain the energy estimates, we multiply the left hand-side and right hand-side of \eqref{Discrete Explicit} by $(U_i^{n+1}- U_i^{n-1})$ and we sum over $i = 1, \ldots, N_{\max}$ to obtain :
\begin{align}
	\sum_{i=1}^{N_{\max}} h_i \bigg(\frac{U_i^{n+1}-2 U_i^n + U_i^{n-1}}{\Delta t^2} \bigg) (U_i^{n+1}-U_i^{n-1}) &- \sum_{i=1}^{N_{\max}}  [\ell_{i+{\frac{1}{2}}} (U_{i+1}^n - U_i^n) - \ell_{i-{\frac{1}{2}}} (U_{i}^n - U_{i-1}^n)](U_i^{n+1}-U_i^{n-1}) \nonumber\\
	&= \sum_{i=1}^{N_{\max}} rhs_{i} (U_i^{n+1}-U_i^{n-1})	\label{Discrete-EQ-1}
\end{align}
\textbf{Estimation of the left hand-side of \eqref{Discrete-EQ-1}}:
We will firstly estimate each term of the right hand-side. 
\textbf{Estimation of the first term of \eqref{Discrete-EQ-1}}: 
\begin{equation} \label{result-energy-kinetic}
\begin{array}{lll}
&\sum_{i=1}^{N_{\max}} h_i \bigg(\frac{U_i^{n+1}-2 U_i^n + U_i^{n-1}}{\Delta t^2} \bigg) (U_i^{n+1}-U_i^{n-1}) \\
&=\sum_{i=1}^{N_{\max}} h_i \bigg[\frac{(U_i^{n+1}- U_i^n) -( U_i^n - U_i^{n-1})}{\Delta t^2} \bigg] [(U_i^{n+1}-U_i^n) + (U_i^n -U_i^{n-1})] \\
&=  \sum_{i=1}^{N_{\max}} h_i \bigg(\frac{U_{i}^{n+1}-U_{i}^n}{\Delta t}\bigg)^2 -  \sum_{i=1}^{N_{\max}} h_i \bigg(\frac{U_{i}^{n}-U_{i}^{n-1}}{\Delta t}\bigg)^{2} \\
&= 2(E_k(U^n)-E_k(U^{n-1})).
\end{array}
\end{equation}
\textbf{Estimation of the second term of \eqref{Discrete-EQ-1}}:
\begin{equation*}
\begin{array}{lll}
- \sum_{i=1}^{N_{\max}} [\ell_{i+{\frac{1}{2}}} (U_{i+1}^n - U_i^n) - \ell_{i-{\frac{1}{2}}} (U_{i}^n - U_{i-1}^n)](U_i^{n+1}-U_i^{n-1}) \\
= - \sum_{i=1}^{N_{\max}} \ell_{i+{\frac{1}{2}}} (U_{i+1}^n - U_i^n)(U_i^{n+1}-U_i^{n-1}) + \sum_{i=1}^{N_{\max}} \ell_{i-{\frac{1}{2}}} (U_{i}^n - U_{i-1}^n)(U_i^{n+1}-U_i^{n-1}).
\end{array}
\end{equation*}
By translation of index $i$ of the second term of the right hand-side of the above equation, we obtain:
\begin{equation*}
\begin{array}{lll}
- \sum_{i=1}^{N_{\max}} [\ell_{i+{\frac{1}{2}}} (U_{i+1}^n - U_i^n) - \ell_{i-{\frac{1}{2}}} (U_{i}^n - U_{i-1}^n)](U_i^{n+1}-U_i^{n-1}) \\
= - \sum_{i=1}^{N_{\max}} \ell_{i+{\frac{1}{2}}} (U_{i+1}^n - U_i^n)(U_i^{n+1}-U_i^{n-1}) + \sum_{i=0}^{N_{\max}-1} \ell_{i+{\frac{1}{2}}} (U_{i+1}^n - U_{i}^n)(U_{i+1}^{n+1}-U_{i+1}^{n-1}).
\end{array}
\end{equation*}
Taking into consideration that $U_0 = U_{N_{\max}+1}=0$, we obtain:
\begin{equation*}
\begin{array}{lll}
- \sum_{i=1}^{N_{\max}} [\ell_{i+{\frac{1}{2}}} (U_{i+1}^n - U_i^n) - \ell_{i-{\frac{1}{2}}} (U_{i}^n - U_{i-1}^n)](U_i^{n+1}-U_i^{n-1}) \\
=   \sum_{i=0}^{N_{\max}} \ell_{i+{\frac{1}{2}}} (U_{i+1}^n - U_{i}^n)[(U_{i+1}^{n+1}- U_i^{n+1})-(U_{i+1}^{n-1} -U_i^{n-1})]\\
=  \sum_{i=0}^{N_{\max}} \ell_{i+{\frac{1}{2}}} (U_{i+1}^n - U_{i}^n)(U_{i+1}^{n+1}- U_i^{n+1}) - \sum_{i=0}^{N_{\max}} \ell_{i+{\frac{1}{2}}} (U_{i+1}^n - U_{i}^n)(U_{i+1}^{n-1} -U_i^{n-1})\\
= 2(E_p(U^n)-E_p(U^{n-1})).
\end{array}
\end{equation*}
The left hand-side of \eqref{Discrete-EQ-1} stands for the discrete time derivative of the total discrete energy defined by \eqref{discrete-energy}. It is left to show that this energy provides the needed properties; {\it{i.e.}} the energy is preserved with $\delta=0$ and 
dissipative whenever $\delta >0$\\[0.1in]
\textbf{Estimation of the right hand-side of \eqref{Discrete-EQ-1}}: First of all, we have 
\begin{equation*} 
\begin{array}{lll}
\sum_{i=1}^{N_{\max}} rhs_{i} (U_i^{n+1}-U_i^{n-1})\\
= \frac{\delta}{2 \Delta t} \sum_{i=N_{\alpha}+2}^{N_{\alpha}+N -1} \bigg (  \frac{U^{n+1}_{i +1 } - U^{n+1}_{i }}{h} -   \frac{U^{n-1}_{i +1 } - U^{n-1}_{i }}{h} -  \frac{U^{n+1}_{i } - U^{n+1}_{i - 1 }}{h} +  \frac{U^{n-1}_{i } - U^{n-1}_{i - 1 }}{h} \bigg) (U_{i}^{n+1} - U_{i}^{n-1}) \\
 + \frac{\delta}{2 \Delta t} \bigg(  \frac{U^{n+1} _ {N_{\alpha} + 2 } - U^{n+1} _ {N_{\alpha}  + 1}}{h} - \frac{U^{n-1} _ {N_{\alpha}  + 2} - U^{n-1} _ {N_{\alpha}  + 1}}{h}  \bigg) (U_{N_{\alpha} + 1}^{n+1} - U_{N_{\alpha}+ 1}^{n-1})\\
- \frac{\delta}{2 \Delta t} \bigg(  \frac{U^{n+1} _ {N_{\alpha} +N } - U^{n+1} _ {N_{\alpha} +N -1 }}{h} - \frac{U^{n-1} _ {N_{\alpha} +N } - U^{n-1} _ {N_{\alpha} +N -1 }}{h} \bigg) (  U^{n+1}_{N_{\alpha} +N } - U^{n-1}_{N_{\alpha} +N }) \\
= \frac{\delta}{2 \Delta t} \sum_{i=N_{\alpha}+2}^{N_{\alpha}+N -1} \bigg (  \frac{U^{n+1}_{i +1 } - U^{n+1}_{i }}{h} -  \frac{U^{n+1}_{i } - U^{n+1}_{i - 1 }}{h} \bigg)(U_{i}^{n+1} - U_{i}^{n-1})\\
 - \frac{\delta}{2 \Delta t} \sum_{i=N_{\alpha}+2}^{N_{\alpha}+N -1}  \bigg( \frac{U^{n-1}_{i +1 } - U^{n-1}_{i }}{h} - \frac{U^{n-1}_{i } - U^{n-1}_{i - 1 }}{h} \bigg) (U_{i}^{n+1} - U_{i}^{n-1})  \\
+ \frac{\delta}{2 \Delta t} \bigg(  \frac{U^{n+1} _ {N_{\alpha} + 2 } - U^{n+1} _ {N_{\alpha}  + 1}}{h} - \frac{U^{n-1} _ {N_{\alpha}  + 2} - U^{n-1} _ {N_{\alpha}  + 1}}{h}  \bigg) (U_{N_{\alpha} + 1}^{n+1} - U_{N_{\alpha}+ 1}^{n-1})
 \\
- \frac{\delta}{2 \Delta t} \bigg(  \frac{U^{n+1} _ {N_{\alpha} +N } - U^{n+1} _ {N_{\alpha} +N -1 }}{h} - \frac{U^{n-1} _ {N_{\alpha} +N } - U^{n-1} _ {N_{\alpha} +N -1 }}{h} \bigg) (  U^{n+1}_{N_{\alpha} +N } - U^{n-1}_{N_{\alpha} +N }).
\end{array}
\end{equation*}
By translation of index of the second term of the right hand-side of the above equation, we obtain
{\small{\begin{equation*} 
\begin{array}{lll}
\frac{\delta}{2 \Delta t} \sum_{i=N_{\alpha}+2}^{N_{\alpha}+N -1} \bigg (  \frac{U^{n+1}_{i +1 } - U^{n+1}_{i }}{h} -   \frac{U^{n-1}_{i +1 } - U^{n-1}_{i }}{h} -  \frac{U^{n+1}_{i } - U^{n+1}_{i - 1 }}{h} +  \frac{U^{n-1}_{i } - U^{n-1}_{i - 1 }}{h} \bigg) (U_{i}^{n+1} - U_{i}^{n-1} ) 
\\
 + \frac{\delta}{2 \Delta t} \bigg(  \frac{U^{n+1} _ {N_{\alpha} + 2 } - U^{n+1} _ {N_{\alpha}  + 1}}{h} - \frac{U^{n-1} _ {N_{\alpha}  + 2} - U^{n-1} _ {N_{\alpha}  + 1}}{h}  \bigg) (U_{N_{\alpha} + 1}^{n+1} - U_{N_{\alpha}+ 1}^{n-1})
 \\
- \frac{\delta}{2 \Delta t} \bigg(  \frac{U^{n+1} _ {N_{\alpha} +N } - U^{n+1} _ {N_{\alpha} +N -1 }}{h} - \frac{U^{n-1} _ {N_{\alpha} +N } - U^{n-1} _ {N_{\alpha} +N -1 }}{h} \bigg) (  U^{n+1}_{N_{\alpha} +N } - U^{n-1}_{N_{\alpha} +N })\\
= \frac{\delta}{2 \Delta t} \sum_{i=N_{\alpha}+2}^{N_{\alpha}+N -1} \bigg (  \frac{U^{n+1}_{i +1 } - U^{n+1}_{i }}{h} -   \frac{U^{n-1}_{i +1 } - U^{n-1}_{i }}{h} \bigg) (U_{i}^{n+1} - U_{i}^{n-1}) \\
 - \frac{\delta}{2 \Delta t} \sum_{i=N_{\alpha}+1}^{N_{\alpha}+N -2} \bigg ( \frac{U^{n+1}_{i+1 } - U^{n+1}_{i  }}{h} - \frac{U^{n-1}_{i+1 } - U^{n-1}_{i  }}{h} \bigg) (U_{i+1}^{n+1} - U_{i+1}^{n-1}) \\
  + \frac{\delta}{2 \Delta t} \bigg(  \frac{U^{n+1} _ {N_{\alpha} + 2 } - U^{n+1} _ {N_{\alpha}  + 1}}{h} - \frac{U^{n-1} _ {N_{\alpha}  + 2} - U^{n-1} _ {N_{\alpha}  + 1}}{h}  \bigg) (U_{N_{\alpha} + 1}^{n+1} - U_{N_{\alpha}+ 1}^{n-1})
 \\
- \frac{\delta}{2 \Delta t} \bigg(  \frac{U^{n+1} _ {N_{\alpha} +N } - U^{n+1} _ {N_{\alpha} +N -1 }}{h} - \frac{U^{n-1} _ {N_{\alpha} +N } - U^{n-1} _ {N_{\alpha} +N -1 }}{h} \bigg) (  U^{n+1}_{N_{\alpha} +N } - U^{n-1}_{N_{\alpha} +N }) \\
= \frac{\delta}{2 \Delta t} \sum_{i=N_{\alpha}+1}^{N_{\alpha}+N -1} \bigg (  \frac{U^{n+1}_{i +1 } - U^{n+1}_{i }}{h} -   \frac{U^{n-1}_{i +1 } - U^{n-1}_{i }}{h} \bigg) (U_{i}^{n+1} - U_{i}^{n-1}) 
\\
- \frac{\delta}{2 \Delta t} \sum_{i=N_{\alpha}+1}^{N_{\alpha}+N -1} \bigg (  \frac{U^{n+1}_{i +1 } - U^{n+1}_{i }}{h} -   \frac{U^{n-1}_{i +1 } - U^{n-1}_{i }}{h} \bigg) (U_{i+1}^{n+1} - U_{i+1}^{n-1})  \\
 - \frac{\delta}{2 \Delta t} \bigg(  \frac{U^{n+1} _ {N_{\alpha} + 2 } - U^{n+1} _ {N_{\alpha}  + 1}}{h} - \frac{U^{n-1} _ {N_{\alpha}  + 2} - U^{n-1} _ {N_{\alpha}  + 1}}{h}  \bigg) (U_{N_{\alpha} + 1}^{n+1} - U_{N_{\alpha}+ 1}^{n-1})
 \\
+ \frac{\delta}{2 \Delta t} \bigg(  \frac{U^{n+1} _ {N_{\alpha} +N } - U^{n+1} _ {N_{\alpha} +N -1 }}{h} - \frac{U^{n-1} _ {N_{\alpha} +N } - U^{n-1} _ {N_{\alpha} +N -1 }}{h} \bigg) (U^{n+1}_{N_{\alpha} +N } - U^{n-1}_{N_{\alpha} +N }) \\
 + \frac{\delta}{2 \Delta t} \bigg(  \frac{U^{n+1} _ {N_{\alpha} + 2 } - U^{n+1} _ {N_{\alpha}  + 1}}{h} - \frac{U^{n-1} _ {N_{\alpha}  + 2} - U^{n-1} _ {N_{\alpha}  + 1}}{h}  \bigg) (U_{N_{\alpha} + 1}^{n+1} - U_{N_{\alpha}+ 1}^{n-1})
 \\
- \frac{\delta}{2 \Delta t} \bigg(  \frac{U^{n+1} _ {N_{\alpha} +N } - U^{n+1} _ {N_{\alpha} +N -1 }}{h} - \frac{U^{n-1} _ {N_{\alpha} +N } - U^{n-1} _ {N_{\alpha} +N -1 }}{h} \bigg) (  U^{n+1}_{N_{\alpha} +N } - U^{n-1}_{N_{\alpha} +N }).
\end{array}
\end{equation*}}}
It follows that
\begin{equation} \label{terms dissipation} 
\begin{array}{lll}
\frac{\delta}{2 \Delta t} \sum_{i=N_{\alpha}+2}^{N_{\alpha}+N -1} \bigg (  \frac{U^{n+1}_{i +1 } - U^{n+1}_{i }}{h} -   \frac{U^{n-1}_{i +1 } - U^{n-1}_{i }}{h} -  \frac{U^{n+1}_{i } - U^{n+1}_{i - 1 }}{h} +  \frac{U^{n-1}_{i } - U^{n-1}_{i - 1 }}{h} \bigg) (U_{i}^{n+1} - U_{i}^{n-1}) 
\\
 + \frac{\delta}{2 \Delta t} \bigg(  \frac{U^{n+1} _ {N_{\alpha} + 2 } - U^{n+1} _ {N_{\alpha}  + 1}}{h} - \frac{U^{n-1} _ {N_{\alpha}  + 2} - U^{n-1} _ {N_{\alpha}  + 1}}{h}  \bigg) (U_{N_{\alpha} + 1}^{n+1} - U_{N_{\alpha}+ 1}^{n-1})
 \\
- \frac{\delta}{2 \Delta t} \bigg(  \frac{U^{n+1} _ {N_{\alpha} +N } - U^{n+1} _ {N_{\alpha} +N -1 }}{h} - \frac{U^{n-1} _ {N_{\alpha} +N } - U^{n-1} _ {N_{\alpha} +N -1 }}{h} \bigg) (  U^{n+1}_{N_{\alpha} +N } - U^{n-1}_{N_{\alpha} +N })\\
= - \frac{\delta \; \textcolor{black}{h}}{2 \Delta t} \sum_{i=N_{\alpha}+1}^{N_{\alpha}+N -1}\bigg( \frac{U^{n+1}_{i +1 } - U^{n+1}_{i }}{h} - \frac{U^{n-1}_{i +1 } - U^{n-1}_{i }}{h} \bigg)^2. 
\end{array}
\end{equation}
Thus, setting $\mathcal{E}^{n} = E_k(U^n) + E_p(U^n)$, we get
\begin{equation*}
\mathcal{E}^{n} - \mathcal{E}^ {n-1} = -  \delta \Delta t \; \textcolor{black}{h} \sum_{i=N_{\alpha}+1}^{N_{\alpha}+N -1}\bigg( \frac{U^{n+1}_{i +1 } - U^{n+1}_{i } - U^{n-1}_{i +1 } + U^{n-1}_{i }}{2 \Delta t \ h} \bigg)^2.
\end{equation*}
Therefore, the proof of Theorem \ref{theorem explicit dissipation} is complete and the discrete total energy \eqref{discrete-energy} is non-increasing over time. 
\end{proof}
\subsection{Construction of the implicit discretized problem.} \label{SI}
In this subsection, we will construct the semi-implicit numerical scheme of System \eqref{Eq(2.1)}-\eqref{Eq(2.5)} using FVM in space combined with an average of the unknown $U$ at time $t_{n+1}$ and $t_{n-1}$ in the formulation of the discrete fluxes. 
The time derivatives will be approximated using FDM. Similarly to Subsection \ref{SE}, a reasonable choice for the approximation of 
the flux $-c_i^2 U_x (x_{i+\frac{1}{2}},t_n)$ is the differential quotient:
\begin{equation} \label{flux-implicit}
F_{i+{\frac{1}{2}}}^{n} =\left\{\begin{array}{lllll}
- \frac{C_1^2}{2} \bigg(\frac{u_1^{n+1}}{h_{\frac{1}{2}}}+\frac{u_1^{n-1}}{h_{\frac{1}{2}}} \bigg), & i = 0, \\
- \frac{c_i^2}{2}  \bigg( \frac{U_{i+1}^{n+1} - U_i^{n+1}}{h_{i+\frac{1}{2}}} + \frac{U_{i+1}^{n-1} - U_i^{n-1}}{h_{i+\frac{1}{2}}} \bigg) , & i = 1, \ldots, N_{\max}-1, \\
\frac{C_3^2}{2} \bigg( \frac{w_{N_{\max}}^{n+1}}{h_{N_{\max}+\frac{1}{2}}} + \frac{w_{N_{\max}}^{n-1}}{h_{N_{\max}+\frac{1}{2}}} \bigg), & i = N_{\max}.
\end{array}\right.    
\end{equation}
As this numerical flux depends on the value of the unknowns $U^n$ and $U^{n+1}$ at time $t_n$ and $t_{n+1}$ we decide to call it \textbf{implicit numerical flux} and the discretized problem is called 
\textbf{implicit discretized problem}.

We will adapt the formula for the implicit numerical fluxes at the inner boundary points \linebreak  $x_{N_\alpha+ \frac{1}{2}}=\alpha$ and $ x_{N_\beta+ \frac{1}{2}}=\beta$ 
using the same argument (continuity of the fluxes) as described in the explicit numerical scheme and the formula for $\ell_{N_{\alpha}+\frac{1}{2}}$ and $\ell_{N_{\alpha}+N+\frac{1}{2}}$ are 
exactly the same. We will therefore not detail the computations.

\bigskip 

\textbf{Discretization of Equation \eqref{Eq(2.1)}.} \\
$\Box$ For $i = 1, \cdots, N_{\alpha} - 1$:
\begin{align*} 
  \int_{K_i} u_{tt} dx - C_1^2  \int_{K_i} u_{xx} dx &= 0,  \\
  \int_{x_{i-\frac{1}{2}}}^{x_{i+\frac{1}{2}}} u_{tt} dx - C_1^2  \int_{x_{i-\frac{1}{2}}}^{x_{i+\frac{1}{2}}} u_{xx} dx &= 0. 
\end{align*}
A direct calculation gives
\begin{equation*}
\begin{array}{lll}
h_i [u_i(t_n)]_{tt} - C_1^2 [ u_x(x_{i+\frac{1}{2}},t_{n}) - u_x(x_{i-\frac{1}{2}},t_n)] = 0.
\end{array}
\end{equation*}
So, using second-order central time discretization at $t_{n}$ and the definition of the numerical fluxes \eqref{flux-implicit}, we obtain:
\begin{equation}\label{Eq-implicit-1}
  h_{i}  \frac{u_i^{n+1}- 2 u_i^n + u_i^{n-1}}{\Delta t^2} - \frac{1}{2} \big[ \ell_{i+ \frac{1}{2}} ( u^{n+1}_{i + 1} - u^{n+1}_{i }) - \ell_{i- \frac{1}{2}} (u^{n+1}_{i } - u^{n+1}_{i - 1 }) \big] - \frac{1}{2} \big[ \ell_{i+ \frac{1}{2}} ( u^{n-1}_{i + 1} - u^{n-1}_{i }) - \ell_{i- \frac{1}{2}} (u^{n-1}_{i } - u^{n-1}_{i - 1 }) \big] = 0.
\end{equation}
$\Box$ For $i = N_{\alpha}$:
\begin{align*} 
  \int_{x_{N_{\alpha}-\frac{1}{2}}}^{x_{N_{\alpha}+\frac{1}{2}}} u_{tt} dx - C_1^2  \int_{x_{N_{\alpha}-\frac{1}{2}}}^{x_{N_{\alpha}+\frac{1}{2}}} u_{xx} dx &= 0, \\
h_{N_{\alpha}} [u_{N_{\alpha}}(t_n)]_{tt} - C_1^2 [ u_x(x_{N_{\alpha}+\frac{1}{2}},t_n) - u_x(x_{N_{\alpha}-\frac{1}{2}},t_n)] &= 0.
\end{align*}
So, using second-order central time discretization at $t_{n}$, the definition of the numerical fluxes \eqref{flux-implicit} and the 
continuity of the flux at the inner boundary $x_{N_{\alpha}+ \frac{1}{2}}= \alpha$, we obtain the same formula for the quantity $\ell_{N_{\alpha}+\frac{1}{2}}$. So we get:
\begin{equation*}\label{Eq-implicit-2}
\begin{array}{lll}
  h_{N_\alpha}  \frac{u_{N_{\alpha}}^{n+1}- 2 u_{N_{\alpha}}^n + u_{N_{\alpha}}^{n-1}}{\Delta t^2}  - \frac{1}{2}\big [ \ell_{N_{\alpha}+\frac{1}{2}} (v^{n+1}_{N_{\alpha} + 1} - u^{n+1}_{N_{\alpha}}) - \ell_{N_{\alpha}-\frac{1}{2}} (u^{n+1}_{N_{\alpha} } - u^{n+1}_{N_{\alpha} -1}) \big] \\
  - \frac{1}{2}\big [ \ell_{N_{\alpha}+\frac{1}{2}} (v^{n-1}_{N_{\alpha} + 1} - u^{n-1}_{N_{\alpha}}) - \ell_{N_{\alpha}-\frac{1}{2}} (u^{n-1}_{N_{\alpha} } - u^{n-1}_{N_{\alpha} -1}) \big] = 0.
  \end{array}
\end{equation*}
\textbf{Discretization of Equation \eqref{Eq(2.2)}.} \\
$\Box$ For $i = N_{\alpha} + 1$:
\begin{align*} 
  \int_{x_{N_{\alpha}+\frac{1}{2}}}^{x_{N_{\alpha}+\frac{3}{2}}} v_{tt} dx - C_2^2  \int_{x_{N_{\alpha}+\frac{1}{2}}}^{x_{N_{\alpha}+\frac{3}{2}}} v_{xx} dx - \delta \int_{x_{N_{\alpha}+\frac{1}{2}}}^{x_{N_{\alpha}+\frac{3}{2}}} v_{xxt} dx &= 0, \\
h_{N_{\alpha}+1} [v_{N_{\alpha}+1}(t_n)]_{tt} - C_2^2 [ v_x(x_{N_{\alpha}+\frac{3}{2}},t_n) - v_x(x_{N_{\alpha}+\frac{1}{2}},t_n)]  - \delta [ v_{xt}(x_{N_{\alpha}+ \frac{3}{2}},t_n) - v_{xt}(x_{N_{\alpha} + \frac{1}{2}},t_n)]&= 0.
\end{align*}
As discussed in Subsection \ref{SE} (numerical scheme for $i = N_{\alpha} + 1$), since the position $x_{N_{\alpha}+\frac{1}{2}}$ represents the point $\alpha$, we use the transmission condition \eqref{Eq(2.5)} at point $\alpha$ and thus
\begin{equation*}
C_2^2 v_x(x_{N_{\alpha}+\frac{1}{2}},t_n) + \delta  v_{xt}(x_{N_{\alpha} + \frac{1}{2}},t_n) = C_1^2 u_x(x_{N_{\alpha}+\frac{1}{2}},t_n).
\end{equation*}
Consequently, we obtain
\begin{equation*}
 h [v_{N_{\alpha}+1}(t_n)]_{tt} + C_1^2 u_x(x_{N_{\alpha}+\frac{1}{2}},t_n)- C_2^2  v_x(x_{N_{\alpha}+\frac{3}{2}},t_n) - \delta v_{xt}(x_{N_{\alpha}+ \frac{3}{2}},t_n) = 0.
\end{equation*}
For the first term of the above equation we apply the second-order central time discretization at $t_n$. However, for the third term, we use spatial forward difference combined with an average of the value of $U$ at time $t_{n+1}$ and $t_{n-1}$. Concerning the fourth term, a forward difference in space with central difference in time is applied. Note that the second term is treated as previously (for $i = N_{\alpha}$). Therefore, we obtain:
\begin{equation}\label{Eq-implicit-3}
\begin{array}{lll}
  h_{N_\alpha +1 }  \frac{v_{N_{\alpha}+1}^{n+1}- 2 v_{N_{\alpha}+1}^n + v_{N_{\alpha}+1}^{n-1}}{\Delta t^2} - \frac{1}{2}\big [ \ell_{N_{\alpha}+\frac{3}{2}} (v^{n+1}_{N_{\alpha} + 2} - v^{n+1}_{N_{\alpha}+1}) - \ell_{N_{\alpha}+\frac{1}{2}} (v^{n+1}_{N_{\alpha}+1 } - u^{n+1}_{N_{\alpha} }) \big]
  \\
  - \frac{1}{2}\big [ \ell_{N_{\alpha}+\frac{3}{2}} (v^{n-1}_{N_{\alpha} + 2} - v^{n-1}_{N_{\alpha}+1}) - \ell_{N_{\alpha}+\frac{1}{2}} (v^{n-1}_{N_{\alpha}+1 } - u^{n-1}_{N_{\alpha} }) \big]
  $ $= \frac{\delta}{2 \Delta t} \bigg[  \frac{{v^{n+1}_{N_{\alpha}+2} - v^{n+1}_{N_{\alpha}+1}}}{h_{N_\alpha +1 }} - \frac{{v^{n-1}_{N_{\alpha}+2} - v^{n-1}_{N_{\alpha}+1}}}{h_{N_\alpha +1 }} \bigg]
\end{array}
\end{equation}
$\Box$ For $i = N_{\alpha} + 2, \cdots, N_{\alpha} + N -1 $:
\begin{align*} 
  \int_{x_{i}-\frac{1}{2}}^{x_{i}+\frac{1}{2}} v_{tt} dx - C_2^2  \int_{x_{i}-\frac{1}{2}}^{x_{i}+\frac{1}{2}} v_{xx} dx - \delta \int_{x_{i}-\frac{1}{2}}^{x_{i}+\frac{1}{2}} v_{xxt} dx &= 0,\\  \nonumber
h_{i} [v_{i}(t_n)]_{tt} - C_2^2 [ v_x(x_{i+\frac{1}{2}},t_n) - v_x(x_{i-\frac{1}{2}},t_n)]  - \delta [ v_{xt}(x_{i+\frac{1}{2}},t_n) - v_{xt}(x_{i-\frac{1}{2}},t_n)]& = 0.
\end{align*}
We proceed exactly as before : we use a second-order central difference on the double time derivative. For the second and third terms we apply a forward difference in space combined with an average of the value of $U$ at time $t_{n+1}$ and $t_{n-1}$. However, we use a forward difference together with centered time discretization for the last two terms and we obtain:
\begin{equation*} \label{Eq-implicit-4}
\begin{array}{lll}
    h_i\frac{v_i^{n+1}- 2 v_i^n + v_i^{n-1}}{\Delta t^2} - \frac{1}{2} \big[ \ell_{i+ \frac{1}{2}} ( v^{n+1}_{i + 1} - v^{n+1}_{i }) - \ell_{i- \frac{1}{2}} (v^{n+1}_{i } - v^{n+1}_{i - 1 }) \big] - \frac{1}{2} \big[ \ell_{i+ \frac{1}{2}} ( v^{n-1}_{i + 1} - v^{n-1}_{i }) - \ell_{i- \frac{1}{2}} (v^{n-1}_{i } - v^{n-1}_{i - 1 }) \big] \\
 = \frac{\delta}{2 \Delta t} \bigg [   \frac{v^{n+1}_{i +1 } - v^{n+1}_{i }}{h_i} -   \frac{v^{n-1}_{i +1 } - v^{n-1}_{i }}{h_i} -  \frac{v^{n+1}_{i } - v^{n+1}_{i - 1 }}{h_i} +  \frac{v^{n-1}_{i } - v^{n-1}_{i - 1 }}{h_i} \bigg]
 \end{array}
\end{equation*}
$\Box$ For $i = N_{\alpha}+N$:
\begin{equation*}
\begin{array}{lll} 
  \int_{x_{N_{\alpha}+N-\frac{1}{2}}}^{x_{N_{\alpha}+N+\frac{1}{2}}} v_{tt} dx - C_2^2  \int_{x_{N_{\alpha}+N-\frac{1}{2}}}^{x_{N_{\alpha}+N+\frac{1}{2}}} v_{xx} dx - \delta \int_{x_{N_{\alpha}+N-\frac{1}{2}}}^{x_{N_{\alpha}+N+\frac{1}{2}}} v_{xxt} dx = 0, \\
h_{N_{\alpha}+N} [v_{N_{\alpha}+N}(t_n)]_{tt} - C_2^2 [ v_x(x_{N_{\alpha}+N+\frac{1}{2}},t_n) - v_x(x_{N_{\alpha}+N-\frac{1}{2}},t_n)]  - \delta [ v_{xt}(x_{N_{\alpha}+ N+ \frac{1}{2}},t_n) - v_{xt}(x_{N_{\alpha} + N - \frac{1}{2}},t_n)]= 0.
 \end{array}
\end{equation*}
Similarly to \eqref{eq-interface-beta}, as the position $x_{N_{\alpha}+N+\frac{1}{2}}$ represents the point $\beta$, then we use the transmission condition \eqref{Eq(2.5)} at point $\beta$ and thus
\begin{equation*}
- C_2^2 v_x(x_{N_{\alpha}+N+\frac{1}{2}},t_n) - \delta  v_{xt}(x_{N_{\alpha} + N+ \frac{1}{2}},t_n) = - C_3^2 w_x(x_{N_{\alpha}+N+\frac{1}{2}},t_n).
\end{equation*}
Consequently, we obtain
\begin{equation*}
h_{N_{\alpha}+N}[v_{N_{\alpha}+N}(t_n)]_{tt} - C_3^2 w_x(x_{N_{\alpha}+N+\frac{1}{2}},t_n) + C_2^2  v_x(x_{N_{\alpha}+N-\frac{1}{2}},t_n)  + \delta  v_{xt}(x_{N_{\alpha} + N - \frac{1}{2}},t_n)   = 0.
\end{equation*}
Similar to the way used for $i = N_{\alpha}+1$ (see \eqref{Eq-implicit-3}), we apply the second-order central difference approximation in time for the second time derivative. For the third term of the above equation, we apply a forward difference in space at position $x_{N_{\alpha}+N+ \frac{1}{2}}$ combined with an average of the value of $U$ at time $t_{n+1}$ and $t_{n-1}$. A central difference in time and a spatial forward difference in space is applied on the fourth term. However, we treat the second term exactly as we treated it in the explicit scheme (see Subsection \ref{SE} equation \eqref{eq-interface-beta}), using the average of the value of $U$ at time $t_{n+1}$ and $t_{n-1}$. Here again as $x_{N_{\alpha}+N+\frac{1}{2}} = \beta$
we write the continuity of the fluxes to obtain the same value of $\ell_{N_{\alpha}+N+\frac{1}{2}}$ as in the explicit discretization.
Therefore, we obtain:
\begin{equation*}\label{Eq-implicit-5}
\begin{array}{lll}
  h_{N_{\alpha}+N}  \frac{v_{N_{\alpha}+N}^{n+1}- 2 v_{N_{\alpha}+N}^n + v_{N_{\alpha}+N}^{n-1}}{\Delta t^2} - \frac{1}{2}\big [ \ell_{N_{\alpha}+N+\frac{1}{2}} (w^{n+1}_{N_{\alpha}+N + 1} - v^{n+1}_{N_{\alpha}+N}) - \ell_{N_{\alpha}+N-\frac{1}{2}} (v^{n+1}_{N_{\alpha}+N } - v^{n+1}_{N_{\alpha} +N-1}) \big] \\
- \frac{1}{2}\big [ \ell_{N_{\alpha}+N+\frac{1}{2}} (w^{n-1}_{N_{\alpha}+N + 1} - v^{n-1}_{N_{\alpha}+N}) - \ell_{N_{\alpha}+N-\frac{1}{2}} (v^{n-1}_{N_{\alpha}+N } - v^{n-1}_{N_{\alpha} +N-1}) \big] \\
= - \frac{\delta}{2 \Delta t} \bigg[  \frac{{v^{n+1}_{N_{\alpha}+N} - v^{n+1}_{N_{\alpha}+N-1}}
}{h_{N_{\alpha}+N}} - \frac{{v^{n-1}_{N_{\alpha}+N} - v^{n-1}_{N_{\alpha}+N-1}}}{h_{N_{\alpha}+N}} \bigg].
\end{array}
\end{equation*}
\textbf{Discretization of Equation \eqref{Eq(2.3)}.} \\
$\Box$ For $i = N_{\alpha} + N + 1$:
\begin{align*} 
  \int_{x_{N_{\alpha}+N+\frac{3}{2}}}^{x_{N_{\alpha}+N+\frac{1}{2}}} w_{tt} dx - C_3^2  \int_{x_{N_{\alpha}+N+\frac{1}{2}}}^{x_{N_{\alpha}+N+\frac{3}{2}}} w_{xx} dx &= 0, \\
h_{N_{\alpha}+N+1} [w_{N_{\alpha}+N+1}(t_n)]_{tt} - C_3^2 [ w_x(x_{N_{\alpha}+N+\frac{3}{2}},t_n) - w_x(x_{N_{\alpha}+N+\frac{1}{2}},t_n)] &= 0.
\end{align*}
The last term of the above equation is treated in the previous part (for $i = N_{\alpha}+N$). Also, using the second-order central difference in time and forward difference discretization in space with an average of $n+1$ and $n-1$ time step diffusion, we obtain
\begin{equation*}\label{Eq-implicit-6}
\begin{array}{lll}
   h_{N_{\alpha}+N+1} \frac{w_{N_{\alpha}+N+1}^{n+1}- 2 w_{N_{\alpha}+N+1}^n + w_{N_{\alpha}+N+1}^{n-1}}{\Delta t^2}  - \frac{1}{2} \big [ \ell_{N_{\alpha}+N+{\frac{3}{2}}} (w^{n+1}_{N_{\alpha}+N+2 } - w^{n+1}_{N_{\alpha} +N+1}) -  \ell_{N_{\alpha}+N+{\frac{1}{2}}}  (w^{n+1}_{N_{\alpha}+N + 1} - v^{n+1}_{N_{\alpha}+N})  \big] \\
   - \frac{1}{2} \big [ \ell_{N_{\alpha}+N+{\frac{3}{2}}} (w^{n-1}_{N_{\alpha}+N+2 } - w^{n-1}_{N_{\alpha} +N+1}) -  \ell_{N_{\alpha}+N+{\frac{1}{2}}}  (w^{n-1}_{N_{\alpha}+N + 1} - v^{n-1}_{N_{\alpha}+N})  \big] = 0.
 \end{array}
\end{equation*}
$\Box$ For $i =N_{\alpha} + N +2  , \cdots, N_{\max}$:        \\[0.1in]
Just like the treatment of the first equation (for $i = 1, \cdots, N_{\alpha}-1$, see \eqref{Eq-implicit-1}), we do the same thing to obtain
\begin{equation*}\label{Eq-implicit-7}
  h_{i}  \frac{w_i^{n+1}- 2 w_i^n + w_i^{n-1}}{\Delta t^2} - \frac{1}{2} \big[ \ell_{i+ \frac{1}{2}} ( w^{n+1}_{i + 1} - w^{n+1}_{i }) - \ell_{i- \frac{1}{2}} (w^{n+1}_{i } - w^{n+1}_{i - 1 }) \big] - \frac{1}{2} \big[ \ell_{i+ \frac{1}{2}} ( w^{n-1}_{i + 1} - w^{n-1}_{i }) - \ell_{i- \frac{1}{2}} (w^{n-1}_{i } - w^{n-1}_{i - 1 }) \big] = 0.
\end{equation*}
Now, the implicit discrete problem combining an average of the value of $U$ at time $t_{n+1}$ and $t_{n-1}$ can be written in the compact form
\begin{equation} \label{DISCRETE-IMPLICIT}
\begin{array}{ll}
h_i \frac{U_i^{n+1}-2U_i^n + U_i^{n-1}}{\Delta t^2} &- \frac{1}{2} [\ell_{i+{\frac{1}{2}}} (U_{i+1}^{n+1} - U_i^{n+1})
 - \ell_{i-{\frac{1}{2}}} (U_{i}^{n+1} - U_{i-1}^{n+1})] \\[0.1in]
&- \frac{1}{2} [\ell_{i+{\frac{1}{2}}} (U_{i+1}^{n-1} - U_i^{n-1}) - \ell_{i-{\frac{1}{2}}} (U_{i}^{n-1} - U_{i-1}^{n-1})] = rhs_i
\end{array}
\end{equation}
where $rhs_i$ is defined by: for
\begin{align*}
&i = 1, \ldots, N_{\alpha}\,,\,rhs_i = 0 \\
&i = N_{\alpha}+1\,,\,rhs_i = \frac{\delta}{2 \Delta t} \bigg[  \frac{{U^{n+1}_{N_{\alpha}+2} - U^{n+1}_{N_{\alpha}+1}}}{h_{N_{\alpha}+1}}	 - \frac{{U^{n-1}_{N_{\alpha}+2} - U^{n-1}_{N_{\alpha}+1}}}{h_{N_{\alpha}+1}} \bigg] \\
&i = N_{\alpha}+2, \ldots, N_{\alpha}+N-1 \,,\, rhs_i =\frac{\delta}{2 \Delta t} \bigg [   \frac{U^{n+1}_{i +1 } - U^{n+1}_{i }}{h_i} -   \frac{U^{n-1}_{i +1 } - U^{n-1}_{i }}{h_i} -  \frac{U^{n+1}_{i } - U^{n+1}_{i - 1 }}{h_i} +  \frac{U^{n-1}_{i } - U^{n-1}_{i - 1 }}{h_i} \bigg] \\
& i = N_{\alpha}+N \,,\, rhs_i = - \frac{\delta}{2 \Delta t} \bigg[  \frac{{U^{n+1}_{N_{\alpha}+N} - U^{n+1}_{N_{\alpha}+N-1}}
}{h_{N_{\alpha}+N}} - \frac{{U^{n-1}_{N_{\alpha}+N} - U^{n-1}_{N_{\alpha}+N-1}}}{h_{N_{\alpha}+N}} \bigg]\\
&i = N_{\alpha}+N+1, \ldots, N_{\max}\,,\,rhs_i = 0 
\end{align*}
Consequently, the discrete implicit problem which represents an approximation of System \eqref{Eq(2.1)}-\eqref{Eq(2.5)} can be written as:
 \begin{equation}\label{DISCRETE-IMPLICIT-Matrix}
\mbox{for } n = 1,\ldots,\mathcal{N} \,,\, \bigg(\mathcal{M} + \frac{\Delta t^2}{2}\mathcal{B}+ \frac{\delta \Delta t}{h}\mathcal{A} \bigg) U^{n+1} = 2 \mathcal{M} U^n - \bigg(\mathcal{M} + \frac{\Delta t^2}{2}\mathcal{B} - \frac{\delta \Delta t}{h}\mathcal{A} \bigg) U^{n-1}
\end{equation}
Writing the numerical scheme \eqref{DISCRETE-IMPLICIT-Matrix} for $n=0$ and using the value of $U^{-1}$ defined by \eqref{artificialSOLUTION} we get :
\begin{equation} \label{DISCRETE-IMPLICIT-Matrix1}
\bigg(2 \mathcal{M} + \Delta t^2 \mathcal{B}\bigg) U^{1}=
2 \mathcal{M} U^0 + 2 \Delta t \bigg(\mathcal{M} + \frac{\Delta t^2}{2}\mathcal{B} - \frac{\delta \Delta t}{h}\mathcal{A} \bigg) \Psi_{\mathcal{T}}  
\end{equation}

\begin{Remark}
Let us remark that the matrices $\bigg(2 \mathcal{M} + \Delta t^2 \mathcal{B}\bigg)$ and $\bigg(\mathcal{M} + \frac{\Delta t^2}{2}\mathcal{B}+ \frac{\delta \Delta t}{h}\mathcal{A} \bigg)$  are irreducible diagonal dominant. 
So these two matrices are also invertible (see \cite{Golub-Meurant}). Therefore, equations \eqref{DISCRETE-IMPLICIT-Matrix} and \eqref{DISCRETE-IMPLICIT-Matrix1} are well-posed.
\end{Remark}

\begin{center}\textbf{\textit{Practical implementation}}\end{center}
\hspace*{-0.1cm}\rule[-1mm]{15cm}{0.05cm}\\
\textbf{Initialisation}\\
$U^0 = \Phi_{\mathcal{T}}$\\
\hspace*{-0.1cm}\rule[-1mm]{15cm}{0.05cm}\\
\textbf{Computation of $\boldsymbol{U^1}$}\\
$\bigg(2 \mathcal{M} + \Delta t^2 \mathcal{B}\bigg) U^{1}=
2 \mathcal{M} U^0 + 2 \Delta t \bigg(\mathcal{M} + \frac{\Delta t^2}{2}\mathcal{B} - \frac{\delta \Delta t}{h}\mathcal{A} \bigg) \Psi_{\mathcal{T}}$ \\
\hspace*{-0.1cm}\rule[-1mm]{15cm}{0.05cm}\\
\textbf{Compute the inverse of the matrix} $\boldsymbol{\bigg(\mathcal{M} + \frac{\Delta t^2}{2}\mathcal{B}+ \frac{\delta \Delta t}{h}\mathcal{A} \bigg)}$.\\
$C =\bigg(\mathcal{M} + \frac{\Delta t^2}{2}\mathcal{B}+ \frac{\delta \Delta t}{h}\mathcal{A} \bigg)^{-1}$.\\
\hspace*{-0.1cm}\rule[-1mm]{15cm}{0.05cm}\\
\textbf{Computation of $\boldsymbol{U^{n+1}}$ for $\boldsymbol{n=1,\ldots,\mathcal{N}}$}\\
for $n=1 \ldots, \mathcal{N}$ \\
\hspace*{1cm} $RHS = 2 \mathcal{M} U^n - \bigg(\mathcal{M} + \frac{\Delta t^2}{2}\mathcal{B} - \frac{\delta \Delta t}{h}\mathcal{A} \bigg) U^{n-1}$ \\
\hspace*{1cm} $U^{n+1} = C \times RHS$\\
endfor\\
\hspace*{-0.1cm}\rule[-1mm]{15cm}{0.05cm}\\

\subsection{Dissipation of the discrete energy for the implicit discretized problem} \label{S(IE)}
In this subsection, we also seek to design a discrete energy that is preserved when $\delta = 0$, whereas when $\delta >0$, we aim to have a dissipation in the discrete energy. \textcolor{black}{Once again, for this aim, we copy the definition of the continuous energy \eqref{energy} and we define:}
\begin{itemize}
\item[$\bullet$] the discrete kinetic energy for $U$ as: 
\[
E_k(U^n) =  \frac{1}{2} \sum_{i=1}^{N_{\max}} h_i \bigg(\frac{U_{i}^{n+1}-U_{i}^n}{\Delta t}\bigg)^2
\]
\item[$\bullet$] the discrete potential energy for $U$ as:
\[E_p(U^n) = \frac{1}{4}\sum_{i=0}^{N_{\max}} \ell_{i+{\frac{1}{2}}} (U_{i+1}^{n+1} - U_i^{n+1})^2 + \frac{1}{4}\sum_{i=0}^{N_{\max}} \ell_{i+{\frac{1}{2}}} (U_{i+1}^{n} - U_i^{n})^2 \ .
\]
\end{itemize}
The total discrete energy is then defined as
\begin{equation}\label{discrete-energy-implicit}
\mathcal{E}^{n} = E_k(U^n) + E_p(U^n).
\end{equation}
Like in Subsection \ref{SE}, we want to prove that the discrete energy defined by \eqref{discrete-energy-implicit} fulfills the properties mentioned above. To this end, we prove the following theorem:
\begin{Theorem} \label{theorem implicit dissipation}
The total discrete energy defined by \eqref{discrete-energy-implicit} of the semi-implicit numerical scheme \eqref{DISCRETE-IMPLICIT} satisfies the following \textcolor{black}{dissipativity estimation}
\begin{equation*}
\mathcal{E}^{n} - \mathcal{E}^ {n-1} = -  \delta \Delta t \; \textcolor{black}{h} \sum_{i=N_{\alpha}+1}^{N_{\alpha}+N -1}\bigg( \frac{U^{n+1}_{i +1 } - U^{n+1}_{i } - U^{n-1}_{i +1 } + U^{n-1}_{i }}{2 \Delta t \ h} \bigg)^2.
\end{equation*}
\end{Theorem}
\begin{proof}
 \textcolor{black}{Once again, let us first mention that this estimation is the discrete version of the dissipativity estimation of the continuous solution \eqref{dissipative}.}\\Similarly to the technique used in Subsection \ref{SE}, to obtain the energy estimates, we multiply the left and right hand-sides of \eqref{DISCRETE-IMPLICIT} by $(U_i^{n+1}-U_i^{n-1})$ and we sum over $i=1, \ldots, N_{\max}$.\\[0.1in]
\textbf{Estimation of the left hand-side of \eqref{DISCRETE-IMPLICIT}}$\times (U_i^{n+1}-U_i^{n-1})$: \\
\noindent Concerning the first term, we get exactly the same result as \eqref{result-energy-kinetic}.  \\[0.1in]
\textbf{Estimation of the second term of the left hand-side of \eqref{DISCRETE-IMPLICIT}}$\times (U_i^{n+1}-U_i^{n-1})$: First of all, we have
\begin{equation}
\begin{array}{lll}
- \frac{1}{2}\sum_{i=1}^{N_{\max}}[\ell_{i+{\frac{1}{2}}} (U_{i+1}^{n+1} - U_i^{n+1}) - \ell_{i-{\frac{1}{2}}} (U_{i}^{n+1} - U_{i-1}^{n+1})](U_i^{n+1}-U_i^{n-1}) \\
- \frac{1}{2} \sum_{i=1}^{N_{\max}}[\ell_{i+{\frac{1}{2}}} (U_{i+1}^{n-1} - U_i^{n-1}) - \ell_{i-{\frac{1}{2}}} (U_{i}^{n-1} - U_{i-1}^{n-1})](U_i^{n+1}-U_i^{n-1}) \\
= - \frac{1}{2}\sum_{i=1}^{N_{\max}}\ell_{i+{\frac{1}{2}}} (U_{i+1}^{n+1} - U_i^{n+1}) (U_i^{n+1}-U_i^{n-1}) + \frac{1}{2}\sum_{i=1}^{N_{\max}}\ell_{i-{\frac{1}{2}}} (U_{i}^{n+1} - U_{i-1}^{n+1}) (U_i^{n+1}-U_i^{n-1}) \\
- \frac{1}{2}\sum_{i=1}^{N_{\max}}\ell_{i+{\frac{1}{2}}} (U_{i+1}^{n-1} - U_i^{n-1}) (U_i^{n+1}-U_i^{n-1}) + \frac{1}{2}\sum_{i=1}^{N_{\max}}\ell_{i-{\frac{1}{2}}} (U_{i}^{n-1} - U_{i-1}^{n-1}) (U_i^{n+1}-U_i^{n-1}).
\end{array}
\end{equation}
By translation of index $i$ for the second and fourth term of the right hand-side of the above equation, we obtain 
\begin{equation*}
\begin{array}{lll}
 - \frac{1}{2}\sum_{i=1}^{N_{\max}}[\ell_{i+{\frac{1}{2}}} (U_{i+1}^{n+1} - U_i^{n+1}) - \ell_{i-{\frac{1}{2}}} (U_{i}^{n+1} - U_{i-1}^{n+1})](U_i^{n+1}-U_i^{n-1}) \\
- \frac{1}{2} \sum_{i=1}^{N_{\max}}[\ell_{i+{\frac{1}{2}}} (U_{i+1}^{n-1} - U_i^{n-1}) - \ell_{i-{\frac{1}{2}}} (U_{i}^{n-1} - U_{i-1}^{n-1})](U_i^{n+1}-U_i^{n-1}) \\
=  - \frac{1}{2}\sum_{i=1}^{N_{\max}}\ell_{i+{\frac{1}{2}}} (U_{i+1}^{n+1} - U_i^{n+1}) (U_i^{n+1}-U_i^{n-1}) + \frac{1}{2}\sum_{i=0}^{N_{\max}-1}\ell_{i+{\frac{1}{2}}} (U_{i+1}^{n+1} - U_{i}^{n+1}) (U_{i+1}^{n+1}-U_{i+1}^{n-1}) \\
- \frac{1}{2}\sum_{i=1}^{N_{\max}}\ell_{i+{\frac{1}{2}}} (U_{i+1}^{n-1} - U_i^{n-1}) (U_i^{n+1}-U_i^{n-1}) + \frac{1}{2}\sum_{i=0}^{N_{\max}-1}\ell_{i+{\frac{1}{2}}} (U_{i+1}^{n-1} - U_{i}^{n-1}) (U_{i+1}^{n+1}-U_{i+1}^{n-1}).
 \end{array}
\end{equation*}
Taking into consideration that $U_0=U_{N_{\max}+1}=0$, it follows
\begin{equation*}
\begin{array}{lll}
 - \frac{1}{2}\sum_{i=1}^{N_{\max}}[\ell_{i+{\frac{1}{2}}} (U_{i+1}^{n+1} - U_i^{n+1}) - \ell_{i-{\frac{1}{2}}} (U_{i}^{n+1} - U_{i-1}^{n+1})](U_i^{n+1}-U_i^{n-1}) \\
- \frac{1}{2} \sum_{i=1}^{N_{\max}}[\ell_{i+{\frac{1}{2}}} (U_{i+1}^{n-1} - U_i^{n-1}) - \ell_{i-{\frac{1}{2}}} (U_{i}^{n-1} - U_{i-1}^{n-1})](U_i^{n+1}-U_i^{n-1}) \\
=  \frac{1}{2}\sum_{i=0}^{N_{\max}}\ell_{i+{\frac{1}{2}}} (U_{i+1}^{n+1} - U_i^{n+1}) (U_{i+1}^{n+1} -U_i^{n+1} -U_{i+1}^{n-1} +U_i^{n-1}) \\
+\frac{1}{2}\sum_{i=1}^{N_{\max}}\ell_{i+{\frac{1}{2}}} (U_{i+1}^{n-1} - U_i^{n-1}) (U_{i+1}^{n+1} -U_i^{n+1} -U_{i+1}^{n-1} +U_i^{n-1}) .
\end{array}
\end{equation*}
A direct calculation gives
{\small{\begin{equation*}
\begin{array}{lll}
 - \frac{1}{2}\sum_{i=1}^{N_{\max}}[\ell_{i+{\frac{1}{2}}} (U_{i+1}^{n+1} - U_i^{n+1}) - \ell_{i-{\frac{1}{2}}} (U_{i}^{n+1} - U_{i-1}^{n+1})](U_i^{n+1}-U_i^{n-1}) \\
- \frac{1}{2} \sum_{i=1}^{N_{\max}}[\ell_{i+{\frac{1}{2}}} (U_{i+1}^{n-1} - U_i^{n-1}) - \ell_{i-{\frac{1}{2}}} (U_{i}^{n-1} - U_{i-1}^{n-1})](U_i^{n+1}-U_i^{n-1}) \\
= \frac{1}{2}\sum_{i=0}^{N_{\max}}\ell_{i+{\frac{1}{2}}} (U_{i+1}^{n+1} - U_i^{n+1} + U_{i+1}^{n-1} - U_i^{n-1}) (U_{i+1}^{n+1} -U_i^{n+1} -U_{i+1}^{n-1} +U_i^{n-1})\\
=\frac{1}{2}\sum_{i=0}^{N_{\max}}\ell_{i+{\frac{1}{2}}} (U_{i+1}^{n+1} - U_i^{n+1})^2 - 
\frac{1}{2}\sum_{i=0}^{N_{\max}}\ell_{i+{\frac{1}{2}}} (U_{i+1}^{n-1} - U_i^{n-1})^2 \\
=\frac{1}{2}\sum_{i=0}^{N_{\max}}\ell_{i+{\frac{1}{2}}} (U_{i+1}^{n+1} - U_i^{n+1})^2 + \frac{1}{2}\sum_{i=0}^{N_{\max}}\ell_{i+{\frac{1}{2}}} (U_{i+1}^{n} - U_i^{n})^2 - \frac{1}{2}\sum_{i=0}^{N_{\max}}\ell_{i+{\frac{1}{2}}} (U_{i+1}^{n} - U_i^{n})^2 - \frac{1}{2}\sum_{i=0}^{N_{\max}}\ell_{i+{\frac{1}{2}}} (U_{i+1}^{n-1} - U_i^{n-1})^2 \\
= 2(E_p(U^n)-E_p(U^{n-1})).
\end{array}
\end{equation*} }}
The left hand-side of \eqref{DISCRETE-IMPLICIT}$\times (U_i^{n+1}-U_i^{n-1})$ represents the discrete time derivative of the total discrete energy defined by \eqref{discrete-energy-implicit}. It remains to show that this energy is conservative when $\delta=0$ and dissipative when $\delta >0$. Therefore, we should study the right hand-side of Equation \eqref{DISCRETE-IMPLICIT}$\times (U_i^{n+1}-U_i^{n-1})$. In fact, we get exactly the same result as in \eqref{terms dissipation}. \\[0.1in]
Thus, setting $\mathcal{E}^n = E_k(U^n) + E_p(U^n)$, we get
\begin{equation*}
\mathcal{E}^{n} - \mathcal{E}^ {n-1} = -  \delta \Delta t \; \textcolor{black}{h} \sum_{i=N_{\alpha}+1}^{N_{\alpha}+N -1}\bigg( \frac{U^{n+1}_{i +1 } - U^{n+1}_{i } - U^{n-1}_{i +1 } + U^{n-1}_{i }}{2 \Delta t \ h} \bigg)^2.
\end{equation*}
Consequently, the proof of Theorem \ref{theorem implicit dissipation} is complete and the discrete total energy \eqref{discrete-energy-implicit} is non-increasing over time. 
\end{proof} 
\noindent Before we start figuring out the stability estimates, we will write Equation \eqref{DISCRETE-IMPLICIT} for $n=0$. After substituting \eqref{artificialSOLUTION}, a direct calculation gives 
{\small{\begin{equation} \label{DISCRETE-IMPLICIT-n=0}
\begin{array}{lll}
2 h_i \frac{U_i^1 - U_i^0}{\Delta t^2} - \frac{2 h_i}{\Delta t} \Psi_i - [\ell_{i+ \frac{1}{2}} (U_{i+1}^1 - U_i^1) - \ell_{i- \frac{1}{2}} (U_{i}^1 - U_{i-1}^1) ] + \Delta t [\ell_{i+ \frac{1}{2}} (\Psi_{i+1} - \Psi_i) - \ell_{i- \frac{1}{2}} (\Psi_{i} - \Psi_{i-1}) ] \\
=\left\{\begin{array}{lllll}
0,& i = 1, \ldots, N_{\alpha},\\
 \delta \frac{\Psi_{N_{\alpha}+2} - \Psi_{N_{\alpha}+1}}{h},& i = N_{\alpha}+1,\\
 \delta  \frac{\Psi_{i+1} - \Psi_i} {h} - \delta  \frac{\Psi_{i} - \Psi_{i-1}}{h}, & i = N_{\alpha}+2, \ldots, N_{\alpha}+N-1,\\
- \delta \frac{\Psi_{N_{\alpha}+N} - \Psi_{N_{\alpha}+N-1}}{h}, & i = N_{\alpha}+N , \\
0, & i = N_{\alpha}+N+1, \ldots, N_{\max}.
\end{array}\right. 
\end{array}
\end{equation}}}
\section{Stability estimates}  \label{Stability section}

\noindent Stability usually refers to numerical schemes producing bounded solution errors based on the approximation scheme being used. In this section, we derive the desired stability estimates for the numerical solution of the implicit scheme \eqref{DISCRETE-IMPLICIT}. 

At this point, we have to define the discrete norm we are using in order to obtain the different stability estimates as well as the convergence
of the implicit solution towards the weak solution defined by Definition~\ref{weak-solution}
\begin{defi}[\textbf{Discrete norms}] \label{discrete-norms}~\\
For $U \in \mathbb{L}^2$, we define $U_{\mathcal{T}}$ by the equation \eqref{UT}. The discrete $\mathbb{L}^2$ norm is defined by :
\begin{equation}\label{L2norm}
|| U_{\mathcal{T}} || = \bigg( \sum_{i=0}^{N_{\max}}  h_i U_i^2  \bigg)^{1/2}.
\end{equation}
For $U \in \mathbb{H}^1_L$, the discrete $\mathbb{H}^1$ semi-norm is defined by :
\begin{equation}\label{H1norm}
|| U_{\mathcal{T}} ||_{1, \mathcal{T}}  = \bigg( \sum_{i=0}^{N_{\max}} \ell_{i+ \frac{1}{2}} (U_{i+1} - U_i)^2  \bigg)^{1/2}
\end{equation}
where we have set the Dirichlet boundary condition as : $U_0 = U_{N_{\max}+1}=0$.
\end{defi}
Within the notations introduced above, we obtain the following discrete norms inequality, which play a crucial role in the proof of the stability estimates and the convergence of the numerical scheme.
\begin{Lemma}[\textbf{Discrete norms inequality}]~
\begin{enumerate}
\item[1.] For all $U \in \mathbb{L}^2$, we have:
\begin{equation}\label{L2estimate}
|| U_{\mathcal{T}} || \leq || U||_{L^2(0,L)}
\end{equation}
\item[2.] There exists a constant $K>0$, independent of the space discretization mesh size such that for all $U \in \mathbb{H}^1_L$, we have:
\begin{equation}\label{H1estimate}
|| U_{\mathcal{T}} ||_{1, \mathcal{T}}  \leq K || U_x||_{L^2(0,L)}
\end{equation}
\end{enumerate}
\end{Lemma}
\begin{proof}~
\begin{enumerate}
\item[1.] We have
\[
|| U_{\mathcal{T}} ||^2 = \sum_{i=0}^{N_{\max}}  h_i U_i^2 = \sum_{i=0}^{N_{\max}} \frac{1}{h_i} \left(\int_{K_i} U(x) dx\right)^2
\]
Since $U_0 = 0$, by Cauchy-Schwartz inequality we thus obtain
\[
\sum_{i=0}^{N_{\max}} \frac{1}{h_i} \left(\int_{K_i} U(x) dx\right)^2 \leq \sum_{i=0}^{N_{\max}} \frac{1}{h_i}\int_{K_i} 1 dx\int_{K_i} U^2(x) dx= \sum_{i=0}^{N_{\max}} \int_{K_i} U^2(x) dx
= \int_0^L U^2(x) dx = || U||^2_{L^2(0,L)}
\]

\item[2.] For the sake of simplicity in the following algebraic computations, we will suppose  the space discretision is uniform that is : 
\[
\forall i =1,N_{\max} \,,\, h_i = h_{i+\frac{1}{2}} = h \ .
\]
The non-uniform case is adapted by taking each term separately.
Let $U \in \mathbb{H}^1_L$. We have
\[
|| U_{\mathcal{T}} ||_{1, \mathcal{T}}^2  = \sum_{i=0}^{N_{\max}} \ell_{i+ \frac{1}{2}} (U_{i+1} - U_i)^2 =
\sum_{i=0}^{N_{\max}} \big(\ell_{i+ \frac{1}{2}} \times h_{i+\frac{1}{2}}\big)
\frac{\left(U_{i+1} - U_i\right)^2}{h_{i+\frac{1}{2}}} 
\]
From the inequalities \eqref{C1-C2} and \eqref{C2-C3} and from the definition of $\ell_{i+ \frac{1}{2}}$ we have : 
\[
|| U_{\mathcal{T}} ||_{1, \mathcal{T}}^2  \leq \max(C_1^2,C_2^2,C_3^2)
\sum_{i=0}^{N_{\max}} \frac{\left(U_{i+1} - U_i\right)^2}{h_{i+\frac{1}{2}}}  
\]
Let $i=1\,,N_{\max}$. We have:
\begin{align*}
(U_{i+1} - U_i)^2 &= \frac{1}{h^2}\left(\int_{K_{i+1}}U(x) dx  - \int_{K_{i}}U(x)dx\right)^2 \\
&=\frac{1}{h^2}\left(\int_{K_{i}}U(x+h) - U(x)dx\right)^2 \\
&=\frac{1}{h^2}\left(\int_{K_i} \int_x^{x+h} U_x(s)ds dx\right)^2 \; .
\end{align*}
Applying Cauchy-Schwartz inequality to get :
\[
(U_{i+1} - U_i)^2 \leq \int_{K_i} \int_x^{x+h} U^2_x(s)ds dx \; .
\] 
We thus obtain:
\begin{align*}
\sum_{i=0}^{N_{\max}} \frac{\left(U_{i+1} - U_i\right)^2}{h} \leq \sum_{i=0}^{N_{\max}}\frac{1}{h}\int_{K_i} \int_x^{x+h} U^2_x(s)ds dx 
\leq \frac{1}{h} \int_{0}^{L}\int_x^{x+h} U^2_x(s)ds dx
\end{align*}
But we have : 
\begin{align*}
\frac{1}{h} \int_{0}^{L}\int_x^{x+h} U^2_x(s)ds dx &= \frac{1}{h} \int_{0}^{L}\int_0^L \mathds{1}_{[x,x+h]}(s)U_x^2(s) ds dx\\
& =\frac{1}{h} \int_{0}^{L}\int_0^L \mathds{1}_{[s-h,s]}(x)U_x^2(s) ds dx \\
& =\frac{1}{h} \int_{0}^{L}\left(\int_0^L \mathds{1}_{[s-h,s]}(x) dx\right)U_x^2(s) ds \\
& =\int_{0}^{L}U_x^2(s) ds
\end{align*}
Thus we finally get :
\[
|| U_{\mathcal{T}} ||_{1, \mathcal{T}}^2  \leq \max(C_1^2,C_2^2,C_3^2) \int_{0}^{L}U_x^2(s) ds
\]
This inequality ends the proof.
\end{enumerate}
\end{proof}
\begin{Theorem}\label{Stability theorem for first term}
Assume that Hypothesis \eqref{H'} and Condition \eqref{TA} hold. Then, the discrete $\mathbb{L}^2$ and $\mathbb{H}^1_L$ norms of the numerical solution of the first time step $U_{\mathcal{T}}^1$ of the implicit numerical scheme \eqref{DISCRETE-IMPLICIT-n=0} are bounded. Moreover, the approximation of the first time derivation $\partial^1 U_{\mathcal{T}}^1$ is also bounded in $\mathbb{L}^2$. More precisely, there exists a constant $K$ independent of the time and space
 discretization parameters such that the following estimate is fulfilled : 
\begin{equation}\label{Initial-Estimate}
|| U_{\mathcal{T}}^1|| +  || \partial^1 U_{\mathcal{T}}^1|| + || U_{\mathcal{T}}^1||_{1, \mathcal{T}} \leq K
\end{equation}
\end{Theorem}
\noindent The proof of Theorem \ref{Stability theorem for first term} is divided into two Lemmas.
\begin{Lemma} \label{Lem1}
Under Hypothesis \eqref{H'} and Condition \eqref{TA}, the numerical solution of the first time step $U_{\mathcal{T}}^1$ of the implicit numerical scheme \eqref{DISCRETE-IMPLICIT-n=0} is bounded in $\mathbb{L}^2$. More precisely, there exists a constant $K_1$ independent of the time and space discretization parameters such that $\left\Vert U_{\mathcal{T}}^1 \right\Vert \leq K_1$.  
\end{Lemma}
\begin{proof}
We split the proof of Lemma \ref{Lem1} into two steps. \\[0.1in]
\textbf{Step 1.}
Multiply the left and right hand-side of Equation \eqref{DISCRETE-IMPLICIT-n=0} by $2 \Delta t^2 U_i^1$ and sum over $i$ to obtain
{\small{\begin{equation*}
\begin{array}{lll}
4 \sum_{i=1}^{N_{\max}}  h_i (U_{i} ^{1})^2 - 4  \sum_{i=1}^{N_{\max}} h_i U_i^0 U_i^1 - 4 \Delta t  \sum_{i=1}^{N_{\max}} h_{i} \Psi_i U_i^1 - 2 \Delta t^2 \sum_{i=1}^{N_{\max}} \ell_{i+{\frac{1}{2}}}(U_{i+1}^1 - U_i^1)U_i^1 + 2 \Delta t^2\sum_{i=1}^{N_{\max}} \ell_{i-{\frac{1}{2}}}(U_{i}^1 - U_{i-1}^1)U_i^1 \\
+ 2 (\Delta t)^3 \sum_{i=1}^{N_{\max}} \ell_{i+{\frac{1}{2}}}(\Psi_{i+1} - \Psi_i)U_i^1 - 2 (\Delta t)^3\sum_{i=1}^{N_{\max}}  \ell_{i-{\frac{1}{2}}}(\Psi_{i} - \Psi_{i-1})U_i^1 \\
=\left\{\begin{array}{lllll}
0,& i = 1, \ldots, N_{\alpha},\\
 2 \Delta t^2 \delta \frac{(\Psi_{N_{\alpha}+2} - \Psi_{N_{\alpha}+1})U_{N_{\alpha}+1}^1}{h},& i = N_{\alpha}+1,\\
  2 \Delta t^2 \delta \sum_{i=N_{\alpha}+2}^{N_{\alpha}+N-1} \frac{(\Psi_{i+1} - \Psi_i)U_i^1} {h}  -  2 \Delta t^2 \delta  \sum_{i=N_{\alpha}+2}^{N_{\alpha}+N-1} \frac{(\Psi_{i} - \Psi_{i-1})U_i^1}{h} , & i = N_{\alpha}+2, \ldots, N_{\alpha}+N-1,\\
-  2 \Delta t^2 \delta \frac{(\Psi_{N_{\alpha}+N} - \Psi_{N_{\alpha}+N-1})U_{N_{\alpha}+N}^1}{h} , & i = N_{\alpha}+N , \\
0, & i = N_{\alpha}+N+1, \ldots, N_{\max}.
\end{array}\right. 
\end{array}
\end{equation*}}}
Substituting \eqref{initialSOLUTION} in the left hand-side of the above equation and applying the translation of index $i$ for the fifth and seventh term of the left hand-side and for the second term$_{\vert_{i = N_{\alpha}+2, \ldots, N_{\alpha}+N-1}}$ of the right hand-side, we obtain
{\small{\begin{equation*}
\begin{array}{lll}
4 \sum_{i=1}^{N_{\max}}  h_i (U_{i} ^{1})^2 - 4  \sum_{i=1}^{N_{\max}} h_i \Phi_i U_i^1 - 4 \Delta t  \sum_{i=1}^{N_{\max}} h_{i} \Psi_i U_i^1 - 2 \Delta t^2 \sum_{i=1}^{N_{\max}} \ell_{i+{\frac{1}{2}}}(U_{i+1}^1 - U_i^1)U_i^1 + 2 \Delta t^2\sum_{i=0}^{N_{\max}-1} \ell_{i+{\frac{1}{2}}}(U_{i+1}^1 - U_{i}^1)U_{i+1}^1 \\
+ 2 (\Delta t)^3 \sum_{i=1}^{N_{\max}} \ell_{i+{\frac{1}{2}}}(\Psi_{i+1} - \Psi_i)U_i^1 - 2 (\Delta t)^3\sum_{i=0}^{N_{\max}-1}  \ell_{i+{\frac{1}{2}}}(\Psi_{i+1} - \Psi_{i})U_{i+1}^1 \\
=  2 \Delta t^2 \delta \frac{(\Psi_{N_{\alpha}+2} - \Psi_{N_{\alpha}+1})U_{N_{\alpha}+1}^1}{h} + 2 \Delta t^2 \delta \sum_{i=N_{\alpha}+2}^{N_{\alpha}+N-1} \frac{(\Psi_{i+1} - \Psi_i)U_i^1} {h}  -  2 \Delta t^2 \delta  \sum_{i=N_{\alpha}+1}^{N_{\alpha}+N-2} \frac{(\Psi_{i+1} - \Psi_{i})U_{i+1}^1}{h} \\
-  2 \Delta t^2 \delta \frac{(\Psi_{N_{\alpha}+N} - \Psi_{N_{\alpha}+N-1})U_{N_{\alpha}+N}^1}{h}.
\end{array}
\end{equation*}}}
Using the fact that $U_0 = U_{N_{\max}+1} =0$, we get
{\small{\begin{equation} \label{RR(1.3)}
\begin{array}{lll}
4 \sum_{i=1}^{N_{\max}}h_i  (U_{i} ^{1})^2 
& =  4  \sum_{i=1}^{N_{\max}}h_i \Phi_i U_i^1 + 4 \Delta t  \sum_{i=1}^{N_{\max}} h_{i} \Psi_i U_i^1 - 2 \Delta t^2 \sum_{i=0}^{N_{\max}} \ell_{i+{\frac{1}{2}}}(U_{i+1}^1 - U_i^1)^2  + 2 (\Delta t)^3 \sum_{i=0}^{N_{\max}} \ell_{i+{\frac{1}{2}}}(\Psi_{i+1} - \Psi_i)(U_{i+1}^1-U_i^1) \\
& - 2 \Delta t^2 \delta \sum_{i=N_{\alpha}+1}^{N_{\alpha}+N-1} \frac{(\Psi_{i+1} - \Psi_i)(U_{i+1}^1-U_i^1)} {h} 
\end{array}
\end{equation}}}
\noindent \textbf{Step 2. Estimation of the terms of the right hand-side of \eqref{RR(1.3)}}: \\[0.1in]
\textbf{Estimation of the first term.} Applying Young's inequality, we obtain
\begin{equation*}
\begin{array}{lll}
4  \sum_{i= 1}^{N_{\max}} h_{i} \Phi_i U_i^1 
\leq 8  \sum_{i= 1}^{N_{\max}}  h_{i} ( \Phi_i)^2 + \frac{1}{2}  \sum_{i= 1}^{N_{\max}}  h_{i} ( U^1_i)^2 .
\end{array}
\end{equation*} 
\textbf{Estimation of the second term.} We apply similarly Young's inequality to get
\begin{equation*}
\begin{array}{lll}
4 \Delta t\sum_{i= 1}^{N_{\max}}  h_{i}  \Psi_i U_i^1
\leq 8  \Delta t^2  \sum_{i= 1}^{N_{\max}} h_{i} ( \Psi_i)^2 + \frac{1}{2}  \sum_{i= 1}^{N_{\max}} h_{i} ( U^1_i)^2.
\end{array}
\end{equation*} 
\textbf{Estimation of the fourth and fifth terms.} Like above, we apply Young's inequality to get
\begin{equation*}
\begin{array}{llll}
2 (\Delta t)^3 \sum_{i=0}^{N_{\max}} \ell_{i+{\frac{1}{2}}}(\Psi_{i+1} - \Psi_i)(U_{i+1}^1-U_i^1)- 2 \Delta t^2 \delta \sum_{i=N_{\alpha}+1}^{N_{\alpha}+N-1} \frac{(\Psi_{i+1} - \Psi_i)(U_{i+1}^1-U_i^1)} {h} \\
 \leq (\Delta t)^4 \sum_{i=0}^{N_{\max}} \ell_{i+ \frac{1}{2}}(\Psi_{i+1} - \Psi_i)^2 +  \Delta t^2 \sum_{i=0}^{N_{\max}} \ell_{i+{\frac{1}{2}}}(U^1_{i+1} - U^1_i)^2 \\
   +  \frac{\delta^2 \Delta t^2}{C_2^2} \sum_{i=N_{\alpha}+1}^{N_{\alpha}+N-1} \frac{ (\Psi_{i+1} - \Psi_i)^2}{h} 
  +  \Delta t^2 C_2^2 \sum_{i=N_{\alpha}+1}^{N_{\alpha}+N-1} \frac{(U^1_{i+1} - U^1_i)^2}{h}.
\end{array}
\end{equation*}  
Taking into consideration that $\frac{1}{h} = \frac{\ell_{i+ \frac{1}{2}}}{C_2^2}; \ i=N_{\alpha}+1, \ldots, N_{\alpha}+N-1$, we obtain 
\begin{equation*}
\begin{array}{llll}
2 (\Delta t)^3 \sum_{i=0}^{N_{\max}} \ell_{i+{\frac{1}{2}}}(\Psi_{i+1} - \Psi_i)(U_{i+1}^1-U_i^1)- 2 \Delta t^2 \delta \sum_{i=N_{\alpha}+1}^{N_{\alpha}+N-1} \frac{(\Psi_{i+1} - \Psi_i)(U_{i+1}^1-U_i^1)} {h} \\  
   \leq \Delta t^2 \bigg(  \Delta t^2 + \frac{\delta^2}{C_2^4} \bigg) \sum_{i=0}^{N_{\max}} \ell_{i+ \frac{1}{2}}(\Psi_{i+1} - \Psi_i)^2 + 2 \Delta t^2 \sum_{i=0}^{N_{\max}} \ell_{i+{\frac{1}{2}}}(U^1_{i+1} - U^1_i)^2 .
\end{array}
\end{equation*}
$\diamond$ Applying the estimates to \eqref{RR(1.3)}, we get
\begin{equation}\label{RR(1.4)}
\begin{array}{lll}
3  \sum_{i=1}^{N_{\max}} h_i(U_{i} ^{1})^2 
\leq 8   \sum_{i= 1}^{N_{\max}} h_i( \Phi_i)^2 
+ 8  \Delta t^2\sum_{i= 1}^{N_{\max}} h_i( \Psi_i)^2  
 + \Delta t^2 \bigg(  \Delta t^2 + \frac{\delta^2}{C_2^4} \bigg) \sum_{i=0}^{N_{\max}} \ell_{i + \frac{1}{2}}(\Psi_{i+1} - \Psi_i)^2
\end{array}
\end{equation}
Hence, we deduce the following estimation
\begin{equation*}
|| U_{\mathcal{T}}^1 ||^2  \leq \frac{8}{3} || \Phi_{\mathcal{T}} ||^2 + \frac{8 \Delta t^2 }{3} || \Psi_{\mathcal{T}} || ^2 + \frac{\Delta t^2}{3} \bigg(  \Delta t^2 + \frac{\delta^2}{C_2^4} \bigg) || \Psi_{\mathcal{T}} ||_{1, \mathcal{T}} ^2.
\end{equation*}
Using estimates \eqref{L2estimate} and \eqref{H1estimate} we get :
\begin{equation*}
|| U_{\mathcal{T}}^1 ||^2  \leq \frac{8}{3} || \Phi||^2 + \frac{8 \Delta t^2 }{3} || \Psi || ^2 + K \frac{\Delta t^2}{3} \bigg(  \Delta t^2 + \frac{\delta^2}{C_2^4} \bigg) ||\Psi_x||^2.
\end{equation*}
Due to Hypothesis \eqref{H'} and time assumption \eqref{TA}, there exists $K_{1}$ which depends only on the initial data of the problem and on the constant $\tau_0$ and independent of the discretization parameters such that :
\begin{equation}\label{U in L2}
|| U_{\mathcal{T}} ^1||^2  \leq K_{1} \ .
\end{equation}
The proof of Lemma \ref{Lem1} is thus complete.
\end{proof}
\begin{Lemma}\label{Lem2}
Under Hypothesis \eqref{H'} and Condition \eqref{TA}, the numerical solution of the first time step $U_{\mathcal{T}}^1$ of the implicit numerical scheme \eqref{DISCRETE-IMPLICIT-n=0} is bounded in $\mathbb{H}^1_L$ and the approximation of the first time derivation is bounded in $\mathbb{L}^2$; more precisely there exists 
a constant $K_{2}$ independent of the time and space discretization parameters such that 
\[
||\partial^1 U^1 _{\mathcal{T}} || ^2 + || U_{\mathcal{T}}^1 ||^2_{1, \mathcal{T}} \leq K_{2}
\]
\end{Lemma}
\begin{proof}
The proof of Lemma \ref{Lem2} is divided into two steps.\\[0.1in]
\textbf{Step 1.} Multiply Equation \eqref{DISCRETE-IMPLICIT-n=0} by $(U_i^1 - U_i^0)$ and sum over $i$ to obtain 
{\small{\begin{equation*} 
\begin{array}{lll}
2\sum_{i=1}^{N_{\max}} h_i \frac{(U_i^1 - U_i^0)^2}{\Delta t^2} - \frac{2}{\Delta t} \sum_{i=1}^{N_{\max}} h_i \Psi_i (U_i^1 - U_i^0) - \sum_{i=1}^{N_{\max}} \ell_{i+ \frac{1}{2}} (U_{i+1}^1 - U_i^1)(U_i^1 - U_i^0) + \sum_{i=1}^{N_{\max}} \ell_{i- \frac{1}{2}} (U_{i}^1 - U_{i-1}^1)(U_i^1 - U_i^0) \\
 + \Delta t \sum_{i=1}^{N_{\max}} \ell_{i+ \frac{1}{2}} (\Psi_{i+1} - \Psi_i)(U_i^1 - U_i^0) - \Delta t \sum_{i=1}^{N_{\max}} \ell_{i- \frac{1}{2}} (\Psi_{i} - \Psi_{i-1})(U_i^1 - U_i^0)  \\
= \delta \frac{(\Psi_{N_{\alpha}+2} - \Psi_{N_{\alpha}+1})(U_{N_{\alpha}+1}^1 - U_{N_{\alpha}+1}^0)}{h} + \delta \sum_{i=N_{\alpha}+2}^{N_{\alpha}+N-1} \frac{(\Psi_{i+1} - \Psi_i)(U_i^1 - U_i^0)} {h} \\
- \delta  \sum_{i=N_{\alpha}+2}^{N_{\alpha}+N-1} \frac{(\Psi_{i} - \Psi_{i-1})(U_i^1 - U_i^0)}{h} - \delta \frac{(\Psi_{N_{\alpha}+N} - \Psi_{N_{\alpha}+N-1})(U_{N_{\alpha}+N}^1 - U_{N_{\alpha}+N}^0)}{h} .
\end{array}
\end{equation*}}}
By translation of index $i$ for the fourth and sixth terms of the left hand-side of the above equation and the third term of the right hand-side of the above equation, we get
{\small{\begin{equation*} 
\begin{array}{lll}
2\sum_{i=1}^{N_{\max}} h_i \frac{(U_i^1 - U_i^0)^2}{\Delta t^2} - \frac{2}{\Delta t} \sum_{i=1}^{N_{\max}} h_i \Psi_i (U_i^1 - U_i^0) - \sum_{i=1}^{N_{\max}} \ell_{i+ \frac{1}{2}} (U_{i+1}^1 - U_i^1)(U_i^1 - U_i^0) + \sum_{i=0}^{N_{\max}-1} \ell_{i+ \frac{1}{2}} (U_{i+1}^1 - U_{i}^1)(U_{i+1}^1 - U_{i+1}^0) \\
 + \Delta t \sum_{i=1}^{N_{\max}} \ell_{i+ \frac{1}{2}} (\Psi_{i+1} - \Psi_i)(U_i^1 - U_i^0) - \Delta t \sum_{i=0}^{N_{\max}-1} \ell_{i+ \frac{1}{2}} (\Psi_{i+1} - \Psi_{i})(U_{i+1}^1 - U_{i+1}^0)  \\
= \delta \frac{(\Psi_{N_{\alpha}+2} - \Psi_{N_{\alpha}+1})(U_{N_{\alpha}+1}^1 - U_{N_{\alpha}+1}^0)}{h} + \delta \sum_{i=N_{\alpha}+2}^{N_{\alpha}+N-1} \frac{(\Psi_{i+1} - \Psi_i)(U_i^1 - U_i^0)} {h} \\
- \delta  \sum_{i=N_{\alpha}+1}^{N_{\alpha}+N-2} \frac{(\Psi_{i+1} - \Psi_{i})(U_{i+1}^1 - U_{i+1}^0)}{h} - \delta \frac{(\Psi_{N_{\alpha}+N} - \Psi_{N_{\alpha}+N-1})(U_{N_{\alpha}+N}^1 - U_{N_{\alpha}+N}^0)}{h} .
\end{array}
\end{equation*}}}
Using the fact that $U_{0}=U_{N_{\max}+1}=0$, it follows
\begin{equation*} 
\begin{array}{lll}
2\sum_{i=1}^{N_{\max}} h_i \frac{(U_i^1 - U_i^0)^2}{\Delta t^2} - \frac{2}{\Delta t} \sum_{i=1}^{N_{\max}} h_i \Psi_i (U_i^1 - U_i^0) + \sum_{i=0}^{N_{\max}} \ell_{i+ \frac{1}{2}} (U_{i+1}^1 - U_{i}^1)[(U_{i+1}^1 - U_{i}^1) - (U_{i+1}^0 - U_i^0)] \\
 - \Delta t \sum_{i=0}^{N_{\max}} \ell_{i+ \frac{1}{2}} (\Psi_{i+1} - \Psi_i)[(U_{i+1}^1 - U_{i}^1) - (U_{i+1}^0 - U_i^0)] \\
= - \delta  \sum_{i=N_{\alpha}+1}^{N_{\alpha}+N-1} \frac{(\Psi_{i+1} - \Psi_{i})[(U_{i+1}^1 - U_{i}^1) - (U_{i+1}^0 - U_i^0)]}{h},
\end{array}
\end{equation*}
to obtain finally
\begin{equation}\label{RRR(1.1)}
\begin{array}{lll}
2\sum_{i=1}^{N_{\max}} h_i \frac{(U_i^1 - U_i^0)^2}{\Delta t^2} + \sum_{i=0}^{N_{\max}} \ell_{i+ \frac{1}{2}} (U_{i+1}^1 - U_{i}^1)^2  \\
= \frac{2}{\Delta t} \sum_{i=1}^{N_{\max}} h_i \Psi_i (U_i^1 - U_i^0) + \sum_{i=0}^{N_{\max}} \ell_{i+ \frac{1}{2}} (U_{i+1}^1 - U_{i}^1)(U_{i+1}^0 - U_i^0) + \Delta t \sum_{i=0}^{N_{\max}} \ell_{i+ \frac{1}{2}} (\Psi_{i+1} - \Psi_i)[(U_{i+1}^1 - U_{i}^1) - (U_{i+1}^0 - U_i^0)] \\
- \delta  \sum_{i=N_{\alpha}+1}^{N_{\alpha}+N-1} \frac{(\Psi_{i+1} - \Psi_{i})[(U_{i+1}^1 - U_{i}^1) - (U_{i+1}^0 - U_i^0)]}{h}.
\end{array}
\end{equation}
\textbf{Step 2. Estimation of the terms of the right hand-side of \eqref{RRR(1.1)}}: \\[0.1in]
\textbf{Estimation of the first term.} Apply Young's inequality to obtain
\begin{equation*}
\begin{array}{lll}
\frac{2}{\Delta t} \sum_{i=1}^{N_{\max}} h_i \Psi_i (U_i^1 - U_i^0)
\leq  \sum_{i=1}^{N_{\max}} h_i \frac{(U_i^1 - U_i^0)^2}{\Delta t^2} +  \sum_{i=1}^{N_{\max}} h_i (\Psi_i)^2 .
\end{array}
\end{equation*}
\textbf{Estimation of the second term.} Also, apply Young's inequality to get
\begin{equation*}
\begin{array}{lll}
\sum_{i=0}^{N_{\max}} \ell_{i+ \frac{1}{2}} (U_{i+1}^1 - U_{i}^1)(U_{i+1}^0 - U_i^0)
\leq \frac{1}{6} \sum_{i=0}^{N_{\max}} \ell_{i+ \frac{1}{2}} (U_{i+1}^1 - U_i^1)^2 + \frac{3 }{2} \sum_{i=0}^{N_{\max}} \ell_{i+ \frac{1}{2}} (U_{i+1}^0 - U_i^0)^2.
\end{array}
\end{equation*}
\textbf{Estimation of the last two terms.} Setting $\varepsilon>0$, we apply Young's inequality as well to obtain
\begin{equation*}
\begin{array}{lll}
\Delta t \sum_{i=0}^{N_{\max}} \ell_{i+ \frac{1}{2}} (\Psi_{i+1} - \Psi_i)[(U_{i+1}^1 - U_{i}^1) - (U_{i+1}^0 - U_i^0)] - \delta  \sum_{i=N_{\alpha}+1}^{N_{\alpha}+N-1} \frac{(\Psi_{i+1} - \Psi_{i})[(U_{i+1}^1 - U_{i}^1) - (U_{i+1}^0 - U_i^0)]}{h} \\
\leq  (\Delta  t)^2 \sum_{i=0}^{N_{\max}} \ell_{i+ \frac{1}{2}} (\Psi_{i+1} - \Psi_i)^2 +   \frac{1}{2} \sum_{i=0}^{N_{\max}} \ell_{i+ \frac{1}{2}} (U^1_{i+1} - U^1_i)^2+ \frac{1}{2} \sum_{i=0}^{N_{\max}} \ell_{i+ \frac{1}{2}} (U^0_{i+1} - U^0_i)^2 \\
+ \frac{\delta^2 \varepsilon}{ C_2^4}  \sum_{i=N_{\alpha}+1}^{N_{\alpha}+N-1} \ell_{i+ \frac{1}{2}} (\Psi_{i+1} - \Psi_{i})^2 + \frac{1}{2 \varepsilon}   \sum_{i=N_{\alpha}+1}^{N_{\alpha}+N-1} \ell_{i+ \frac{1}{2}} (U^1_{i+1} - U^1_i)^2 + \frac{ 1}{2 \varepsilon}   \sum_{i=N_{\alpha}+1}^{N_{\alpha}+N-1} \ell_{i+ \frac{1}{2}} (U^0_{i+1} - U^0_i)^2 \\
\leq \bigg[ \Delta t^2 + \frac{\delta^2 \varepsilon}{ C_2^4}\bigg] \sum_{i=0}^{N_{\max}} \ell_{i+ \frac{1}{2}} (\Psi_{i+1} - \Psi_i)^2 + \frac{1}{2} (1+ \varepsilon^{-1}) \sum_{i=0}^{N_{\max}} \ell_{i+ \frac{1}{2}} (U^1_{i+1} - U^1_i)^2 +  \frac{1}{2} (1+ \varepsilon^{-1}) \sum_{i=0}^{N_{\max}} \ell_{i+ \frac{1}{2}} (U^0_{i+1} - U^0_i)^2,
\end{array}
\end{equation*}
 \textcolor{black}{since} $\frac{1}{h} = \frac{\ell_{i+ \frac{1}{2}}}{C_2^2}.$ \\[0.1in]
Inserting the above estimates in \eqref{RRR(1.1)}, we deduce 
\begin{equation}\label{RRR(1.2)}
\begin{array}{lll}
 \sum_{i=1}^{N_{\max}} h_i \frac{(U_i^1 - U_i^0)^2}{\Delta t^2} + \bigg( \frac{1}{3} - \frac{1}{2 \varepsilon}  \bigg) \sum_{i=0}^{N_{\max}} \ell_{i+ \frac{1}{2}} (U_{i+1}^1 - U_{i}^1)^2 \\
\leq \sum_{i=1}^{N_{\max}} h_i (\Psi_i)^2 + \bigg( 2+ \frac{1}{2 \varepsilon}  \bigg) \sum_{i=0}^{N_{\max}} \ell_{i+ \frac{1}{2}} (U_{i+1}^0 - U_{i}^0)^2 + \bigg[ \Delta t^2 + \frac{\delta^2 \varepsilon}{ C_2^4}\bigg] \sum_{i=0}^{N_{\max}} \ell_{i+ \frac{1}{2}} (\Psi_{i+1} - \Psi_i)^2.
\end{array}
\end{equation}
Setting $\varepsilon = 2$, we deduce that 
\begin{equation*}
||\partial^1 U^1 _{\mathcal{T}} || ^2 + || U_{\mathcal{T}}^1 ||^2_{1, \mathcal{T}} \leq 12 || \Psi_{\mathcal{T}}||^2 + 27  || \Phi _{\mathcal{T}} ||^2_{1, \mathcal{T}} + 12 \bigg[ \Delta t^2 + \frac{2 \delta^2 }{ C_2^4}\bigg]   || \Psi _{\mathcal{T}} ||^2_{1, \mathcal{T}}
\end{equation*}

We conclude as done in the previous lemma : using estimates \eqref{L2estimate} and \eqref{H1estimate}, there exists $K_{2}$ which depends only  
on the initial data of the problem and on $\tau_0$ 
\begin{equation} \label{U in H1}
||\partial^1 U^1 _{\mathcal{T}} || ^2 + || U_{\mathcal{T}}^1 ||^2_{1, \mathcal{T}} \leq K_{2} \ .
\end{equation}
Therefore, the proof of Lemma \ref{Lem2} is thus complete.
\end{proof}
\begin{proof}[\textbf{Proof of Theorem \ref{Stability theorem for first term}}]  Combining estimates \eqref{U in L2} and \eqref{U in H1} concludes the boundedness of the first time step $U^1_{\mathcal{T}}$ in spaces $\mathbb{H}^1_L$ and $\mathbb{L}^2$ as well as the boundedness of its first time derivation approximation  in $\mathbb{L}^2$, that is there exists a constant $K$ independent of the time and space
 discretization parameters such that the following estimation is fulfilled : 
\begin{equation*}
|| U_{\mathcal{T}}^1|| +  || \partial^1 U_{\mathcal{T}}^1|| + || U_{\mathcal{T}}^1||_{1, \mathcal{T}} \leq K
\end{equation*}
\end{proof}
Now, we will move to the stability estimates of the discretized numerical scheme \eqref{DISCRETE-IMPLICIT} for $n\geq 1$. To this end, we aim to prove the following theorem: 
\begin{Theorem} \label{Theorem Stability n>1}
Assume that Hypothesis \eqref{H'} and Condition \eqref{TA} hold. Then, the numerical solution $U_{\mathcal{T},\Delta t}$ of the numerical scheme \eqref{DISCRETE-IMPLICIT} is bounded in $L^{\infty}([0,T]; (0,L)) \cap L^{\infty}([0,T]; \mathbb{H}^1_L)$ and its time derivative approximation is bounded in $L^{\infty}([0,T]; \mathbb{L}^2)$ that is there exists a constant $K$ independent of the time and space discretization parameters such that the following estimation is fulfilled :
\begin{equation} \label{stability-estimate}
\forall n= 1, \cdots, \mathcal{N} -1 \,,\, || \partial^1 U_{\mathcal{T}}^n ||^2 + || U_{\mathcal{T}}^n||^2 + || U_{\mathcal{T}}^n||^2_{1, \mathcal{T}} \leq K 
\end{equation}
\end{Theorem}
\begin{proof}
For the proof of Theorem \ref{Theorem Stability n>1}, we multiply \eqref{DISCRETE-IMPLICIT} by $U_{i}^{n+1} - U_{i}^{n-1}$ for $i = 1 , \cdots, N_{\max}$. Applying exactly the same steps in calculating the discrete energy in Subsection \ref{S(IE)}; {\it{i.e.}} the steps before writing the terms as discrete derivatives with respect to time, we can use the results obtained to get directly
\begin{equation*}
\begin{array}{lll}
\sum_{i=1}^{N_{\max}} h_i \bigg(\frac{U_{i}^{n+1}-U_{i}^n}{\Delta t}\bigg)^2  + \frac{1}{2}\sum_{i=0}^{N_{\max}}\ell_{i+{\frac{1}{2}}} (U_{i+1}^{n+1} - U_i^{n+1})^2  

\\
 = \sum_{i=1}^{N_{\max}} h_i \bigg(\frac{U_{i}^{n}-U_{i}^{n-1}}{\Delta t}\bigg)^{2} + \frac{1}{2}\sum_{i=0}^{N_{\max}}\ell_{i+{\frac{1}{2}}} (U_{i+1}^{n-1} - U_i^{n-1})^2 - \frac{\delta \; \textcolor{black}{h} }{2 \Delta t} \sum_{i=N_{\alpha}+1}^{N_{\alpha}+N -1}\bigg( \frac{U^{n+1}_{i +1 } - U^{n+1}_{i }}{h} - \frac{U^{n-1}_{i +1 } - U^{n-1}_{i }}{h} \bigg)^2.
\end{array}
\end{equation*}
Using the fact that $- \frac{\delta \; \textcolor{black}{h} }{2 \Delta t} \sum_{i=N_{\alpha}+1}^{N_{\alpha}+N -1}\bigg( \frac{U^{n+1}_{i +1 } - U^{n+1}_{i }}{h} - \frac{U^{n-1}_{i +1 } - U^{n-1}_{i }}{h} \bigg)^2 \leq 0$,
we obtain directly
\begin{equation}\label{RRR(n>1)}
\begin{array}{lll}

\sum_{i=1}^{N_{\max}} h_i \bigg(\frac{U_{i}^{n+1}-U_{i}^n}{\Delta t}\bigg)^2  + \frac{1}{2}\sum_{i=0}^{N_{\max}}\ell_{i+{\frac{1}{2}}} (U_{i+1}^{n+1} - U_i^{n+1})^2

\leq \sum_{i=1}^{N_{\max}} h_i \bigg(\frac{U_{i}^{n}-U_{i}^{n-1}}{\Delta t}\bigg)^{2} + \frac{1}{2}\sum_{i=0}^{N_{\max}}\ell_{i+{\frac{1}{2}}} (U_{i+1}^{n-1} - U_i^{n-1})^2.
\end{array}
\end{equation}
For $n=1$, we get from \eqref{RRR(n>1)}
\begin{equation}\label{bdd  in time and space for n=1}
\begin{array}{lll}
2 || \partial^1 U_{\mathcal{T}}^2 ||^2 +  || U_{\mathcal{T}}^2||^2_{1, \mathcal{T}} \leq 2  || \partial^1 U_{\mathcal{T}}^1 ||^2 +  || \Phi_{\mathcal{T}}||^2_{1, \mathcal{T}} .
\end{array}
\end{equation}
From the Discrete Poincare inequality \cite[Lemma 9.1]{ERG2000}, we have $|| U_{\mathcal{T}}^2||^2 \leq c_0  || U_{\mathcal{T}}^2||^2_{1, \mathcal{T}}$, where
the constant $c_0$ is independent of the discretization parameters.. Thus, we obtain from \eqref{bdd  in time and space for n=1}
\begin{equation*}
\begin{array}{lll}
|| U_{\mathcal{T}}^2||^2 \leq 2 c_0  || \partial^1 U_{\mathcal{T}}^1 ||^2 +   c_0  || \Phi_{\mathcal{T}}||^2_{1, \mathcal{T}} .
\end{array}
\end{equation*}
Again, we argue as in the end of the proof of the two preceding lemmas : using estimates \eqref{L2estimate} and \eqref{H1estimate}, there exists $K$ which depends only  on the initial data of the problem and on $\tau_0$  such that
\begin{equation*}
|| U_{\mathcal{T}}^2||^2 + || \partial^1 U_{\mathcal{T}}^2 ||^2 +  || U_{\mathcal{T}}^2||^2_{1, \mathcal{T}} \leq K 
\end{equation*}
We argue similarly for $n = 2 , \cdots, \mathcal{N}-1$, to finally obtain: 
\begin{equation*}
 || \partial^1 U_{\mathcal{T}}^n ||^2 + || U_{\mathcal{T}}^n||^2 + || U_{\mathcal{T}}^n||^2_{1, \mathcal{T}} \leq K; \quad n= 1, \cdots, \mathcal{N} -1
\end{equation*}
The proof of Theorem \ref{Theorem Stability n>1} is thus complete.
\end{proof}


\section{Convergence} \label{Convergence section}
\noindent In this section, we want to prove the convergence of the numerical solution of Problem \eqref{DISCRETE-IMPLICIT}.
The strategy in this section is based on the stability estimate \eqref{stability-estimate} obtained in Section \ref{Stability section}, and passing through the limit, we obtain the convergence to the weak solution. 

To this end, we will first write the continuous weak formulation problem \eqref{continuous weak variational problem} in its discrete form.

As we will use a density argument, we set function $P \in C^{\infty} ([0,T]; C^{\infty}_0 (0,L))$ such that $P(x,T)=0$ and let size($\mathcal{T}$) be small enough. Define $P_{i}^n = P(x_i,t_n)$ for $i = 0 , \cdots, N_{\max}$. 
We recall that the reconstructed solution is defined by  $U_{\mathcal{T},\Delta t}$ constant over the control volumes $K_i$ at time $t \in [t_n,t_{n+1})$ as
defined by equation \eqref{UTN}. 

Now, we are ready to get the discrete variational problem, we will pass through several steps. \\[0.1 in]
\textbf{Step 1.} Multiply \eqref{DISCRETE-IMPLICIT-n=0} by $\frac{\Delta t}{2}  P_{i}^0$ and sum over $i$ to obtain
{\small{\begin{equation*}
\begin{array}{lll}
 \Delta t \sum_{i=1}^{N_{\max}} h_i \frac{U_i^1 - U_i^0}{\Delta t^2} P_{i}^0 - \sum_{i=1}^{N_{\max}} h_i \Psi_i P_{i}^0 - \frac{\Delta t}{2} \sum_{i=1}^{N_{\max}} \ell_{i+ \frac{1}{2}} (U_{i+1}^1 - U_i^1)P_{i}^0 + \frac{\Delta t}{2} \sum_{i=1}^{N_{\max}} \ell_{i- \frac{1}{2}} (U_{i}^1 - U_{i-1}^1) P_{i}^0 \\
 + \frac{\Delta t^2}{2} \sum_{i=1}^{N_{\max}}  \ell_{i+ \frac{1}{2}} (\Psi_{i+1} - \Psi_i)P_{i}^0 -  \frac{\Delta t^2}{2} \sum_{i=1}^{N_{\max}}\ell_{i- \frac{1}{2}} (\Psi_{i} - \Psi_{i-1})P_{i}^0 \\
=   \frac{\delta \Delta t}{2} \frac{\Psi_{N_{\alpha}+2} - \Psi_{N_{\alpha}+1}}{h} P_{N_{\alpha}+1}^0  +  \frac{\delta \Delta t}{2} \sum_{i= N_{\alpha}+2}^{N_{\alpha}+N-1} \frac{\Psi_{i+1} - \Psi_i} {h} P_{i}^0 -\frac{\delta \Delta t}{2} \sum_{i= N_{\alpha}+2}^{N_{\alpha}+N-1} \frac{\Psi_{i} - \Psi_{i-1}}{h} P_{i}^0  - \frac{\delta \Delta t}{2} \frac{\Psi_{N_{\alpha}+N} - \Psi_{N_{\alpha}+N-1}}{h} P_{N_{\alpha}+N}^0.
\end{array}
\end{equation*}
}}
By translation of index $i$ for the fourth and sixth terms of the left hand-side and the third term of the right hand-side of the above equation, and using the fact that $P_0 = P_{N_{\max}+1}=0$, a direct calculation gives
{\small{\begin{equation*}
\begin{array}{lll}
  \Delta t \sum_{i=1}^{N_{\max}} h_i \frac{U_i^1 - U_i^0}{\Delta t^2} P_{i}^0 - \sum_{i=1}^{N_{\max}} h_i \Psi_i P_{i}^0 - \frac{\Delta t}{2} \sum_{i=1}^{N_{\max}} \ell_{i+ \frac{1}{2}} (U_{i+1}^1 - U_i^1)P_{i}^0 + \frac{\Delta t}{2} \sum_{i=0}^{N_{\max}-1} \ell_{i+ \frac{1}{2}} (U_{i+1}^1 - U_{i}^1) P_{i+1}^0 \\
 + \frac{\Delta t^2}{2} \sum_{i=1}^{N_{\max}}  \ell_{i+ \frac{1}{2}} (\Psi_{i+1} - \Psi_i)P_{i}^0 -  \frac{\Delta t^2}{2} \sum_{i=0}^{N_{\max}-1}\ell_{i+ \frac{1}{2}} (\Psi_{i+1} - \Psi_{i})P_{i+1}^0 \\
=   \frac{\delta \Delta t}{2} \frac{\Psi_{N_{\alpha}+2} - \Psi_{N_{\alpha}+1}}{h} P_{N_{\alpha}+1}^0  +  \frac{\delta \Delta t}{2} \sum_{i= N_{\alpha}+2}^{N_{\alpha}+N-1} \frac{\Psi_{i+1} - \Psi_i} {h} P_{i}^0 -\frac{\delta \Delta t}{2} \sum_{i= N_{\alpha}+1}^{N_{\alpha}+N-2} \frac{\Psi_{i+1} - \Psi_{i}}{h} P_{i+1}^0  - \frac{\delta \Delta t}{2} \frac{\Psi_{N_{\alpha}+N} - \Psi_{N_{\alpha}+N-1}}{h} P_{N_{\alpha}+N}^0.
\end{array}
\end{equation*} }}
Consequently, 
{\small{\begin{equation*}
\begin{array}{lll}
  \Delta t \sum_{i=1}^{N_{\max}} h_i \frac{U_i^1 - U_i^0}{\Delta t^2} P_{i}^0 - \sum_{i=1}^{N_{\max}} h_i \Psi_i P_{i}^0  + \frac{\Delta t}{2} \sum_{i=0}^{N_{\max}} \ell_{i+ \frac{1}{2}} (U_{i+1}^1 - U_{i}^1)( P_{i+1}^0 -P_{i}^0) -  \frac{\Delta t^2}{2} \sum_{i=0}^{N_{\max}}\ell_{i+ \frac{1}{2}} (\Psi_{i+1} - \Psi_{i})(P_{i+1}^0 -P_{i}^0) \\
=  \frac{\delta \Delta t}{2} \sum_{i= N_{\alpha}+1}^{N_{\alpha}+N-1} \frac{(\Psi_{i+1} - \Psi_i)(P_{i}^0 - P_{i+1}^0)} {h_i} .
\end{array}
\end{equation*} }}
Now, we modify the above equation as follows:
{\small{\begin{equation*}
\begin{array}{lll}
  \Delta t \sum_{i=1}^{N_{\max}} h_i \frac{U_i^1 - U_i^0}{\Delta t^2} P_{i}^0 - \sum_{i=1}^{N_{\max}} h_i \Psi_i P_{i}^0  + \frac{\Delta t}{2} \sum_{i=0}^{N_{\max}} \ell_{i+ \frac{1}{2}} (U_{i+1}^1 - U_{i}^1)( P_{i+1}^0 -P_{i}^0 + P_{i+1}^1 - P_{i+1}^1 - P_{i}^1 + P_{i}^1)\\
   -  \frac{\Delta t^2}{2} \sum_{i=0}^{N_{\max}}\ell_{i+ \frac{1}{2}} (\Psi_{i+1} - \Psi_{i})(P_{i+1}^0 -P_{i}^0)
=  \frac{\delta \Delta t}{2} \sum_{i= N_{\alpha}+1}^{N_{\alpha}+N-1} \frac{(\Psi_{i+1} - \Psi_i)(P_{i}^0 - P_{i+1}^0)} {h_i},
\end{array}
\end{equation*} }}
to get finally after small rearrangement 
{\small{\begin{equation} \label{eq-implicit-total-n=0}
\begin{array}{lll}
  \Delta t \sum_{i=1}^{N_{\max}} h_i \frac{U_i^1 - U_i^0}{\Delta t^2} P_{i}^0 - \sum_{i=1}^{N_{\max}} h_i \Psi_i P_{i}^0  + \frac{\Delta t}{2} \sum_{i=0}^{N_{\max}} \ell_{i+ \frac{1}{2}} (U_{i+1}^1 - U_{i}^1)( P_{i+1}^1  - P_{i}^1 )  + \frac{\delta \Delta t}{2} \sum_{i= N_{\alpha}+1}^{N_{\alpha}+N-1} \frac{(\Psi_{i+1} - \Psi_i)(P_{i+1}^0 - P_{i}^0)} {h} \\
  =   \frac{\Delta t^2}{2} \sum_{i=0}^{N_{\max}}\ell_{i+ \frac{1}{2}} (\Psi_{i+1} - \Psi_{i})(P_{i+1}^0 -P_{i}^0) 
  + \frac{\Delta t}{2} \sum_{i=0}^{N_{\max}} \ell_{i+ \frac{1}{2}} (U_{i+1}^1 - U_{i}^1)(P_{i+1}^1 - P_{i}^1 - P_{i+1}^0 + P_{i}^0)

\end{array}
\end{equation} }}
\textbf{Step 2.} Now, we continue for $n\geq 1$ and we multiply \eqref{DISCRETE-IMPLICIT} by $\Delta t P_{i}^n$.  Using the same techniques of translation of index $i$, we obtain as a result
{\small{\begin{equation*}
\begin{array}{lll}
\Delta t \sum_{i=1}^{N_{\max}} h_i \frac{U_i^{n+1} - U_i^n}{\Delta t^2} P_{i}^n -    \Delta t \sum_{i=1}^{N_{\max}} \frac{U_i^n - U_i^{n-1}}{\Delta t^2} P_{i}^n + \frac{\Delta t}{2}  \sum_{i=0}^{N_{\max}} \ell_{i+\frac{1}{2}} (U_{i+1}^{n+1} - U_i^{n+1})(P_{i+1}^ n- P_{i}^n) \\
+ \frac{\Delta t}{2}  \sum_{i=0}^{N_{\max}} \ell_{i+\frac{1}{2}} (U_{i+1}^{n-1} - U_i^{n-1})(P_{i+1}^ n- P_{i}^n)
+ \frac{\delta }{2} \sum_{i=N_{\alpha}+1}^{N_{\alpha}+N-1}\frac{(U_{i+1}^{n+1} - U_i^{n+1})(P_{i+1}^n - P_{i}^n)}{h} \\
 -  \frac{\delta }{2} \sum_{i=N_{\alpha}+1}^{N_{\alpha}+N-1}\frac{ (U_{i+1}^{n-1} - U_i^{n-1})(P_{i+1}^n - P_{i}^n)}{h} = 0.
\end{array}
\end{equation*}}}
\textbf{Step 3.} Summing the above equation for $n = 1, \cdots, \mathcal{N}-1$, we obtain
{\small{\begin{equation} \label{implicit,n>1}
\begin{array}{lll}
\sum_{n=1}^{\mathcal{N}-1}\Delta t \sum_{i=1}^{N_{\max}} h_i \frac{U_i^{n+1} - U_i^n}{\Delta t^2} P_{i}^n -   \sum_{n=1}^{\mathcal{N}-1} \Delta t \sum_{i=1}^{N_{\max}} \frac{U_i^n - U_i^{n-1}}{\Delta t^2} P_{i}^n + \frac{1}{2} \sum_{n=1}^{\mathcal{N}-1} \Delta t \sum_{i=0}^{N_{\max}} \ell_{i+\frac{1}{2}} (U_{i+1}^{n+1} - U_i^{n+1})(P_{i+1}^ n- P_{i}^n) \\
+ \frac{1}{2} \sum_{n=1}^{\mathcal{N}-1} \Delta t  \sum_{i=0}^{N_{\max}} \ell_{i+\frac{1}{2}} (U_{i+1}^{n-1} - U_i^{n-1})(P_{i+1}^ n- P_{i}^n)
+ \delta \sum_{n=1}^{\mathcal{N}-1} \Delta t \sum_{i=N_{\alpha}+1}^{N_{\alpha}+N-1} \frac{(U_{i+1}^{n+1} - U_i^{n+1})(P_{i+1}^n - P_{i}^n)}{2 h\Delta t} \\
 -  \delta \sum_{n=1}^{\mathcal{N}-1} \Delta t \sum_{i=N_{\alpha}+1}^{N_{\alpha}+N-1} \frac{(U_{i+1}^{n-1} - U_i^{n-1})(P_{i+1}^n - P_{i}^n)}{2 h\Delta t} = 0.
\end{array}
\end{equation}}}
\textbf{Working on the first two terms of Equation \eqref{implicit,n>1}}: Using the fact that $P_i^\mathcal{N} = 0$, we translate the index $n$ of the first term to get 
\begin{equation} \label{M1}
\begin{array}{lll}
 S_{1} = \sum_{n=1}^{\mathcal{N}-1}\Delta t \sum_{i=1}^{N_{\max}} h_i \frac{U_i^{n+1} - U_i^n}{\Delta t^2} P_{i}^n -   \sum_{n=1}^{\mathcal{N}-1} \Delta t \sum_{i=1}^{N_{\max}} \frac{U_i^n - U_i^{n-1}}{\Delta t^2} P_{i}^n
  \\
=  \sum_{n=2}^{\mathcal{N}}  \Delta t \sum_{i=1}^{N_{\max}} h_i \frac{U_i^{n} - U_i^{n-1}}{\Delta t^2} P_{i}^{n-1} -   \sum_{n=2}^{\mathcal{N}} \Delta t \sum_{i=1}^{N_{\max}}  h_i  \frac{U_i^n - U_i^{n-1}}{\Delta t^2} P_{i}^n -  \Delta t \sum_{i=1}^{N_{\max}} h_{i} \frac{U_i^1 - U_i^{0}}{\Delta t^2} P_{i}^1.
\end{array}
\end{equation}
\textbf{Working on the second two terms of Equation \eqref{implicit,n>1}}: Using the fact that $P_i^\mathcal{N} = 0$, we do the following modification
\begin{equation*} 
\begin{array}{lll}
S_{2} &=&\frac{1}{2} \sum_{n=1}^{\mathcal{N}-1} \Delta t \sum_{i=0}^{N_{\max}} \ell_{i+\frac{1}{2}} (U_{i+1}^{n+1} - U_i^{n+1})(P_{i+1}^ n- P_{i}^n)
+ \frac{1}{2} \sum_{n=1}^{\mathcal{N}-1} \Delta t  \sum_{i=0}^{N_{\max}} \ell_{i+\frac{1}{2}} (U_{i+1}^{n-1} - U_i^{n-1})(P_{i+1}^ n- P_{i}^n)\\
&=& \frac{1}{2} \sum_{n=1}^{\mathcal{N}-1} \Delta t \sum_{i=0}^{N_{\max}} \ell_{i+\frac{1}{2}} (U_{i+1}^{n+1} - U_i^{n+1})(P_{i+1}^ n- P_{i}^n + P_{i+1}^{n+1} - P_{i+1}^{n+1} +P_{i}^{n+1} - P_{i}^{n+1} ) \\
&&+ \frac{1}{2} \sum_{n=1}^{\mathcal{N}-1} \Delta t  \sum_{i=0}^{N_{\max}} \ell_{i+\frac{1}{2}} (U_{i+1}^{n-1} - U_i^{n-1})(P_{i+1}^ n- P_{i}^n + P_{i+1}^{n-1} - P_{i+1}^{n-1} +P_{i}^{n-1} - P_{i}^{n-1} )\\
&=& \frac{1}{2} \sum_{n=1}^{\mathcal{N}-1} \Delta t \sum_{i=0}^{N_{\max}} \ell_{i+\frac{1}{2}} (U_{i+1}^{n+1} - U_i^{n+1})( P_{i+1}^{n+1} - P_{i}^{n+1} )\\
&&+ \frac{1}{2} \sum_{n=1}^{\mathcal{N}-1} \Delta t  \sum_{i=0}^{N_{\max}} \ell_{i+\frac{1}{2}} (U_{i+1}^{n-1} - U_i^{n-1})(P_{i+1}^{n-1} -  P_{i}^{n-1} ) \\
&&+ \frac{1}{2} \sum_{n=1}^{\mathcal{N}-1} \Delta t \sum_{i=0}^{N_{\max}} \ell_{i+\frac{1}{2}} (U_{i+1}^{n+1} - U_i^{n+1})(P_{i+1}^ n- P_{i}^n  - P_{i+1}^{n+1} +P_{i}^{n+1} ) \\
&&+ \frac{1}{2} \sum_{n=1}^{\mathcal{N}-1} \Delta t  \sum_{i=0}^{N_{\max}} \ell_{i+\frac{1}{2}} (U_{i+1}^{n-1} - U_i^{n-1})(P_{i+1}^ n- P_{i}^n  - P_{i+1}^{n-1} +P_{i}^{n-1}).
\end{array}
\end{equation*}
By translation of index $n$ for the first two terms of the right hand-side of the above equation, we obtain
{\small{\begin{equation} \label{M2}
\begin{array}{lll}
S_{2} =& \frac{1}{2} \sum_{n=2}^{\mathcal{N}} \Delta t \sum_{i=0}^{N_{\max}} \ell_{i+\frac{1}{2}} (U_{i+1}^{n} - U_i^{n})( P_{i+1}^{n} - P_{i}^{n} )\\
&+ \frac{1}{2} \sum_{n=0}^{\mathcal{N}-2} \Delta t  \sum_{i=0}^{N_{\max}} \ell_{i+\frac{1}{2}} (U_{i+1}^{n} - U_i^{n})(P_{i+1}^{n} -  P_{i}^{n} ) \\
&+ \frac{1}{2} \sum_{n=1}^{\mathcal{N}-1} \Delta t \sum_{i=0}^{N_{\max}} \ell_{i+\frac{1}{2}} (U_{i+1}^{n+1} - U_i^{n+1})(P_{i+1}^ n- P_{i}^n  - P_{i+1}^{n+1} +P_{i}^{n+1} ) \\
&+ \frac{1}{2} \sum_{n=1}^{\mathcal{N}-1} \Delta t  \sum_{i=0}^{N_{\max}} \ell_{i+\frac{1}{2}} (U_{i+1}^{n-1} - U_i^{n-1})(P_{i+1}^ n- P_{i}^n  - P_{i+1}^{n-1} +P_{i}^{n-1})\\
S_{2}= &\sum_{n=2}^{\mathcal{N}} \Delta t \sum_{i=0}^{N_{\max}} \ell_{i+\frac{1}{2}} (U_{i+1}^{n} - U_i^{n})( P_{i+1}^{n} - P_{i}^{n} )  \\
&+ \frac{1}{2} \sum_{n=0}^{1} \Delta t  \sum_{i=0}^{N_{\max}} \ell_{i+\frac{1}{2}} (U_{i+1}^{n} - U_i^{n})(P_{i+1}^{n} -  P_{i}^{n} ) \\
&- \frac{ \Delta t}{2} \sum_{i=0}^{N_{\max}}  \ell_{i+\frac{1}{2}} (U_{i+1}^{\mathcal{N}-1} - U_i^{\mathcal{N}-1})(P_{i+1}^{\mathcal{N}-1} -  P_{i}^{\mathcal{N}-1}) \\
&+ \frac{1}{2} \sum_{n=1}^{\mathcal{N}-1} \Delta t \sum_{i=0}^{N_{\max}} \ell_{i+\frac{1}{2}} (U_{i+1}^{n+1} - U_i^{n+1})(P_{i+1}^ n- P_{i}^n  - P_{i+1}^{n+1} +P_{i}^{n+1} ) \\
&+ \frac{1}{2} \sum_{n=1}^{\mathcal{N}-1} \Delta t  \sum_{i=0}^{N_{\max}} \ell_{i+\frac{1}{2}} (U_{i+1}^{n-1} - U_i^{n-1})(P_{i+1}^ n- P_{i}^n  - P_{i+1}^{n-1} +P_{i}^{n-1}).
\end{array}
\end{equation}}}
Substituting \eqref{M1} and \eqref{M2} in \eqref{implicit,n>1}, we obtain 
\begin{equation}\label{total-implicit-n>1}
\begin{array}{lll}
-  \sum_{n=2}^{\mathcal{N}}  \Delta t \sum_{i=1}^{N_{\max}} h_{i} \frac{(U_i^{n} - U_i^{n-1})(P_{i}^n - P_{i}^{n-1})}{\Delta t^2} -  \Delta t \sum_{i=1}^{N_{\max}} h_{i} \frac{U_i^1 - U_i^{0}}{\Delta t^2} P_{i}^1  \\
+ \sum_{n=2}^{\mathcal{N}} \Delta t  \sum_{i=0}^{N_{\max}} \ell_{i+\frac{1}{2}} (U_{i+1}^{n} - U_i^{n})(P_{i+1}^ {n}- P_{i}^{n}) + \frac{1}{2} \sum_{n=0}^{1} \Delta t  \sum_{i=0}^{N_{\max}} \ell_{i+\frac{1}{2}} (U_{i+1}^{n} - U_i^{n})(P_{i+1}^ {n}- P_{i}^{n}) \\
+ \delta \sum_{n=1}^{\mathcal{N}-1} \Delta t \sum_{i=N_{\alpha}+1}^{N_{\alpha}+N-1} \frac{(U_{i+1}^{n+1} - U_i^{n+1} - U_{i+1}^{n-1} + U_i^{n-1})(P_{i+1}^n - P_{i}^n)}{2 h \Delta t}  = R_1,
\end{array}
\end{equation}
where 
\begin{equation*}
\begin{array}{lll}
R_1 &=&  \frac{ \Delta t}{2} \sum_{i=0}^{N_{\max}}  \ell_{i+\frac{1}{2}} (U_{i+1}^{\mathcal{N}-1} - U_i^{\mathcal{N}-1})(P_{i+1}^{\mathcal{N}-1} -  P_{i}^{\mathcal{N}-1}) \\
&&+ \frac{1}{2} \sum_{n=1}^{\mathcal{N}-1} \Delta t \sum_{i=0}^{N_{\max}} \ell_{i+\frac{1}{2}} (U_{i+1}^{n+1} - U_i^{n+1})(P_{i+1}^{n+1}  - P_{i}^{n+1} - P_{i+1}^ n + P_{i}^n  ) \\
&&+ \frac{1}{2} \sum_{n=1}^{\mathcal{N}-1} \Delta t  \sum_{i=0}^{N_{\max}} \ell_{i+\frac{1}{2}} (U_{i+1}^{n-1} - U_i^{n-1})( P_{i+1}^{n-1}  - P_{i}^{n-1} - P_{i+1}^ n + P_{i}^n ).
\end{array}
\end{equation*}
Combining \eqref{eq-implicit-total-n=0} and \eqref{total-implicit-n>1} , we obtain
\begin{equation}\label{total-implicit-n>=1}
\begin{array}{lll}
-  \sum_{n=1}^{\mathcal{N}}  \Delta t \sum_{i=1}^{N_{\max}} h_{i} \frac{(U_i^{n} - U_i^{n-1})(P_{i}^n - P_{i}^{n-1})}{\Delta t^2}   -   \sum_{i=1}^{N_{\max}} h_i \Psi_i P_{i}^0 
+  \sum_{n=1}^{\mathcal{N}} \Delta t  \sum_{i=0}^{N_{\max}}\ell_{i+\frac{1}{2}} (U_{i+1}^{n} - U_i^{n})(P_{i+1}^ {n}- P_{i}^{n})
\\
+ \delta \sum_{n=1}^{\mathcal{N}-1} \Delta t \sum_{i=N_{\alpha}+1}^{N_{\alpha}+N-1} \frac{(U_{i+1}^{n+1} - U_i^{n+1} - U_{i+1}^{n-1} + U_i^{n-1})(P_{i+1}^n - P_{i}^n)}{2 h \Delta t} = R,
\end{array}
\end{equation}
where 
\begin{equation*}
\begin{array}{lll}
R &=& \frac{ \Delta t}{2} \sum_{i=0}^{N_{\max}}  \ell_{i+\frac{1}{2}} (U_{i+1}^{\mathcal{N}-1} - U_i^{\mathcal{N}-1})(P_{i+1}^{\mathcal{N}-1} -  P_{i}^{\mathcal{N}-1}) \\
&&+ \frac{1}{2} \sum_{n=0}^{\mathcal{N}-1} \Delta t \sum_{i=0}^{N_{\max}} \ell_{i+\frac{1}{2}} (U_{i+1}^{n+1} - U_i^{n+1})(P_{i+1}^{n+1}  - P_{i}^{n+1} - P_{i+1}^ n + P_{i}^n  ) \\
&&+ \frac{1}{2} \sum_{n=1}^{\mathcal{N}-1} \Delta t  \sum_{i=0}^{N_{\max}} \ell_{i+\frac{1}{2}} (U_{i+1}^{n-1} - U_i^{n-1})( P_{i+1}^{n-1}  - P_{i}^{n-1} - P_{i+1}^ n + P_{i}^n )\\
&&+ \frac{\Delta t}{2} \sum_{i=0}^{N_{\max}} \ell_{i+ \frac{1}{2}} (\Delta t \Psi_{i+1} - \Delta t \Psi_i - U_{i+1}^0 + U_i^0)(P_{i+1}^0 - P_{i}^0) - \frac{\delta \Delta t}{2} \sum_{i=N_{\alpha}+1}^{N_{\alpha}+N-1} \frac{(\Psi_{i+1} - \Psi_i)(P_{i+1}^0 - P_{i}^0)}{h}.

\end{array}
\end{equation*}
Thus, we deduce from \eqref{total-implicit-n>=1}
\begin{equation*}
\begin{array}{l}
-  \int_0^T \int_0^{L} \partial^1 U_{\mathcal{T}, \Delta t} \partial^1 P_{\mathcal{T}, \Delta t} dx dt 
+  \int_0^T \sum_{i=0}^{N_{\max}} \ell_{i+\frac{1}{2}}(U_{i+1, \Delta t} - U_{i, \Delta t})(P_{i+1, \Delta t} - P_{i, \Delta t}) dt
\\
+ \delta \int_0 ^ T \sum_{i=N_{\alpha}+1}^{N_{\alpha}+N-1} \frac{[\partial^{1/2}(U_{i+1, \Delta t} - U_{i, \Delta t})](P_{i+1, \Delta t} - P_{i, \Delta t})}{h} dt  -  \int_0^{L} \Psi_{\mathcal{T}} P_{\mathcal{T}}^0 dx = R.
\end{array}
\end{equation*}
The above equation can be reformulated as the following
\begin{align}
-  \int_0^T \int_0^{\alpha} \partial^1 u_{\mathcal{T}, \Delta t} \partial^1 P_{\mathcal{T}, \Delta t} dx dt 
&-  \int_0^T \int_{\alpha}^{\beta} \partial^1 v_{\mathcal{T}, \Delta t} \partial^1 P_{\mathcal{T}, \Delta t} dx dt
-  \int_0^T \int_{\beta}^L \partial^1 w_{\mathcal{T}, \Delta t} \partial^1 P_{\mathcal{T}, \Delta t} dx dt \notag \\
& +  \int_0^T \sum_{i=0}^{N_{\alpha}-1} \ell_{i+\frac{1}{2}} (u_{i+1, \Delta t} - u_{i, \Delta t})(P_{i+1, \Delta t} - P_{i, \Delta t}) dt \notag\\
& +  \int_0^T \frac{2 C_1^2 C_2^2}{C_1^2 h + C_2^2 h_{\alpha}} (v_{N_{\alpha}+1, \Delta t} - u_{N_{\alpha}, \Delta t})(P_{N_{\alpha}+1, \Delta t} - P_{N_{\alpha}, \Delta t}) dt\notag \\
& +  \int_0^T \sum_{i=N_{\alpha}+1}^{N_{\alpha}+N-1} \ell_{i+\frac{1}{2}} (v_{i+1, \Delta t} - v_{i, \Delta t})(P_{i+1, \Delta t} - P_{i, \Delta t}) dt \notag\\
& +  \int_0^T \frac{2 C_3^2 C_2^2}{C_3^2 h + C_2^2 h_{\beta}} (w_{N_{\alpha}+N+1, \Delta t} - v_{N_{\alpha}+N, \Delta t})(P_{N_{\alpha}+N+1, \Delta t} - P_{N_{\alpha}+N, \Delta t}) dt \label{discrete weak formulation}\\
&+   \int_0^T \sum_{i=N_{\alpha}+N+1}^{N_{\max}} \ell_{i+\frac{1}{2}} (w_{i+1, \Delta t} - w_{i, \Delta t})(P_{i+1, \Delta t} - P_{i, \Delta t}) dt \notag\\
&+ \delta \int_0 ^ T \sum_{i=N_{\alpha}+1}^{N_{\alpha}+N-1}  \frac{[\partial^{1/2}(v_{i+1, \Delta t} - v_{i, \Delta t})](P_{i+1, \Delta t} - P_{i, \Delta t})}{h} dt \notag\\
&-  \int_0^{\alpha} \psi_{\mathcal{T}} P_{\mathcal{T}}^0 dx -  \int_{\alpha}^{\beta} \zeta_{\mathcal{T}} P_{\mathcal{T}}^0 dx -   \int_{\beta}^L \theta_{\mathcal{T}} P_{\mathcal{T}}^0 dx = R \notag
\end{align}
\begin{Theorem} \label{Convergence theorem}
Assume that Hypothesis \eqref{H'} holds. For $m \in \N$, let $\mathcal{T}_m$ be the admissible mesh described in Section \eqref{Section Mesh} and $\Delta t_m$ be the time step satisfying condition \eqref{TA} .  Let $(u_{\mathcal{T}_m, \Delta t_m}, v_{\mathcal{T}_m, \Delta t_m}, w_{\mathcal{T}_m, \Delta t_m})$ be the discrete solution of \eqref{DISCRETE-IMPLICIT-n=0} and \eqref{DISCRETE-IMPLICIT}. Assume that size$(\mathcal{T}_m) \to 0$ and $\Delta t_m \to 0$ as $m \to \infty$. Then there exists a subsequence denoted also by  $(u_{\mathcal{T}_m, \Delta t_m}, v_{\mathcal{T}_m, \Delta t_m}, w_{\mathcal{T}_m, \Delta t_m})$ that converges weakly$^{\star}$ to the weak solution $(u,v,w) \in L^{\infty}([0,T];L^{2}(0,\alpha)) \times L^{\infty}([0,T];L^{2}(\alpha,\beta)) \times L^{\infty}([0,T];L^{2}(\beta,L))$ of problem \eqref{continuous weak variational problem}.
\end{Theorem}
\begin{proof} \textbf{From now on, for the sake of simplicity, we will denote $\Delta t$ for $\Delta t_m$ and $h$ for $h_m$}.

From Theorems \ref{Stability theorem for first term} and \ref{Theorem Stability n>1}, we obtain that  $(u_{\mathcal{T}_m, \Delta t_m}, v_{\mathcal{T}_m, \Delta t_m}, w_{\mathcal{T}_m, \Delta t_m})$ is bounded in \linebreak
 $L^{\infty}([0,T];L^{2}(0,\alpha))~\times~L^{\infty}([0,T];L^{2}(\alpha,\beta))~\times~L^{\infty}([0,T];L^{2}(\beta,L))$.

Then there exists a subsequence, still denoted by  $(u_{\mathcal{T}_m, \Delta t_m}, v_{\mathcal{T}_m, \Delta t_m}, w_{\mathcal{T}_m, \Delta t_m})$, such that
$$
(u_{\mathcal{T}_m, \Delta t_m}, v_{\mathcal{T}_m, \Delta t_m}, w_{\mathcal{T}_m, \Delta t_m}) \rightharpoonup^{\star} (u,v,w)  \textrm{ in }
 L^{\infty}([0,T];L^{2}(0,\alpha)) \times L^{\infty}([0,T];L^{2}(\alpha,\beta)) \times L^{\infty}([0,T];L^{2}(\beta,L))
$$
as $m \to \infty$. 

It is left to show that $(u,v,w)$ is the weak solution of problem \eqref{continuous weak variational problem}. 

Thanks to \cite{EHM}, there exists $(u,v,w) \in \mathcal{C}^0([0,T]; \mathbb{L}^2)$ with $(u (.,0), v(.,0), w(.,0)) = (\varphi(.), \eta(.), \gamma(.))$ and a subsequence denoted again by  $(u_{\mathcal{T}_m, \Delta t_m}, v_{\mathcal{T}_m, \Delta t_m}, w_{\mathcal{T}_m, \Delta t_m})$ converging to $(u,v,w)$ in $L^{\infty}([0,T]; \mathbb{L}^2)$. Now, the proof of Theorem \ref{Convergence theorem} is divided into steps. \\[0.1in]
\textbf{Step 1.} 
Before we study the convergence of the numerical solution to the weak one, we have to show that $R \to 0$ as $m \to \infty$. The first term of $R$ can be estimated as the following using Holder's inequality
\begin{equation*}
\begin{array}{lll}
\frac{ \Delta t}{2} \sum_{i=0}^{N_{\max}}  \ell_{i+\frac{1}{2}} (U_{i+1}^{\mathcal{N}-1} - U_i^{\mathcal{N}-1})(P_{i+1}^{\mathcal{N}-1} -  P_{i}^{\mathcal{N}-1}) &\leq \frac{\Delta t}{2} \bigg[ \sum_{i=0}^{N_{\max}} \ell_{i+\frac{1}{2}}(U_{i+1}^{\mathcal{N}-1} - U_i^{\mathcal{N}-1})^2   \bigg]^{1/2} \bigg[ \sum_{i=0}^{N_{\max}} \ell_{i+\frac{1}{2}}(P_{i+1}^{\mathcal{N}-1} - P_{i}^{\mathcal{N}-1})^2  \bigg]^{1/2} \\
& =  \frac{ \Delta t}{2} || U^{\mathcal{N}-1}_{\mathcal{T}_m}||_{1, \mathcal{T}} || P^{\mathcal{N}-1}_{\mathcal{T}_m}||_{1, \mathcal{T}} \to 0 \quad \textrm{as} \ m \to \infty \ (\Delta t \to 0).
\end{array}
\end{equation*}
Considering the second term of $R$, we firstly write:
\begin{equation*}
\begin{array}{ll}
\frac{1}{2} \sum_{n=0}^{\mathcal{N}-1} &\Delta t \sum_{i=0}^{N_{\max}} \ell_{i+\frac{1}{2}} (U_{i+1}^{n+1} - U_i^{n+1})(P_{i+1}^{n+1}  - P_{i}^{n+1} - P_{i+1}^ n + P_{i}^n  )\\
&=\frac{\Delta t}{2} \sum_{n=1}^{\mathcal{N}-1} \Delta t \sum_{i=0}^{N_{\max}} \ell_{i+ \frac{1}{2}} (U_{i+1}^{n+1} - U_i^{n+1})  \frac{(P_{i+1}^{n+1}  - P_{i}^{n+1} - P_{i+1}^ n + P_{i}^n)}{\Delta t} \\
\end{array}
\end{equation*}
Using Holder's inequality leads to:
\begin{equation*}
\begin{array}{ll}
\frac{1}{2} \sum_{n=0}^{\mathcal{N}-1} &\Delta t \sum_{i=0}^{N_{\max}} \ell_{i+\frac{1}{2}} (U_{i+1}^{n+1} - U_i^{n+1})(P_{i+1}^{n+1}  - P_{i}^{n+1} - P_{i+1}^ n + P_{i}^n  )\\
&\leq \frac{\Delta t }{2} \sum_{n=1}^{\mathcal{N}-1} \Delta t \bigg[ \sum_{i=0}^{N_{\max}} \ell_{i+ \frac{1}{2}} (U_{i+1}^{n+1} - U_i^{n+1})^2 \bigg]^{1/2}  \bigg[ \sum_{i=0}^{N_{\max}} \ell_{i+ \frac{1}{2}} \frac{(P_{i+1}^ {n+1}- P_{i}^{n+1} - P_{i+1}^{n}  + P_{i}^{n} )^2}{\Delta t^2} \bigg]^{1/2} \\
&= \frac{\Delta t}{2}  \int_0^T || U_{\mathcal{T}_m, \Delta t_m} ||_{1, \mathcal{T}} || \partial^1 P_{\mathcal{T}_m, \Delta t_m} ||_{1, \mathcal{T}} dt \to 0 \quad \textrm{as} \ m \to \infty \ (\Delta t \to 0).
\end{array}
\end{equation*}
We continue in a similar way
\begin{equation*}
\begin{array}{lll}
 \frac{\Delta t}{2} \sum_{i=0}^{N_{\max}} \ell_{i+ \frac{1}{2}} (\Delta t \Psi_{i+1} - \Delta t \Psi_i - U_{i+1}^0 + U_i^0)(P_{i+1}^0 - P_{i}^0)  \\
= \frac{\Delta t^2}{2} \sum_{i=0}^{N_{\max}} \ell_{i+\frac{1}{2}} (\Psi_{i+1} -  \Psi_i) (P_{i+1}^0 - P_{i}^0)  - \frac{ \Delta t}{2} \sum_{i=0}^{N_{\max}} \ell_{i+\frac{1}{2}} (U_{i+1}^0 -  U_i^0) (P_{i+1}^0 - P_{i}^0)\\
\leq \frac{\Delta t^2}{2}  \bigg[ \sum_{i=0}^{N_{\max}}  \ell_{i+\frac{1}{2}}( \Psi_{i+1} - \Psi_i)^2 \bigg]^{1/2}  \bigg[ \sum_{i=0}^{N_{\max}}  \ell_{i+\frac{1}{2}}(P_{i+1}^ 0- P_{i}^0 )^2 \bigg]^{1/2}  \\
+ \frac{\Delta t }{2} \bigg[ \sum_{i=0}^{N_{\max}} \ell_{i+\frac{1}{2}} ( U_{i+1}^0 - U_i^0)^2 \bigg]^{1/2}  \bigg[ \sum_{i=0}^{N_{\max}}\ell_{i+\frac{1}{2}}  (P_{i+1}^ 0- P_{i}^0 )^2 \bigg]^{1/2} \\
=  \frac{\Delta t^2 }{2}  || \Psi_{\mathcal{T}_m}||_{1, \mathcal{T}} || P^{0}_{\mathcal{T}_m}||_{1, \mathcal{T}} +
\frac{ \Delta t}{2} || U^0_{\mathcal{T}_m}||_{1, \mathcal{T}} || P^{0}_{\mathcal{T}_m}||_{1, \mathcal{T}} \to 0 \quad \textrm{as} \ m \to \infty \ (\Delta t \to 0).
\end{array}
\end{equation*}
We perform the same technique for the other terms and we deduce that $R \to 0$ as $m \to \infty$.\\[0.1in]
\textbf{Step 2.} Now, we move to show the convergence of the numerical solution to the weak solution of problem \eqref{continuous weak variational problem}.

Due to the regularity of the function $P$, $\partial^1 P_{\mathcal{T}_m, \Delta t_m} \to P_t$ strongly in $L^2([0,T]; (L^{2}0,\alpha))$. However, the stability estimate \eqref{stability-estimate} shows that $ \partial^1 u_{\mathcal{T}_m, \Delta t_m} $ is bounded in $L^{\infty}([0,T]; L^2(0,\alpha))$ and thus $\partial^1 u_{\mathcal{T}_m, \Delta t_m} \rightharpoonup^{\star} u_t$ in $L^{\infty}([0,T]; L^2(0,\alpha))$. We obtain
\begin{equation*}
-  \int_0^T \int_0^{\alpha} \partial^1 u_{\mathcal{T}_m, \Delta t_m} \partial^1 P_{\mathcal{T}_m, \Delta t_m} dx dt \rightarrow -  \int_0^T \int_0^{\alpha} u_t P_t dx dt \quad \textrm{as} \ m \to \infty.
\end{equation*}
Similarly,
\begin{eqnarray}
-  \int_0^T \int_{\alpha}^{\beta} \partial^1 v_{\mathcal{T}_m, \Delta t_m} \partial^1 P_{\mathcal{T}_m, \Delta t_m} dx dt &\rightarrow& -  \int_0^T \int_{\alpha}^{\beta} v_t P_t dx dt  \quad \textrm{as} \ m \to \infty, \nonumber \\
-  \int_0^T \int_{\beta}^L \partial^1 w_{\mathcal{T}_m, \Delta t_m} \partial^1 P_{\mathcal{T}_m, \Delta t_m} dx dt &\rightarrow& -  \int_0^T \int_{\beta}^L w_t P_t dx dt  \quad \textrm{as} \ m \to \infty. \nonumber
\end{eqnarray}
Also, using the fact that for all $\psi \in H^1(0,\alpha)$, $\psi_{\mathcal{T}_m} \to \psi \in L^2(0,\alpha)$, we deduce that 
\begin{equation*}
-  \int_0^{\alpha} \psi_{\mathcal{T}_m} P_{\mathcal{T}_m}^0 dx \to - \int_0^{\alpha} \psi(x) P(x,0) dx \quad \textrm{as} \ m \to \infty,
\end{equation*}  
for function $P$ regular enough. Similarly, 
\begin{eqnarray}
-  \int_{\alpha}^{\beta} \zeta_{\mathcal{T}_m} P_{\mathcal{T}_m}^0 dx &\to& -  \int_{\alpha}^{\beta} \zeta(x) P(x,0) dx \quad \textrm{as} \ m \to \infty, \nonumber \\
-   \int_{\beta}^L \theta_{\mathcal{T}_m} P_{\mathcal{T}_m}^0 dx  &\to& -   \int_{\beta}^L \theta(x) P(x,0) dx \quad \textrm{as} \ m \to \infty. \nonumber
\end{eqnarray}
\textbf{Step 3.} Let
\begin{eqnarray}
T_1 &=&  \int_0^T \sum_{i=0}^{N_{\alpha}-1} \ell_{i+\frac{1}{2}} (u_{i+1, \Delta t} - u_{i, \Delta t})(P_{i+1, \Delta t} - P_{i, \Delta t}) dt \nonumber \\
&+&  \int_0^T \frac{2 C_1^2 C_2^2}{C_1^2 h + C_2^2 h_{\alpha}} (v_{N_{\alpha}+1, \Delta t} - u_{N_{\alpha}, \Delta t})(P_{N_{\alpha}+1, \Delta t} - P_{N_{\alpha}, \Delta t}) dt, \nonumber \\
T_2&=&  \int_0^T \sum_{i=N_{\alpha}+1}^{N_{\alpha}+N-1} \ell_{i+\frac{1}{2}} (v_{i+1, \Delta t} - v_{i, \Delta t})(P_{i+1, \Delta t} - P_{i, \Delta t}) dt \nonumber \\
&+&  \int_0^T \frac{2 C_3^2 C_2^2}{C_3^2 h + C_2^2 h_{\beta}} (w_{N_{\alpha}+N+1, \Delta t} - v_{N_{\alpha}+N, \Delta t})(P_{N_{\alpha}+N+1, \Delta t} - P_{N_{\alpha}+N, \Delta t}) dt, \nonumber \\
T_3 &=&  \int_0^T \sum_{i=N_{\alpha}+N+1}^{N_{\max}} \ell_{i+\frac{1}{2}} (w_{i+1, \Delta t} - w_{i, \Delta t})(P_{i+1, \Delta t} - P_{i, \Delta t}) dt
, \nonumber \\
T_4 &=& \delta \int_0 ^ T \sum_{i=N_{\alpha}+1}^{N_{\alpha}+N-1}  \frac{[\partial^{1/2}(v_{i+1, \Delta t_m} - v_{i, \Delta t_m})](P_{i+1, \Delta t_m} - P_{i, \Delta t_m})}{h}  dt. \nonumber
\end{eqnarray}
Let us first treat the term $T_1$. For this sake, consider the following auxiliary term
\begin{equation*}
\tilde{T}_1 = - C_1^2 \int_0^T \int_{0}^{\alpha}  u_{\mathcal{T}_m, \Delta t_m} \ P_{xx}(x,t) dxdt + C_1^2 \int_0^T v_{N_{\alpha}+1, \Delta t_m} P_x(\alpha,t) dt.
\end{equation*}
Set the following trace operator
\begin{eqnarray}
{\textrm{tr}}: H^1(\alpha, \beta) &\longrightarrow& L^2( (\alpha, \beta)), \nonumber 
\end{eqnarray}
and define the following space
\begin{eqnarray}
H^{1/2} ((\alpha, \beta )) &=& \{ v \in L^2((\alpha, \beta)) \ | \ \exists \tilde{v} \in H^1(\alpha, \beta): v = \textrm{tr}(\tilde{v}) \}. \nonumber 
\end{eqnarray}
Regarding that $x_{N_{\alpha}+1} = \alpha + \frac{h}{2}$, we get $x_{N_{\alpha}+1} \to \alpha$ as $h \to 0$. However, from the stability results, $v_{\mathcal{T}_m}$ is bounded in $H^1(\alpha,\beta)$ and hence, there exists a subsequence, denoted again by $v_{\mathcal{T}_m}$ converging to $v$ in $H^1 (\alpha, \beta)$ as $m \to \infty$. Thus, due to the continuity of the trace operator, ${\textrm{tr}} (v_{\mathcal{T}_m}) \to {\textrm{tr}} (v)$ in $H^{1/2} (\alpha, \beta)$, and therefore $v_{N_{\alpha}+1} \to v(\alpha)= u(\alpha)$. So, we can say that $$v_{N_{\alpha}+1}(t) \to u(\alpha,t)$$
\noindent almost everywhere in time as size$(\mathcal{T}) \to 0$. Consequently, $v_{{N_{\alpha}+1}, \Delta t_m} \to u(\alpha,t)$ as $m \to \infty$ for all $t \in [0,T]$ due to the continuity of the functions in time.\\[0.1in]
Also, as $ u_{\mathcal{T}_m, \Delta t_m} \to u $ in $L^{\infty}([0,T]; L^2(0, \alpha))$, we obtain
\begin{equation*}\label{convergence of T1 tilde}
\tilde{T}_1 \to - C_1^2 \int_0^T \int_{0}^{\alpha} u P_{xx} dxdt + C_1^2 \int_0^T u(\alpha,t) P_x(\alpha,t) dt = C_1^2 \int_0^T \int_{0}^{\alpha} u_x(x,t)P_x(x,t)dxdt.
\end{equation*}
Now, we reformulate $\tilde{T}_1$ as the following:
\begin{equation*} \label{formulation T1 tilde}
\begin{array}{lll}
\tilde{T}_1 &= -  \sum_{n=1}^{\mathcal{N}} \Delta t \sum_{i=0}^{N_{\alpha}} c_i^2 u_i^n \int_{K_{i}} P_{xx}(x,t_n) dx + \sum_{n=1}^{\mathcal{N}} \Delta t \ c_{N_{\alpha}}^2 v_{N_{\alpha}+1}^n P_x(\alpha, t_n) 
\\
&- C_1^2 \int_0^T \int_0^{\alpha}  u_{\mathcal{T}_m, \Delta t_m} (P_{xx}(x,t) - P_{xx} (x,t_n))dxdt - C_1^2 \int_0^T v_{N_{\alpha}+1, \Delta t _m} (P_x(\alpha, t_n) - P_x(\alpha, t)) dt \\
 &=  \tilde{T}_{11} + \tilde{T}_{12} + \tilde{T}_{13} \ ,
\end{array}
\end{equation*}
where 
\begin{align*}
	\tilde{T}_{11} &=-  \sum_{n=1}^{\mathcal{N}} \Delta t \sum_{i=1}^{N_{\alpha}} u_i^n \int_{K_{i}} c_i^2 \ P_{xx}(x,t_n) dx  \\
	\tilde{T}_{12} &=\sum_{n=1}^{\mathcal{N}} \Delta t \ c_{N_{\alpha}}^2 v_{N_{\alpha}+1}^n P_x(\alpha, t_n) \\
	\tilde{T}_{13} &=- C_1^2 \int_0^T \int_0^{\alpha}  u_{\mathcal{T}_m, \Delta t_m} (P_{xx}(x,t) - P_{xx} (x,t_n))dxdt - C_1^2 \int_0^T v_{N_{\alpha}+1, \Delta t _m} (P_x(\alpha, t_n) - P_x(\alpha, t)) dt \\
\end{align*}
The continuity of the trace operator implies that $\vert v_{N_{\alpha}+1, \Delta t_m} \vert \leq c \vert \vert v_{\mathcal{T}_m , \Delta t_m}   \vert \vert $ where $c$ is 
a positive constant independent of the discretization parameters. 

Concerning that $u_{\mathcal{T}_m, \Delta t_m} $ is bounded in $L^{\infty}([0,T]; L^2(0, \alpha))$,  $v_{\mathcal{T}_m, \Delta t_m} $ is bounded in $L^{\infty}([0,T]; L^2(\alpha, \beta))$ and the function $P$ is smooth enough, then $\tilde{T}_{13} \to 0$ as $m \to \infty$. 

For a pedagogical purpose, we will detail carefully the computations for the treatment of the two terms $\tilde{T}_{11}$ and $\tilde{T}_{12}$ and we will refer to these arguments to treat the other terms in the sequel.
 
\begin{equation*}
\begin{array}{lll}
\tilde{T}_{11} &= -  \sum_{n=1}^{\mathcal{N}} \Delta t \sum_{i=1}^{N_{\alpha}} u_i^n \int_{K_{i}} c_i^2 \ P_{xx}(x,t_n) dx  \\
& = -  \sum_{n=1}^{\mathcal{N}} \Delta t \sum_{i=1}^{N_{\alpha}} u_i^n C_1^2  (P_{x}(x_{i+\frac{1}{2}},t_n) - P_{x}(x_{i-\frac{1}{2}},t_n) )\\
& = -  \sum_{n=1}^{\mathcal{N}} \Delta t \sum_{i=1}^{N_{\alpha}} u_i^n C_1^2  P_{x}(x_{i+\frac{1}{2}},t_n) + \sum_{n=1}^{\mathcal{N}} \Delta t \sum_{i=1}^{N_{\alpha}} u_i^n C_1^2  P_{x}(x_{i-\frac{1}{2}},t_n) \\
\end{array}
\end{equation*} 
Since $u_0^n = 0$, using an index translation leads to :
\begin{equation*}
\begin{array}{lll}
\tilde{T}_{11} & =  -   \sum_{n=1}^{\mathcal{N}} \Delta t \sum_{i=1}^{N_{\alpha}} u_i^n C_1^2  P_{x}(x_{i+\frac{1}{2}},t_n) + \sum_{n=1}^{\mathcal{N}} \Delta t \sum_{i=1}^{N_{\alpha}} u_{i+1}^n C_1^2  P_{x}(x_{i+\frac{1}{2}},t_n)\\
& =  \sum_{n=1}^{\mathcal{N}} \Delta t \sum_{i=1}^{N_{\alpha}} (u_{i+1}^n - u_i^n) C_1^2  P_{x}(x_{i+\frac{1}{2}},t_n) - \sum_{n=1}^{\mathcal{N}} \Delta t u_{N_\alpha}^n C_1^2  P_{x}(x_{N_\alpha+\frac{1}{2}},t_n)
\end{array}
\end{equation*} 
Using equation \eqref{li} which defines $\ell_{i+\frac{1}{2}}$, we may replace $C_1^2$ by $\ell_{i+\frac{1}{2}} h_{i+\frac{1}{2}}$ in the first term to obtain : 
	\begin{equation*}
		\begin{array}{lll}
		\tilde{T}_{11} &  =  \sum_{n=1}^{\mathcal{N}} \Delta t \sum_{i=1}^{N_{\alpha}} (u_{i+1}^n - u_i^n)  \ell_{i+\frac{1}{2}} h_{i+\frac{1}{2}} P_{x}(x_{i+\frac{1}{2}},t_n) - \sum_{n=1}^{\mathcal{N}} \Delta t u_{N_\alpha}^n C_1^2  P_{x}(x_{N_\alpha+\frac{1}{2}},t_n)
	\end{array}			
	\end{equation*}
Since $P_i^n = P(x_i,t_n)$ and since $P$ is regular, a Taylor expansion in the neighborhood of $x_{i+\frac{1}{2}}$ leads to 
\[
\ell_{i+\frac{1}{2}} \left(P_{i+1}^n - P_i^n\right) = \ell_{i+\frac{1}{2}} h_{i+\frac{1}{2}} P_{x}(x_{i+\frac{1}{2}},t_n) + \mathcal{O}(\mbox{size}(\mathcal{T})) .
\]
Using equation \eqref{eq-interface-alpha} leads to 
\[
C_1^2  P_{x}(x_{N_\alpha+\frac{1}{2}},t_n) = \ell_{N_\alpha+\frac{1}{2}}\left(P_{N_\alpha+1}^n - P_{N_\alpha}^n\right) + \mathcal{O}(\mbox{size}(\mathcal{T})) .
\]
Since 
$\sum_{n=1}^{\mathcal{N}} \Delta t  \mathcal{O}(\mbox{size}(\mathcal{T})= T \mathcal{O}(\mbox{size}(\mathcal{T})) = \mathcal{O}(\mbox{size}(\mathcal{T}))$ we obtain finally
\begin{equation}\label{T11 tilde}
\tilde{T}_{11} =  \sum_{n=1}^{\mathcal{N}} \Delta t \sum_{i=0}^{N_{\alpha}-1}  \ell_{i+\frac{1}{2}}(u_{i+1}^n - u_i^n)(P_{i+1}^n - P_{i}^n) -  
\sum_{n=1}^{\mathcal{N}} \Delta t \ u_{N_{\alpha}}^n \ell_{N_\alpha+\frac{1}{2}}(P_{N_{\alpha}+1}^n - P_{N_{\alpha}}^n) + \mathcal{O}(\mbox{size}(\mathcal{T}))  \ .
\end{equation}
We also get: 
\begin{equation*}
\begin{array}{ll}
\tilde{T}_{12} &= \sum_{n=1}^{\mathcal{N}} \Delta t \ c_{N_{\alpha}}^2 v_{N_{\alpha}+1}^n P_x(\alpha, t_n)  = \sum_{n=1}^{\mathcal{N}} \Delta t \ C_1^2 v_{N_{\alpha}+1}^n P_x(\alpha, t_n) \\
&=  \sum_{n=1}^{\mathcal{N}} \Delta t \ C_1^2 v_{N_{\alpha}+1}^n P_x(x_{N_\alpha+\frac{1}{2}}, t_n) \\
\end{array}
\end{equation*}
Using again equation \eqref{eq-interface-alpha} leads to:
\begin{equation}\label{T12 tilde}
\tilde{T}_{12} = \sum_{n=1}^{\mathcal{N}} \Delta t \ \ell_{N_{\alpha}+\frac{1}{2}} v_{N_{\alpha}+1}^n (P_{N_{\alpha}+1}^n - P_{N_{\alpha}}^n) + \mathcal{O}(\mbox{size}(\mathcal{T}))  \ .
\end{equation}
Similarly, to treat the term $T_2$, consider the following auxiliary term
\begin{equation*}
\tilde{T}_2 = - C_2^2 \int_0^T \int_{\alpha}^{\beta}  v_{\mathcal{T}_m, \Delta t_m} \ P_{xx}(x,t) dxdt + C_2^2 \int_0^T v_{N_{\alpha}+N, \Delta t_m} P_x(\beta,t) dt - 
C_2^2 \int_0^T v_{N_{\alpha}+1, \Delta t_m} P_x(\alpha,t) dt.
\end{equation*}
Regarding that $x_{N_{\alpha}+N} = \beta - \frac{h}{2}$ , we get $x_{N_{\alpha}+N} \to \beta$ as $h \to 0$. However, from the stability results, $v_{\mathcal{T}_m}$ is bounded in $H^1(\alpha,\beta)$ and hence, there exists a subsequence, denoted again by $v_{\mathcal{T}_m}$ converging to $v$ in $H^1 (\alpha, \beta)$ as $m \to \infty$. Thus, due to the continuity of the trace operator, ${\textrm{tr}} (v_{\mathcal{T}_m}) \to {\textrm{tr}} (v)$ in $H^{1/2} (\alpha, \beta)$, and therefore $v_{N_{\alpha}+N} \to v(\beta)$ . So, we can say that $$v_{N_{\alpha}+N}(t) \to v(\beta,t)$$
\noindent almost everywhere in time as size$(\mathcal{T}) \to 0$. Consequently, $v_{{N_{\alpha}+N}, \Delta t_m} \to v(\beta,t)$ as $m \to \infty$ for all $t \in [0,T]$ due to the continuity of the functions in time. \\[0.1in]
As $ v_{\mathcal{T}_m, \Delta t_m} \to v$ in $L^{\infty}([0,T]; L^2(\alpha, \beta))$, we obtain
\begin{equation*}\label{convergence of T2 tilde}
\tilde{T}_2 \to - C_2^2 \int_0^T \int_{\alpha}^{\beta} v(x,t)P_{xx}(x,t)dxdt  + C_2^2 \int_0^T v(\beta,t) P_x(\beta,t) - C_2^2 \int_0^T v(\alpha,t) P_x(\alpha,t) = C_2^2 \int_0^T \int_{\alpha}^{\beta} v_x(x,t)P_x(x,t)dxdt .
\end{equation*}
Again, we reformulate $\tilde{T}_2$ as the following: 
\begin{equation*}
\begin{array}{lll}
\tilde{T}_2 &= -  \sum_{n=1}^{\mathcal{N}} \Delta t \sum_{i=N_{\alpha}+1}^{N_{\alpha}+N} c_i^2 v_i^n \int_{K_{i}} P_{xx}(x,t_n) dx + \sum_{n=1}^{\mathcal{N}} \Delta t \ c_{N_{\alpha}+N}^2 v_{N_{\alpha}+N}^n P_x(\beta, t_n)  - \sum_{n=1}^{\mathcal{N}} \Delta t \ c_{N_{\alpha}+1}^2 v_{N_{\alpha}+1}^n P_x(\alpha, t_n) 
\\
&- C_2^2 \int_0^T \int_{\alpha}^{\beta}  v_{\mathcal{T}_m, \Delta t_m} (P_{xx}(x,t) - P_{xx} (x,t_n))dxdt - C_2^2 \int_0^T v_{N_{\alpha}+N, \Delta t _m} (P_x(\beta, t_n) - P_x(\beta, t)) dt \\
&+ C_2^2 \int_0^T v_{N_{\alpha}+1, \Delta t _m} (P_x(\alpha, t_n) - P_x(\alpha, t)) dt \\
 &=  \tilde{T}_{21} + \tilde{T}_{22} + \tilde{T}_{23} \ ,
\end{array}
\end{equation*}
where 
\begin{align*}
\tilde{T}_{21} &= -  \sum_{n=1}^{\mathcal{N}} \Delta t \sum_{i=N_{\alpha}+1}^{N_{\alpha}+N} c_i^2 v_i^n \int_{K_{i}} P_{xx}(x,t_n) dx \\
\tilde{T}_{22} &= \sum_{n=1}^{\mathcal{N}} \Delta t \ c_{N_{\alpha}+N}^2 v_{N_{\alpha}+N}^n P_x(\beta, t_n)  - \sum_{n=1}^{\mathcal{N}} \Delta t \ c_{N_{\alpha}+1}^2 v_{N_{\alpha}+1}^n P_x(\alpha, t_n)  \\
\tilde{T}_{23} &=- C_2^2 \int_0^T \int_{\alpha}^{\beta}  v_{\mathcal{T}_m, \Delta t_m} (P_{xx}(x,t) - P_{xx} (x,t_n))dxdt - C_2^2 \int_0^T v_{N_{\alpha}+N, \Delta t _m} (P_x(\beta, t_n) - P_x(\beta, t)) dt \\
&+ C_2^2 \int_0^T v_{N_{\alpha}+1, \Delta t _m} (P_x(\alpha, t_n) - P_x(\alpha, t)) dt
\end{align*}
The continuity of the trace operator implies that $\vert v_{N_{\alpha}+N, \Delta t_m} \vert \leq c \vert \vert v_{\mathcal{T}_m , \Delta t_m}   \vert \vert $ where $c$ is a positive constant independent of 
the discretization parameters.

Concerning that $v_{\mathcal{T}_m, \Delta t_m} $ is bounded in $L^{\infty}([0,T]; L^2(\alpha, \beta))$ and the function $P$ is smooth enough, then $\tilde{T}_{23} \to 0$ as $m \to \infty$. 

Using the same argument and computations as done for the term $\tilde{T}_{11}$ and $\tilde{T}_{12}$ , we get  
\begin{equation} \label{tilde T21}
\begin{array}{lll}
\tilde{T}_{21} &=  \sum_{n=1}^{\mathcal{N}} \Delta t \sum_{i=N_{\alpha}+1}^{N_{\alpha}+N-1}  \ell_{i+\frac{1}{2}}(v_{i+1}^n - v_i^n)(P_{i+1}^n - P_{i}^n) +  \sum_{n=1}^{\mathcal{N}} \Delta t \ v_{N_{\alpha}+1}^n \ell_{N_\alpha + \frac{1}{2}} (P_{N_{\alpha}+1}^n - P_{N_{\alpha}}^n) \\
& -  \sum_{n=1}^{\mathcal{N}} \Delta t \ v_{N_{\alpha}+N}^n \ell_{N_\alpha + N + \frac{1}{2}} (P_{N_{\alpha}+N+1}^n - P_{N_{\alpha}+N}^n)  + \mathcal{O}(\mbox{size}(\mathcal{T}))  \ .
\end{array}
\end{equation} 
and 
\begin{equation}\label{T22 tilde}
\begin{array}{lll}
\tilde{T}_{22}  = & \sum_{n=1}^{\mathcal{N}} \Delta t \ \ell_{N_{\alpha}+N+\frac{1}{2}} v_{N_{\alpha}+N}^n (P_{N_{\alpha}+N+1}^n - P_{N_{\alpha}+N}^n)  \\
& - \sum_{n=1}^{\mathcal{N}} \Delta t \ \ell_{N_{\alpha}+\frac{1}{2}} v_{N_{\alpha}+1}^n (P_{N_{\alpha}+1}^n - P_{N_{\alpha}}^n) +  \mathcal{O}(\mbox{size}(\mathcal{T}))  \ .
\end{array}
\end{equation}
Similarly, to treat the term $T_3$, we introduce the auxiliary term 
\begin{equation*}
\tilde{T}_3 = - C_3^2 \int_0^T \int_{\beta}^L  w_{\mathcal{T}_m, \Delta t_m} \ P_{xx}(x,t) dxdt  - C_3^2 \int_0^T v_{N_{\alpha}+N, \Delta t_m} P_x(\beta,t) dt.
\end{equation*}
Applying the trace operator as previous and taking into consideration that $ w_{\mathcal{T}_m, \Delta t_m} \to w$ in $L^{\infty}([0,T]; L^2(\beta, L))$, we get
\begin{equation*}\label{convergence of T3 tilde}
\tilde{T}_3 \to - C_3^2 \int_0^T \int_{\beta}^L w(x,t) P_{xx}(x,t)dxdt - C_3^2 \int_0^T w(\beta) P_x(\beta,t) = C_3^2 \int_0^T \int_{\beta}^L w_x(x,t) P_x(x,t) dxdt .
\end{equation*}
Similarly, we reformulate $\tilde{T}_3$ as the following: 
\begin{equation*}
\begin{array}{lll}
\tilde{T}_3 &= -  \sum_{n=1}^{\mathcal{N}} \Delta t \sum_{i=N_{\alpha}+N+1}^{N_{\max}} c_i^2 w_i^n \int_{K_{i}} P_{xx}(x,t_n) dx  - \sum_{n=1}^{\mathcal{N}} \Delta t \ c_{N_{\alpha}+N+1}^2 v_{N_{\alpha}+N}^n P_x(\beta, t_n) \\
&+ C_3^2 \int_0^T v_{N_{\alpha}+N, \Delta t _m} (P_x(\beta, t_n) - P_x(\beta, t)) dt
 - C_3^2 \int_0^T \int_{\beta}^L  w_{\mathcal{T}_m, \Delta t_m} (P_{xx}(x,t) - P_{xx} (x,t_n))dxdt \\
 &=  \tilde{T}_{31} + \tilde{T}_{32} + \tilde{T}_{33} \, 
\end{array}
\end{equation*}
where 
\begin{align*}
\tilde{T}_{31} &=-  \sum_{n=1}^{\mathcal{N}} \Delta t \sum_{i=N_{\alpha}+N+1}^{N_{\max}} c_i^2 w_i^n \int_{K_{i}} P_{xx}(x,t_n) dx \\
\tilde{T}_{32} &=- \sum_{n=1}^{\mathcal{N}} \Delta t \ c_{N_{\alpha}+N+1}^2 v_{N_{\alpha}+N}^n P_x(\beta, t_n) \\
\tilde{T}_{33} &= C_3^2 \int_0^T v_{N_{\alpha}+N, \Delta t _m} (P_x(\beta, t_n) - P_x(\beta, t)) dt
- C_3^2 \int_0^T \int_{\beta}^L  w_{\mathcal{T}_m, \Delta t_m} (P_{xx}(x,t) - P_{xx} (x,t_n))dxdt  \ .
\end{align*}
Since $w_{\mathcal{T}_m, \Delta t_m} $ is bounded in $L^{\infty}([0,T]; L^2(\beta,L))$ and the function $P$ is smooth enough, and making the same discussion for the boundedness of the term $v_{N_{\alpha}+N, \Delta t _m}$, we deduce that $\tilde{T}_{33} \to 0$ as $m \to \infty$. 

Using the same argument and computations as done for the terms $\tilde{T}_{11}$ and $\tilde{T}_{12}$ , we get  
\begin{equation} \label{tilde T31}
\begin{array}{lll}
\tilde{T}_{31} &= \sum_{n=1}^{\mathcal{N}} \Delta t \sum_{N_{\alpha}+N+1}^{N_{\max}}  \ell_{i+\frac{1}{2}}(w_{i+1}^n - w_i^n)(P_{i+1}^n - P_{i}^n) \\
&+   \sum_{n=1}^{\mathcal{N}} \Delta t \ w_{N_{\alpha}+N+1}^n  \ell_{N_{\alpha}+N+\frac{1}{2}}(P_{N_{\alpha}+N+1}^n - P_{N_{\alpha}+N}^n) + \mathcal{O}(\mbox{size}(\mathcal{T})).
\end{array}
\end{equation} 
and 
\begin{equation}\label{T32 tilde}
\begin{array}{lll}
\tilde{T}_{32} &=  - \sum_{n=1}^{\mathcal{N}} \Delta t \ \ell_{N_{\alpha}+N+\frac{1}{2}} v_{N_{\alpha}+N}^n (P_{N_{\alpha}+N+1}^n - P_{N_{\alpha}+N}^n) + \mathcal{O}(\mbox{size}(\mathcal{T})).
\end{array}
\end{equation}
Consequently, 
\begin{equation*}
\tilde{T}_{11} + \tilde{T}_{21} + \tilde{T}_{31} =  T_1 + T_2 + T_3 + \mathcal{O}(\mbox{size}(\mathcal{T}))  \quad {\textrm{and}} \quad \tilde{T}_{12} + \tilde{T}_{22} + \tilde{T}_{32} =  0 + \mathcal{O}(\mbox{size}(\mathcal{T}))
\end{equation*}
and therefore
\begin{equation*}
T_1 + T_2 + T_3 \to C_1^2 \int_0^T \int_{0}^{\alpha} u_x P_x dxdt + C_2^2 \int_0^T \int_{\alpha}^{\beta} v_x P_x dxdt + C_3^2 \int_0^T \int_{\beta}^L w_x P_x dxdt \quad \textrm{as} \ m \to \infty.
\end{equation*}
Finally, to treat the term $T_4$, consider the following term
\begin{equation*}
\tilde{T}_4 = - \delta \int_0^T \int_{\alpha}^{\beta}  \partial^{1/2}v_{\mathcal{T}_m, \Delta t_m} \ P_{xx}(x,t) dxdt  - \delta \int_0^T  \partial^{1/2}v_{N_{\alpha}+1, \Delta t_m} \ P_{x}(\alpha,t) dt+ \delta \int_0^T \partial^{1/2} v_{N_{\alpha}+N, \Delta t_m} P_x(\beta,t) dt \ .
\end{equation*}
Applying the trace operator and the fact that $ \partial^{1/2} v_{\mathcal{T}_m, \Delta t_m} \rightharpoonup^* v_t$ in $ L^{\infty}([0,T]; L^2(\alpha, \beta))$ (due to the stability results), we obtain
\begin{equation*}\label{convergence of T4 tilde}
\tilde{T}_4 \rightharpoonup^* - \delta \int_0^T \int_{\alpha}^{\beta} v_t P_{xx} dxdt + \delta \int_0^T v_t(\beta) P_x(\beta,t) dt - \delta \int_0^T v_t(\alpha) P_x(\alpha,t) dt  = \delta \int_0^T \int_{\alpha}^{\beta} v_{tx} P_x dxdt.
\end{equation*}
Also, we reformulate $\tilde{T}_4$ as the following:
\begin{equation*}
\begin{array}{lll}
\tilde{T}_4 &= - \delta \sum_{n=1}^{\mathcal{N}} \Delta t \sum_{i=N_{\alpha}+1}^{N_{\alpha}+N-1}  \partial^{1/2} v_i^n \int_{K_{i}} P_{xx}(x,t_n) dx  - \delta \sum_{n=1}^{\mathcal{N}} \Delta t   \partial^{1/2} v_{N_{\alpha}+1}^n  P_{x}(\alpha,t_n) + \delta \sum_{n=1}^{\mathcal{N}} \Delta t   \partial^{1/2} v_{N_{\alpha}+N}^n  P_{x}(\beta,t_n)  
\\
&+ \delta \int_0^T \int_{\alpha}^{\beta} \partial^{1/2} v_{\mathcal{T}_m, \Delta t_m} (P_{xx}(x,t_n) - P_{xx} (x,t))dxdt + \delta \int_0^T  \partial^{1/2} v_{N_{\alpha}+1, \Delta t_m} (P_{x}(\alpha,t_n) - P_{x} (\alpha,t))dt \\
& - \delta \int_0^T  \partial^{1/2} v_{N_{\alpha}+N, \Delta t_m} (P_{x}(\beta,t_n) - P_{x} (\beta,t))dt
\\
 &=  \tilde{T}_{41} + \tilde{T}_{42} + \tilde{T}_{43} \ ,
\end{array}
\end{equation*}
where
\begin{align*}
\tilde{T}_{41}  &= - \delta \sum_{n=1}^{\mathcal{N}}\Delta t \sum_{i=N_{\alpha}+1}^{N_{\alpha}+N-1} \partial^{1/2} v_i^n \int_{K_{i}} P_{xx}(x,t_n) dx \\
\tilde{T}_{42} & = - \delta \sum_{n=1}^{\mathcal{N}} \Delta t   \partial^{1/2} v_{N_{\alpha}+1}^n  P_{x}(\alpha,t_n) + \delta \sum_{n=1}^{\mathcal{N}} \Delta t   \partial^{1/2} v_{N_{\alpha}+N}^n  P_{x}(\beta,t_n) \\
\tilde{T}_{43} & = \delta \int_0^T \int_{\alpha}^{\beta} \partial^{1/2} v_{\mathcal{T}_m, \Delta t_m} (P_{xx}(x,t_n) - P_{xx} (x,t))dxdt + \delta \int_0^T  \partial^{1/2} v_{N_{\alpha}+1, \Delta t_m} (P_{x}(\alpha,t_n) - P_{x} (\alpha,t))dt \\
& - \delta \int_0^T  \partial^{1/2} v_{N_{\alpha}+N, \Delta t_m} (P_{x}(\beta,t_n) - P_{x} (\beta,t))dt
\end{align*}
Since $\partial^{1/2} v_{\mathcal{T}_m, \Delta t_m} $ is bounded in $L^{\infty}([0,T]; L^2( \alpha,\beta))$, function $P$ is smooth enough and using the fact that $\partial^{1/2} v_{N_{\alpha}+1}^n$ and $\partial^{1/2} v_{N_{\alpha}+N}^n$ are bounded due to the continuity of the trace operator,  then $\tilde{T}_{43} \to 0$ as $m \to \infty$. 

Using the centered time discretization definition, we can write $\tilde{T}_{41}$ as 
\begin{equation*}
\begin{array}{lll}
\tilde{T}_{41} &= - \delta \sum_{n=1}^{\mathcal{N}}\Delta t \sum_{i=N_{\alpha}+1}^{N_{\alpha}+N-1} \partial^{1/2} v_i^n \left(P_{x}(x_{i+\frac{1}{2}},t_n) - P_{x}(x_{i-\frac{1}{2}},t_n)\right)\\
\end{array}
\end{equation*} 
Since $P_i^n = P(x_i,t_n)$ and since $P$ is regular, a Taylor expansion in the neighborhood of $x_{i+\frac{1}{2}}$ and $x_{i-\frac{1}{2}}$ leads to 
\begin{equation*}
\begin{array}{lll}
\tilde{T}_{41} &=  - \delta \sum_{n=1}^{\mathcal{N}}\Delta t \sum_{i=N_{\alpha}+1}^{N_{\alpha}+N-1}  \frac{v_i^{n+1}-v_i^{n-1}}{2 \Delta t} \bigg(  
\frac{P_{i+1}^n - P_{i}^n}{h_{i+ \frac{1}{2}}} - \frac{P_{i}^n - P_{i-1}^n}{h_{i- \frac{1}{2}}}  \bigg) + \mathcal{O}(\mbox{size}(\mathcal{T})) \ .
\end{array}
\end{equation*} 
Using translation of index $i$ for the second term of the above equation, we obtain 
\begin{equation*}\label{T41 tilde}
\begin{array}{lll}
\tilde{T}_{41} &= \delta \sum_{n=1}^{\mathcal{N}} \Delta t \sum_{i=N_{\alpha}+1}^{N_{\alpha}+N-1} \frac{(v_{i+1}^{n+1} - v_i^{n+1} - v_{i+1}^{n-1} + v_i^{n-1} )(P_{i+1}^n - P_{i}^n)}{2 \Delta t \ {h_{i+ \frac{1}{2}}}} \\
&+ \delta \sum_{n=1}^{\mathcal{N}} \Delta t \  \frac{(v_{N_{\alpha}+1}^{n+1} - v_{N_{\alpha}+1}^{n-1} )(P_{N_{\alpha}+1}^n - P_{N_{\alpha}}^n)}{2 \Delta t \ h_{N_{\alpha}+ \frac{1}{2}}} \\
& - \delta  \sum_{n=1}^{\mathcal{N}} \Delta t \ \frac{(v_{N_{\alpha}+N}^{n+1} - v_{N_{\alpha}+N}^{n-1} )(P_{N_{\alpha}+N}^n - P_{N_{\alpha}+N-1}^n)}{2 \Delta t \ h_{N_{\alpha}+N+\frac{1}{2}}}
+ \mathcal{O}(\mbox{size}(\mathcal{T})) \ . \\
& = T_4 +  \delta  \sum_{n=1}^{\mathcal{N}}\Delta t \ \frac{(v_{N_{\alpha}+1}^{n+1} - v_{N_{\alpha}+1}^{n-1} )(P_{N_{\alpha}+1}^n - P_{N_{\alpha}}^n)}{2 \Delta t \ h_{N_{\alpha}+\frac{1}{2}}} \\
& - \delta \sum_{n=1}^{\mathcal{N}}\Delta t \ \frac{(v_{N_{\alpha}+N}^{n+1} - v_{N_{\alpha}+N}^{n-1} )(P_{N_{\alpha}+N}^n - P_{N_{\alpha}+N-1}^n)}{2 \Delta t \ h_{N_{\alpha}+N+\frac{1}{2}}}
+ \mathcal{O}(\mbox{size}(\mathcal{T})) \ .
\end{array}
\end{equation*}
In addition, using again the centered time discretization definition and since $P$ is regular, a Taylor expansion in the neighborhood of $x_{N_\alpha+\frac{1}{2}} = \alpha $ and $x_{N_\alpha+N + \frac{1}{2}} = \beta$ leads to
\begin{equation*}
\begin{array}{lll}
\tilde{T}_{42} & =  -  \delta  \sum_{n=1}^{\mathcal{N}}\Delta t \ \frac{(v_{N_{\alpha}+1}^{n+1} - v_{N_{\alpha}+1}^{n-1} )(P_{N_{\alpha}+1}^n - P_{N_{\alpha}}^n)}{2 \Delta t \ h_{N_{\alpha}+\frac{1}{2}}} \\
& + \delta \sum_{n=1}^{\mathcal{N}}\Delta t \ \frac{(v_{N_{\alpha}+N}^{n+1} - v_{N_{\alpha}+N}^{n-1} )(P_{N_{\alpha}+N}^n - P_{N_{\alpha}+N-1}^n)}{2 \Delta t \ h_{N_{\alpha}+N+\frac{1}{2}}} + \mathcal{O}(\mbox{size}(\mathcal{T})) \ .
\end{array}
\end{equation*} 
Consequently as $\tilde{T}_{4} = T_4 +  \mathcal{O}(\mbox{size}(\mathcal{T}))$, we get 
\begin{equation*}
T_4 \rightarrow\delta \int_0^T \int_{\alpha}^{\beta} v_{tx} P_x dxdt \quad \textrm{as} \ m \to \infty.
\end{equation*}
Therefore, for the limit solution $U=(u,v,w)$, the variational problem \eqref{continuous weak variational problem} holds for any  arbitrary function $P \in C^{\infty}([0,T]; C_0^{\infty}(0,L))$ regarding that $P(T,x) = 0$.

 By density argument, Problem \eqref{continuous weak variational problem} holds for $(p,q,z) \in L^2([0,T]; \mathbb{H}^1_L)$ with $(p_t, q_t,z_t) \in L^2 ([0,T]; \mathbb{L}^2)$ such that $p(T,x)= q (T,x) = z(T,x)=0$. The proof of Theorem \eqref{Convergence theorem} is thus complete.
\end{proof}
\section{Numerical Experiments: Validation of the theoretical results} \label{Numerical Experiments}
\noindent In this section, we present some examples to illustrate graphically the theoretical results obtained in \cite{GhNAWE}  where the authors studied the asymptotic behavior of the energy of the continuous case of system \eqref{Eq(2.1)}-\eqref{Eq(2.5)} and obtained an exponential decay when the three speeds are equal and an optimal energy decay rate of type $t^{-4}$ when they are different. 

In every numerical experiment, we test the explicit scheme \eqref{Discrete Explicit} and suppose $L=3, \ \alpha=1, \ \beta=2$ and the final time $T = 10\;000$. The discretization is given by $N_{\alpha}=20, \ N= 10, \ N_{\beta}=20$ and $\mathcal{N}= 400\;000$. Consequently, $h_{\alpha}= 0.05, \ h = 0.1, \ h_{\beta} = 0.05$ and the time step is chosen as $\Delta t = \frac{T}{\mathcal{N}} = \textrm{CFL} \times \Delta x$. We will study the asymptotic behavior for the following initial conditions of the form 
$$U_0(x) = \frac{4}{L^2} x (L-x), \ \Psi(x) = -\frac{4}{L^2} x (L-x).$$
\begin{Remark}
We have also implemented and tested the implicit scheme with the same parameters and the same discretization. The decay of the solution is exactly the same that is the exponential stability is observed with exactly the same rate
and also the polynomial stability. As in this research project, we are interested in the decay rate of the solution of a non regular damped wave equation, we present the explicit scheme. 

In a forgoing work, we will concentrate our attention in the numerical verification of the convergence with respect to the spatial mesh size and the
time step, the computation of a numerical order of discretization in space and also in the comparison between the two schemes.
\end{Remark}
\subsection{Equal speed of propagation}
\noindent  Let $C^2_1 = C_2^2 = C_3^2 = 1$.\\[0.1in]
\noindent \textbf{Case 1. No damping: conservation of the total energy.} When $\delta=0$, Figure \ref{fig 9} shows that the total energy is conserved along time. Thus, this numerical test shows that in the absence of the damping term, the total energy is completely conserved. Consequently, the numerical scheme \eqref{Discrete Explicit} does not produce any numerical dissipation and therefore the numerical behavior observed is only due to the considered model. 
Indeed, in the case of different propagation speed,  numerical tests show also conservation of the energy in the absence of the damping, {\it{i.e.}}, when $\delta=0$ (see for instance Figure \ref{fig 9}). 

\noindent \textbf{Case 2. Exponential stability.} When $\delta=1$, Figure \ref{fig 4}-a shows that the total energy goes to zero as $t \to \infty$, and thus we have dissipation. First, for a graph in $\ln$-$\ln$ scale, there is a an energy decay of order $t^{-\alpha}$, whereas, for a graph of $\ln$-time scale, we obtain an energy decay of order $e^{- \omega t}$. By simple linear regression (least squares), Figure \ref{fig 4}-c shows that the polynomial decay of the energy is very fast ($\alpha \simeq 54.8760$) and thus we deduce that the energy decays faster than polynomial. In fact, we can observe from Figure \ref{fig 4}-b an exponential decay, as the graph of $-\ln (E(t))$ versus time $t$ plots a straight line. By simple linear regression (least squares), we obtain the rate of decay numerically ($\omega \simeq 0.0008$) and it is found to be very small. However, the final time profile confirms that $U$ is small but it shows that high frequencies are not completely controlled (see Figure \ref{fig 4}-d). \\[0.1in]
\subsection{Different speed of propagation}
\noindent  In this part, we aim to verify the theoretical results obtained in \cite{GhNAWE}. To this end, we consider several cases and we  obtain the decay rates numerically by simple linear regression (least squares). \\[0.1in]
\noindent \textbf{Case 1. ($C_1 > C_3 > C_2$) Polynomial stability.} Set $C_1^2=9, \ C_2^2 = 1, \ C_3^2 = 4$ and $\delta=1$. Since the final time is large, we show two figures for the energy, the first one is for $t=0$ to $t=100$ and the second one represents the energy from $T/2$ till $T$. Indeed, Figure \ref{fig 5} shows that the total energy is decreasing and goes finally to zero, and thus we look for an exponential or polynomial decay. The graph of $-\ln(E(t))$ versus $t$ plots a curve of the decay of the energy with a very small coefficient ($\omega = 0.0001$) which shows that the energy decays slower than exponential (see Figure \ref{fig 5}-c). However, Figure \ref{fig 5}-d plots the graph of $-\ln(E(t))$ versus $\ln(t)$ which permits to show that $E(t)$ tends to zero as $1/t^{\alpha}$  with $\alpha \simeq 4.17$, since the curve is asymptotically a straight line. Finally, the final time profile confirms that $U$ is small but also shows that high frequencies are not completely controlled (see Figure \ref{fig 5}-e). \\[0.1in]
\noindent \textbf{Case 2. ($C_2 > C_1 > C_3$) Polynomial stability.} Set $C_1^2=2, \ C_2^2 = 4, \ C_3^2 = 0.25$ and $\delta=1$. Similarly, since the final time is large, we show two figures for the energy, the first one is for $t=0$ to $t=100$ and the second one represents the energy from $T/2$ till $T$. Under the same discussion, Figure \ref{fig 6} shows that $E(t)$ tends to zero as $1/t^{\alpha}$  with $\alpha \simeq 4.4608$. Also, the final time profile confirms that $U$ is small, although the high frequencies are not completely controlled, especially at the beginning of the experiment. \\[0.1in]
\noindent \textbf{Case 3. ($C_3 > C_2 > C_1$) Polynomial stability.} Set $C_1^2=2, \ C_2^2 = 4, \ C_3^2 = 6$ and $\delta=1$. Figure \ref{fig 7} shows that the total energy goes finally to zero. Therefore, to explore the speed of convergence to zero, we plot similarly the graph of $-\ln(E(t))$ versus time and the graph of $-\ln(E(t))$ versus $\ln(t)$. Figure \ref{fig 7}-c permits to show that the  energy decays slower than exponential ($\omega=0.0005$), and hence we look for a polynomial decay. As a result, $E(t)$ tends to zero as $1/t^{\alpha}$  with $\alpha \simeq 3.3990$ (see Figure \ref{fig 7}-d where the graph is asymptotically a straight line). Finally, the final time profile confirms that $U$ is small but it shows that high frequencies are not completely controlled. \\[0.1in]
\noindent \textbf{Case 4. ($C_1 = C_3$) Polynomial stability.} Set $C_1^2 = C_3^2 =2, \ C_2^2 = 4$ and $\delta=1$. Since Figure \ref{fig 8} shows that the total energy goes finally to zero, then we intent to study the speed of convergence to zero. Using the same argument,  Figure \ref{fig 8}-c permits to show that the  energy decays slower than exponential ($\omega=0.0006$), while Figure \ref{fig 8}-d states that $E(t)$ tends to zero as $1/t^{\alpha}$  with $\alpha \simeq 4.4350$. Finally, the final time profile confirms that although $U$ is small, high frequencies are not completely controlled. 
\begin{Remark}
When the propagation speeds are not equal, we obtain a polynomial decay of the energy with slight different rates. However, in some cases, we get a numerical polynomial convergence better than $t^{-4},$ but it will probably be $t^{-4}$ if we increase the time. In this test, we do not perform very long simulation to confirm, for reason of computation time.
\end{Remark}
\begin{Remark}
Under the same speed of propagation ($C_1^2 = C_2^2 = C_3^2=1$), the exponential decay rate depends on the size of the domain where the Kelvin-Voigt damping is acting. In fact, let $\alpha =0.1, \ \beta = 2.9$ and choose $N_{\alpha} = N_{\beta} =4, \ N= 100$ so that $h_{\alpha} = h_{\beta} = 0.025, \ h=0.028$. Also, set $T=100$ and $\mathcal{N} = 4000$ such that $\Delta t= \frac{T}{\mathcal{N}}= \textrm{CFL} \times \Delta x$. Figure \ref{fig 10} shows that $E(t)$ decays exponentially to zero with $\omega =  0.4312$, and thus, the exponential decreasing rate increases as we add more viscoelastic material. Moreover, the final time profile shows that the high frequency oscillations are exponentially dissipated after a while. 
\end{Remark}
\section{Conclusions}
As announced, we have constructed a finite volume (explicit or implicit) space discretisation for the transmission problem of a 1-D wave equation with non smooth wave speed an localized Kelvin-Voigt damping. We have proved that the discrete solution converges to the continuous weak solution when
the discretisation step size goes to zero. Moreover, from the numerical experiments, we have also proved that the decay rate of the discrete solution
towards the null solution when damping acts and the speeds are equal or different is  in compliance with the theoretical study performed by two 
of the authors in \cite{GhNAWE}.

In the future, the investigation of the order of the space discretisation could also be performed. We expect a first order discretisation in 
space since the elliptic part is of first order.
It could be done by a careful numerical analysis but also by performing numerical experiments.

As a complete study of the stability of N-D transmission problem in viscoelasticity with localized Kelvin-Voigt damping under different types of geometric conditions has been performed by two
of the authors in an open subset of $\R^N$ \cite{Nasser-Noun-Wehbe-2020}, the authors worked actually on the numerical approximation by a  finite volume method 
of a such transmission problem in 2D to confirm again the theoretical result on the type of decay of the solution : exponential or polynomial.

\section*{Acknowledgments}
The authors wish to express their gratitude to the anonymous reviewer for his meticulous review and perceptive comments, which were instrumental in enhancing the clarity and rigor of this article.

\textbf{Conflicts of interest : } The authors have no conflicts of interest to declare that are relevant to the content of this article.

\bibliographystyle{siam}

\begin{thebibliography}{}

\end{thebibliography}


\begin{thebibliography}{10}
	
	\bibitem{Alves-Rivera-Sepulveda-Villagran-2014}
	{\sc M.~Alves, J.~M. Rivera, M.~Sep{\'{u}}lveda, and O.~V. Villagr{\'{a}}n},
	{\em \href {https://doi.org/10.1137/130923233}{The Lack of Exponential
			Stability in Certain Transmission Problems with Localized {K}elvin--{V}oigt
			Dissipation}}, {SIAM} Journal on Applied Mathematics, 74 (2014),
	pp.~345--365.
	
	\bibitem{Alves-Rivera-Sepulveda-Villagran-Garay-2013}
	{\sc M.~Alves, J.~M. Rivera, M.~Sep{\'{u}}lveda, O.~V. Villagr{\'{a}}n, and
		M.~Z. Garay}, {\em \href{https://doi.org/10.1002/mana.201200319}{The
			asymptotic behavior of the linear transmission problem in viscoelasticity}},
	Mathematische Nachrichten, 287 (2013), pp.~483--497.
	
	\bibitem{Ammari-Hassine-Robbiano-2019}
	{\sc K.~Ammari, F.~Hassine, and L.~Robbiano}, {\em Stabilization for the wave
		equation with singular {K}elvin-{V}oigt damping}, Archive for Rational
	Mechanics and Analysis, 236 (2019), pp.~577--601.
	
	\bibitem{Ammari-Liu-Shel-2019}
	{\sc K.~Ammari, Z.~Liu, and F.~Shel}, {\em
		\href{https://doi.org/10.1007/s00233-019-10064-7} {Stability of the wave
			equations on a tree with local {K}elvin--{V}oigt damping}}, Semigroup Forum,
	(2019).
	
	\bibitem{ERG2000}
	{\sc R.~Eymard, T.~Gallou\"{e}t, and R.~Herbin}, {\em Finite volume methods},
	in Handbook of numerical analysis, {V}ol. {VII}, vol.~VII of Handb. Numer.
	Anal., North-Holland, Amsterdam, 2000, pp.~713--1020.
	
	\bibitem{EHM}
	{\sc R.~Eymard, A.~Handlovicov{\'a}, and K.~Mikula}, {\em Study of a finite
		volume scheme for the regularized mean curvature flow level set equation},
	Ima Journal of Numerical Analysis, 31 (2011), pp.~813--846.
	
	\bibitem{Gao2007}
	{\sc F.~Gao and C.-M. Chi}, {\em
		\href{https://doi.org/10.1016/j.amc.2006.09.057} {Unconditionally stable
			difference schemes for a one-space-dimensional linear hyperbolic equation}},
	Applied Mathematics and Computation, 187 (2007), pp.~1272--1276.
	
	\bibitem{Ghader-Nasser-Wehbe-delay}
	{\sc M.~Ghader, R.~Nasser, and A.~Wehbe}, {\em
		\href{{https://app.dimensions.ai/details/publication/pub.1132411878}}
		{Stability results for an elastic–viscoelastic wave equation with localized
			{K}elvin-{V}oigt damping and with an internal or boundary time delay}},
	Asymptotic Analysis,  (2020), pp.~1--57.
	
	\bibitem{GhNAWE}
	\leavevmode\vrule height 2pt depth -1.6pt width 23pt, {\em Optimal polynomial
		stability of a string with locally distributed {K}elvin-{V}oigt damping and
		nonsmooth coefficient at the interface}, Math. Methods Appl. Sci., 44 (2021),
	pp.~2096--2110.
	
	\bibitem{Golub-Meurant}
	{\sc G.~H. Golub and G.~A. Meurant}, {\em R\'{e}solution num\'{e}rique des
		grands syst\`emes lin\'{e}aires}, vol.~49 of Collection de la Direction des
	\'{E}tudes et Recherches d'\'{E}lectricit\'{e} de France [Collection of the
	Department of Studies and Research of \'{E}lectricit\'{e} de France],
	\'{E}ditions Eyrolles, Paris, 1983.
	
	\bibitem{HamiltonB}
	{\sc B.~Hamilton}, {\em Finite Difference and Finite Volume Methods for
		Wave-based Modelling of Room Acoustics}, PhD thesis, University of Edinburgh,
	2016.
	
	\bibitem{Larsson1991}
	{\sc S.~Larsson, V.~Thom\'{e}e, and L.~B. Wahlbin}, {\em Finite-element methods
		for a strongly damped wave equation}, IMA Journal of Numerical Analysis, 11
	(1991), pp.~115--142.
	
	\bibitem{Leveque}
	{\sc R.~J. LeVeque}, {\em Finite-Volume Methods for Hyperbolic Problems},
	Cambridge University Press, 2002.
	
	\bibitem{Maryati-Rivera-Rambaud-2018}
	{\sc T.~Maryati, J.~Mu{\~n}oz~Rivera, A.~Rambaud, and O.~Vera}, {\em
		\href{http://www.scopus.com/inward/record.url?eid=2-s2.0-85049553195&partnerID=MN8TOARS}{Stability
			of an N-component timoshenko beam with localized {K}elvin--{V}oigt and
			frictional dissipation}}, Electronic Journal of Differential Equations, 2018
	(2018).
	
	\bibitem{Raposo-Bastos-Avila-2011}
	{\sc C.~Raposo, W.~Bastos, and J.~Avila}, {\em \href{
			http://www.naturalspublishing.com/Article.asp?ArtcID=91} {A Transmission
			Problem for Euler-Bernoulli beam with {K}elvin-{V}oigt Damping}}, Applied
	Mathematics and Information Sciences, 5 (2011), pp.~17--28.
	
	\bibitem{riecanova2018}
	{\sc I.~Rie\v{c}anov\'{a} and A.~Handlovi\v{c}ov\'{a}}, {\em Study of the
		numerical solution to the wave equation}, Acta Mathematica Universitatis
	Comenianae. New Series, 87 (2018), pp.~317--332.
	
	\bibitem{Rincon2013}
	{\sc M.~A. Rincon and M.~Copetti}, {\em
		\href{https://doi.org/10.11948/2013013}{Numerical analysis for a locally
			damped wave equation}}, Journal of Applied Analysis and Computation, 3
	(2013), pp.~169--182.
	
	\bibitem{Rivera-Villagran-Sepulveda-2018}
	{\sc J.~E.~M. Rivera, O.~V. Villagran, and M.~Sepulveda}, {\em
		\href{https://doi.org/10.1080/03605302.2018.1475490} {Stability to localized
			viscoelastic transmission problem}}, Communications in Partial Differential
	Equations, 43 (2018), pp.~821--838.
	
	\bibitem{SATO1994}
	{\sc T.~Sato and S.~M. Richardson}, {\em
		\href{http://www.sciencedirect.com/science/article/pii/0377025794850194}{Explicit
			numerical simulation of time-dependent viscoelastic flow problems by a finite
			element/finite volume method}}, Journal of Non-Newtonian Fluid Mechanics, 51
	(1994), pp.~249 -- 275.
	
	\bibitem{zuazua2003}
	{\sc L.~Tebou and E.~Zuazua}, {\em
		\href{https://doi.org/10.1007/s00211-002-0442-9}{Uniform exponential long
			time decay for the space semi-discretization of a locally damped wave
			equation via an artificial numerical viscosity}}, Numerische Mathematik, 95
	(2003), pp.~563--598.
	
	\bibitem{Wang2016}
	{\sc S.~Wang, K.~Virta, and G.~Kreiss}, {\em High order finite difference
		methods for the wave equation with non-conforming grid interfaces}, Journal
	of Scientific Computing, 68 (2016), pp.~1002--1028.
	
	\bibitem{Nasser-Noun-Wehbe-2020}
	{\sc A.~Wehbe, , R.~Nasser, and N.~Noun}, {\em
		\href{https://doi.org/10.3934/mcrf.2020050} {Stability of N-D transmission
			problem in viscoelasticity with localized {K}elvin-{V}oigt damping under
			different types of geometric conditions}}, Mathematical Control {\&} Related
	Fields, 0 (2019), pp.~0--0.
	
	\bibitem{Xin2014}
	{\sc X.~Yu, Z.~Ren, Q.~Zhang, and C.~Xu}, {\em
		\href{https://doi.org/10.1155/2014/982574}{A Numerical Method of the
			Euler-Bernoulli Beam with Optimal Local {K}elvin-{V}oigt Damping}}, Journal
	of Applied Mathematics, 2014 (2014), pp.~1--7.
	
	\bibitem{WZhang2016}
	{\sc W.~Zhang, Y.~Zhuang, and E.~T. Chung}, {\em
		\href{https://doi.org/10.1093/gji/ggw148}{A new spectral finite volume method
			for elastic wave modelling on unstructured meshes}}, Geophysical Journal
	International, 206 (2016), pp.~292--307.
	
	\bibitem{ZHANG2017}
	{\sc W.~Zhang, Y.~Zhuang, and L.~Zhang}, {\em
		\href{http://www.sciencedirect.com/science/article/pii/S0021999117302516}{A
			new high-order finite volume method for 3D elastic wave simulation on
			unstructured meshes}}, Journal of Computational Physics, 340 (2017), pp.~534
	-- 555.
	
\end{thebibliography}

\newpage
\begin{figure}[!tbp]
  \centering
  \begin{minipage}[b]{0.45\textwidth}
    \includegraphics[width=\textwidth]{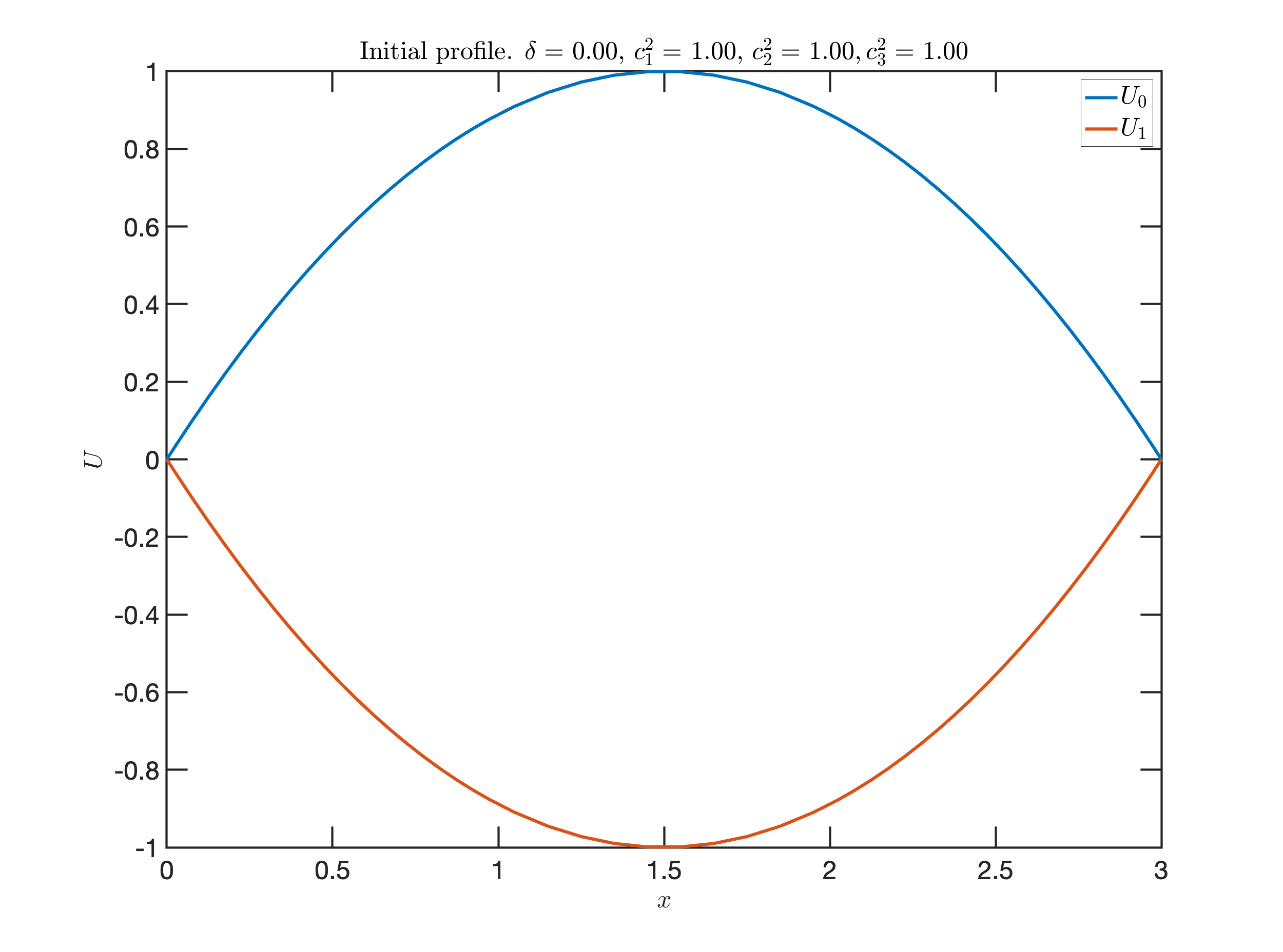}
    \subcaption{Initial profile} \label{fig 2}
  \end{minipage}
  \hfill
  \begin{minipage}[b]{0.45\textwidth}
    \includegraphics[width=\textwidth]{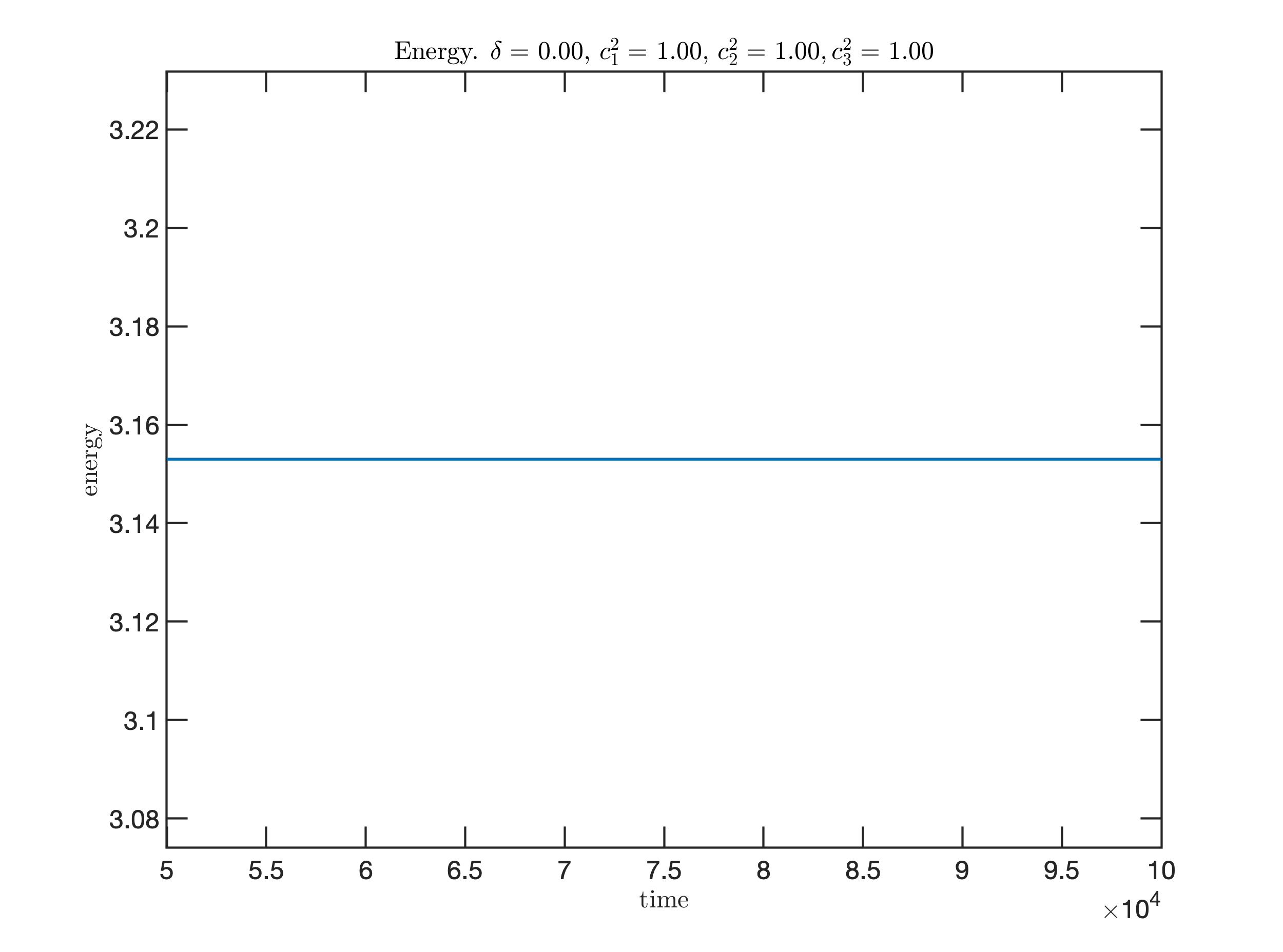}
    \subcaption{No damping/equal speed} \label{fig 3}
  \end{minipage}
  \hfill
  \begin{minipage}[b]{0.55\textwidth}
\includegraphics[width=\textwidth]{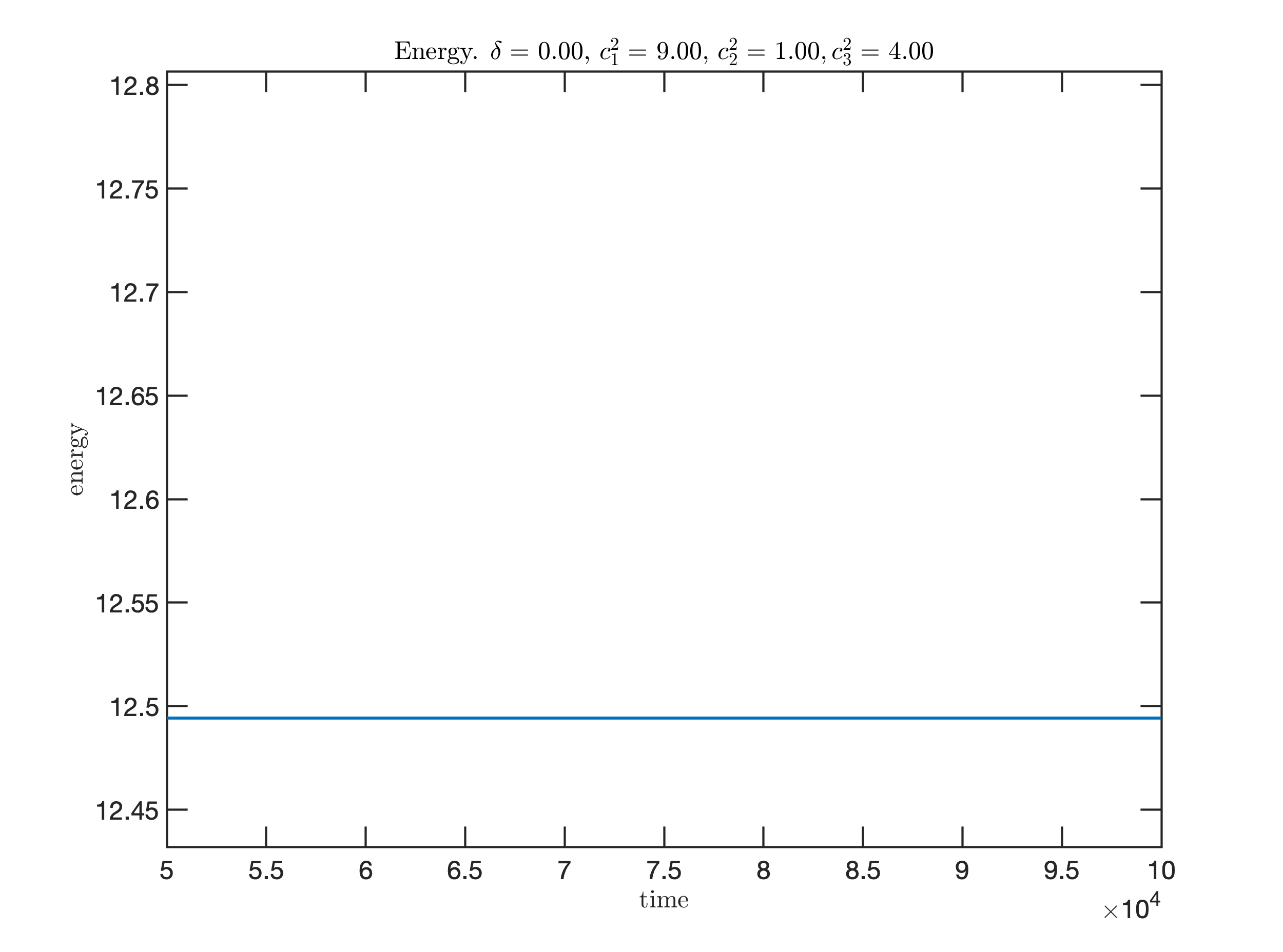}
    \subcaption{No damping/different speed}
    \end{minipage}    
    \caption{No damping}  \label{fig 9}
\end{figure}
\begin{figure}[!tbp]
  \centering
  \begin{minipage}[b]{0.45\textwidth}
    \includegraphics[width=\textwidth]{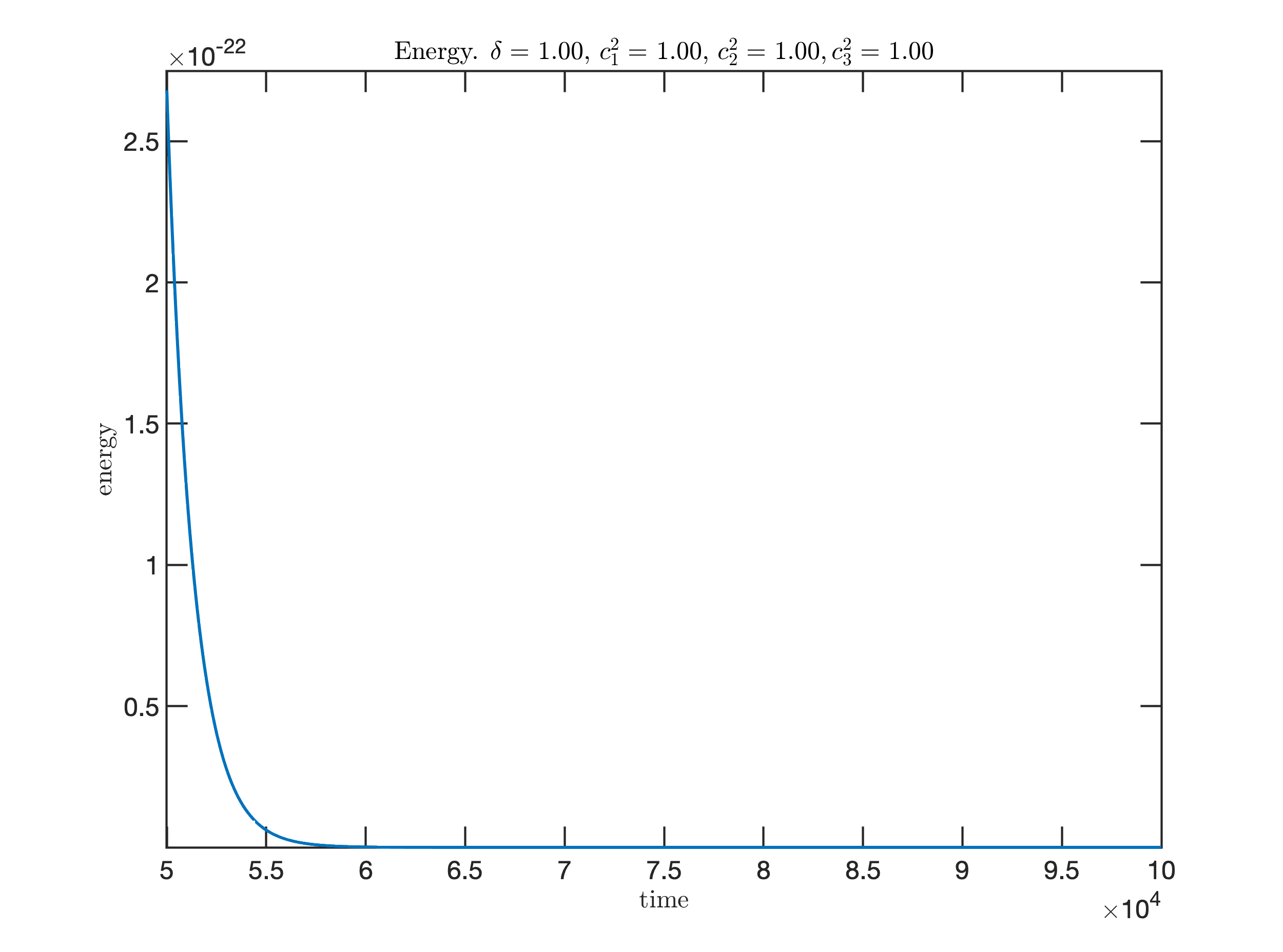}
    \subcaption{Energy} 
  \end{minipage}
  \hfill
  \begin{minipage}[b]{0.45\textwidth}
    \includegraphics[width=\textwidth]{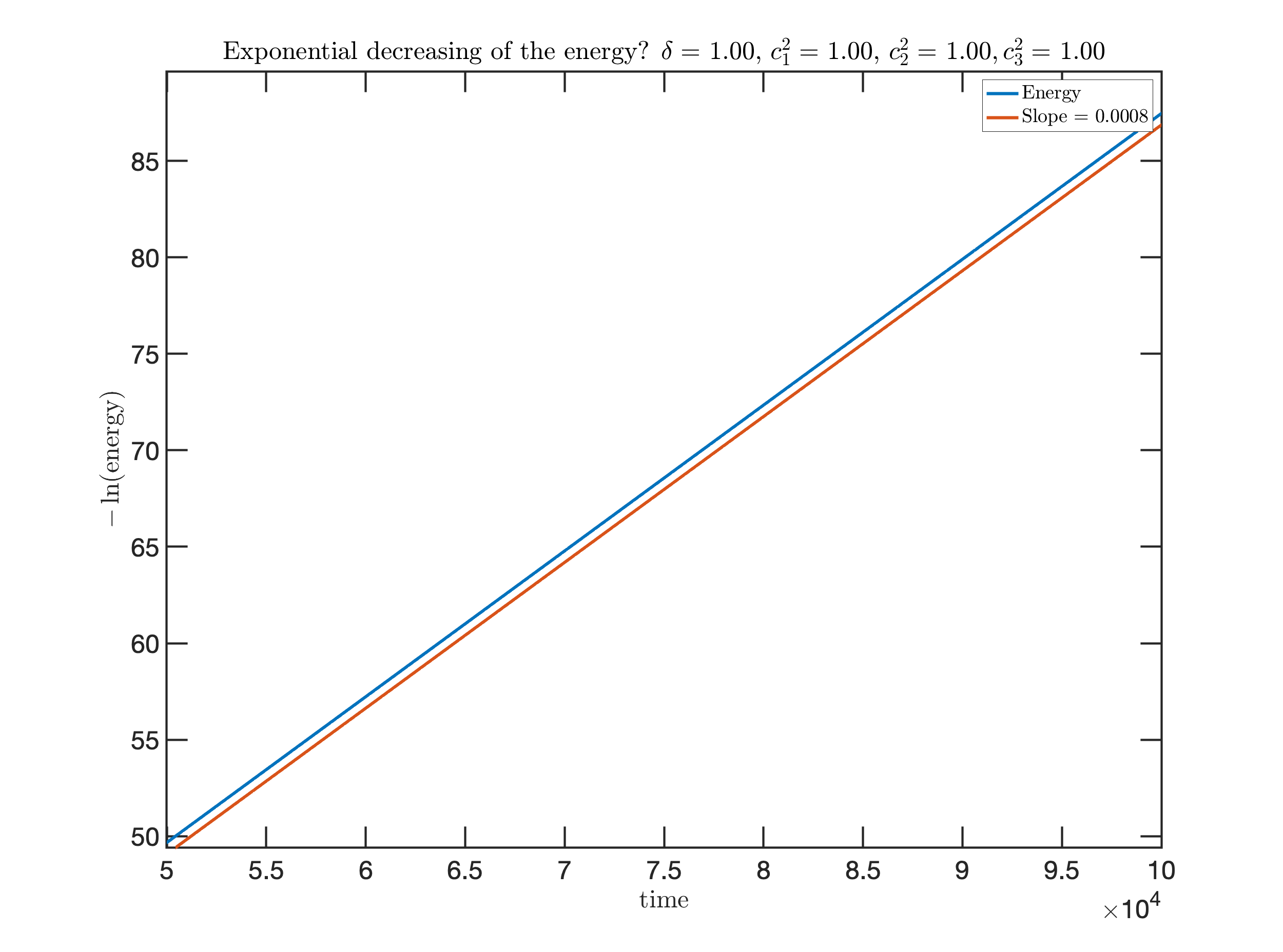}
    \subcaption{Exponential decay} 
  \end{minipage}
  \hfill
  \begin{minipage}[b]{0.45\textwidth}
    \includegraphics[width=\textwidth]{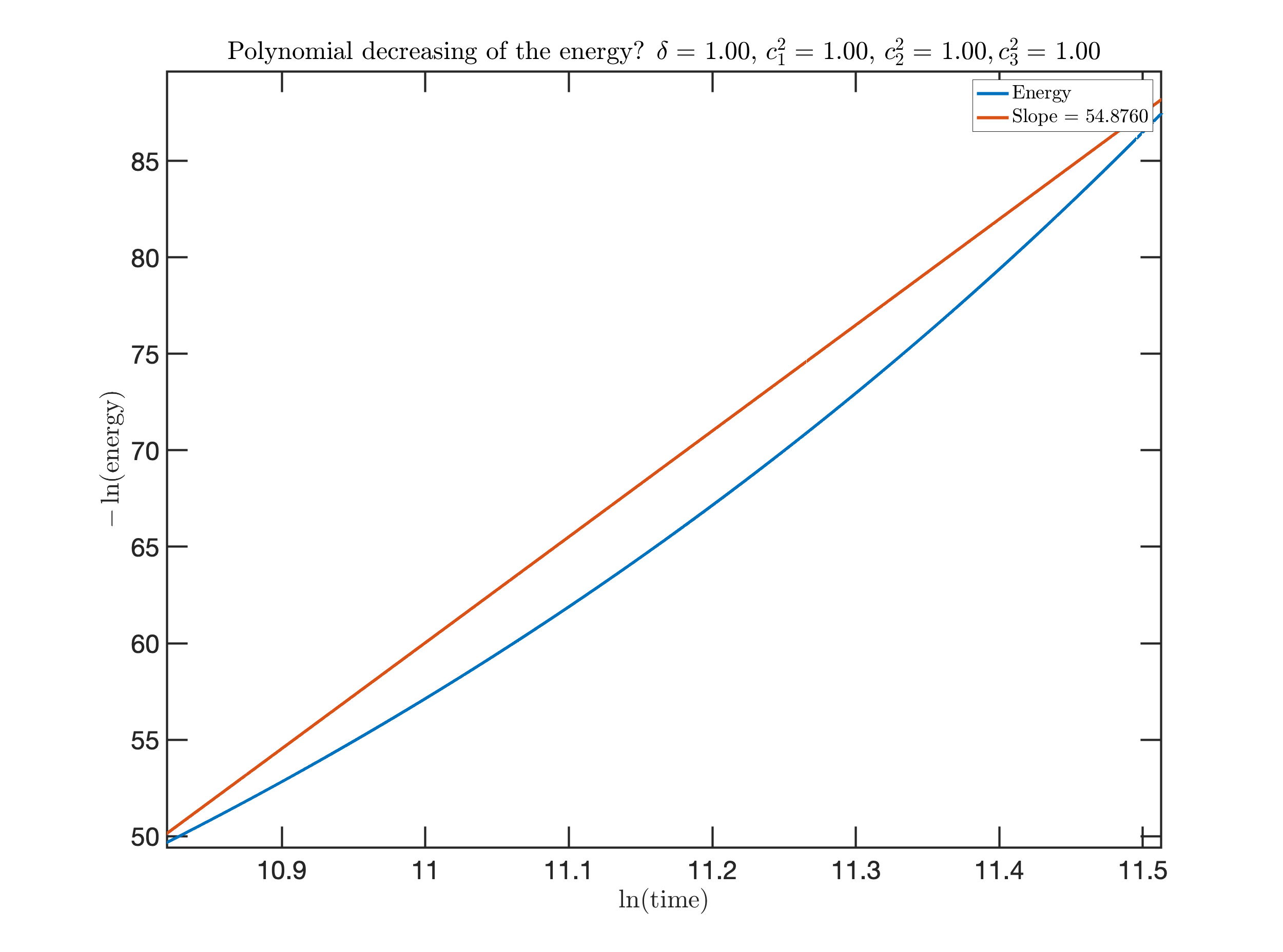}
    \subcaption{Polynomial decay} 
  \end{minipage}
  \hfill
  \begin{minipage}[b]{0.45\textwidth}
    \includegraphics[width=\textwidth]{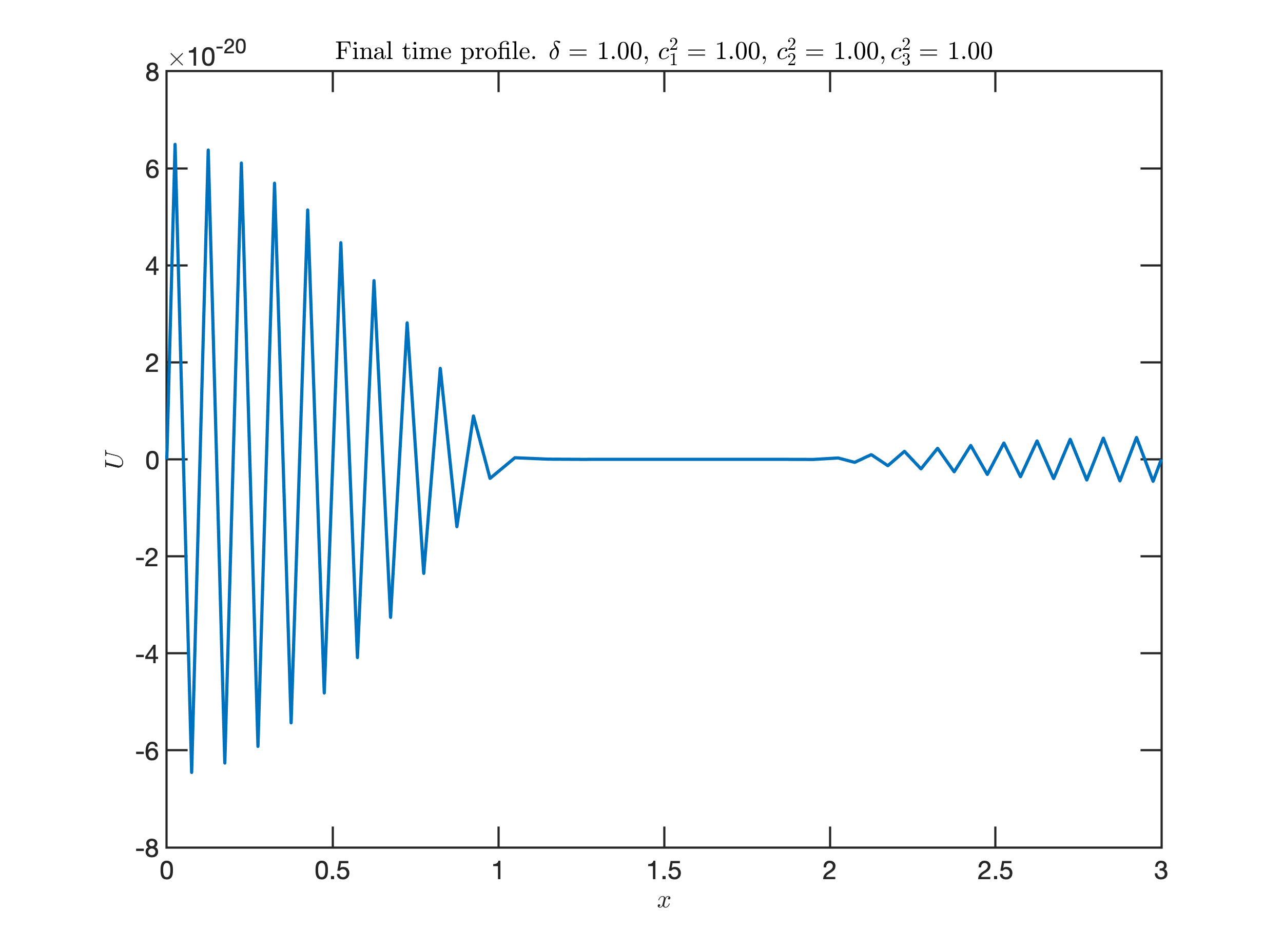}
    \subcaption{Final time profile} 
  \end{minipage}
  \caption{Long time behavior when $C_1^2 = C_2^2 = C_3^2=1$} \label{fig 4}
\end{figure}
\begin{figure}[!tbp]
  \centering
  \begin{minipage}[b]{0.45\textwidth}
    \includegraphics[width=\textwidth]{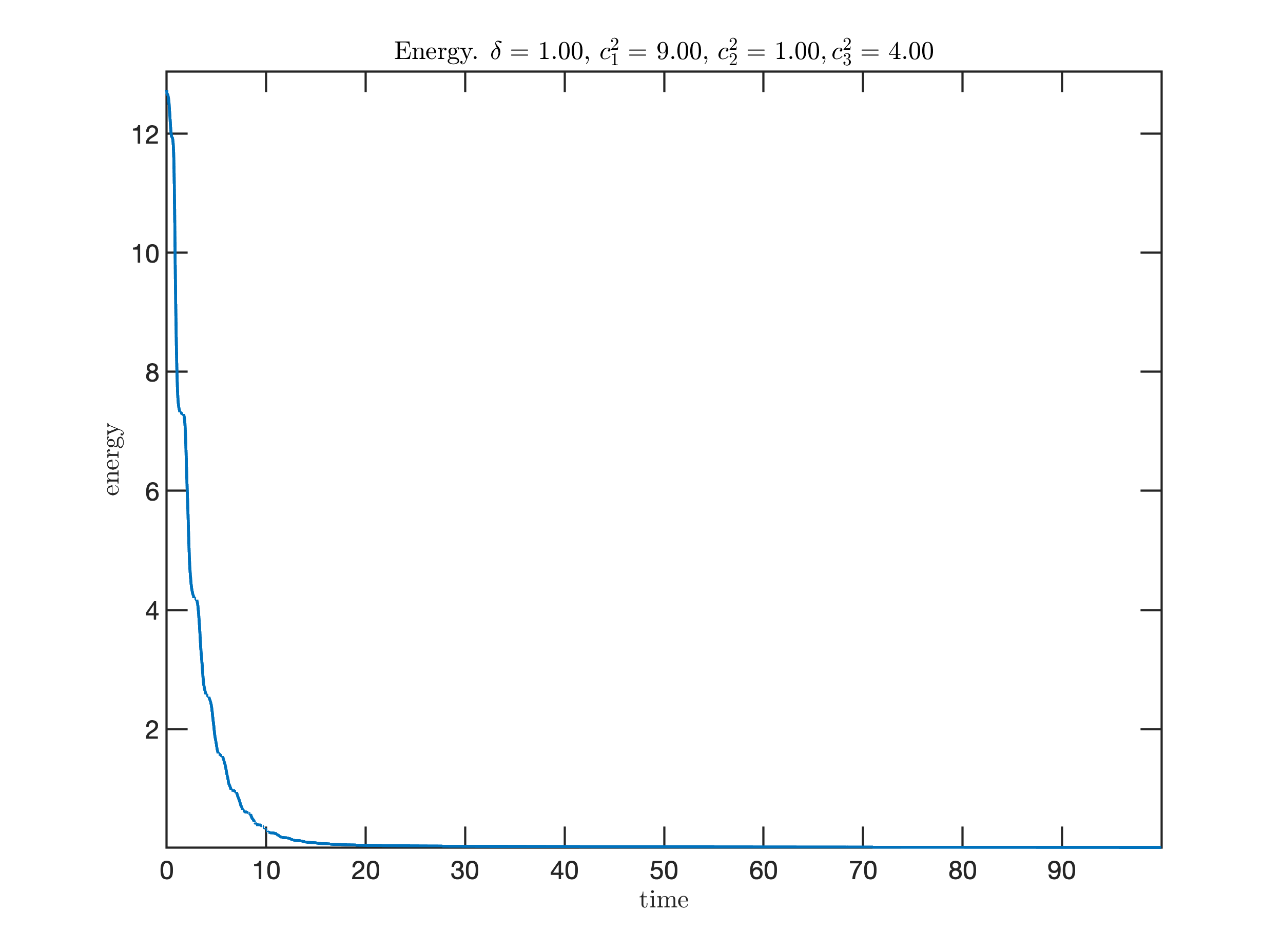}
    \subcaption{Energy ($t=0$ till $t=100$)} 
  \end{minipage}
  \hfill
  \begin{minipage}[b]{0.45\textwidth}
    \includegraphics[width=\textwidth]{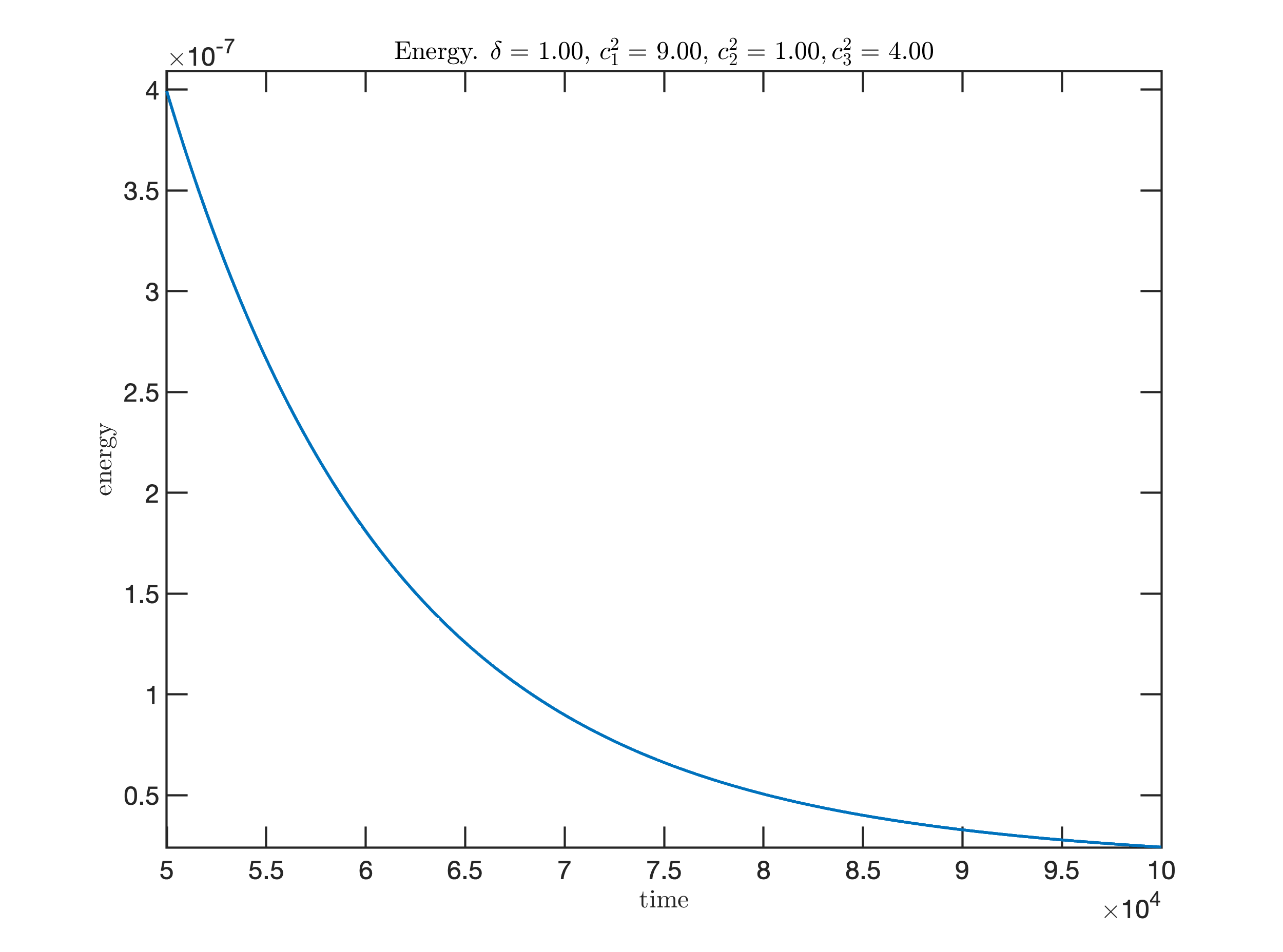}
    \subcaption{ Energy ($T/2$ till $T$)} 
  \end{minipage}
  \hfill
  \begin{minipage}[b]{0.45\textwidth}
    \includegraphics[width=\textwidth]{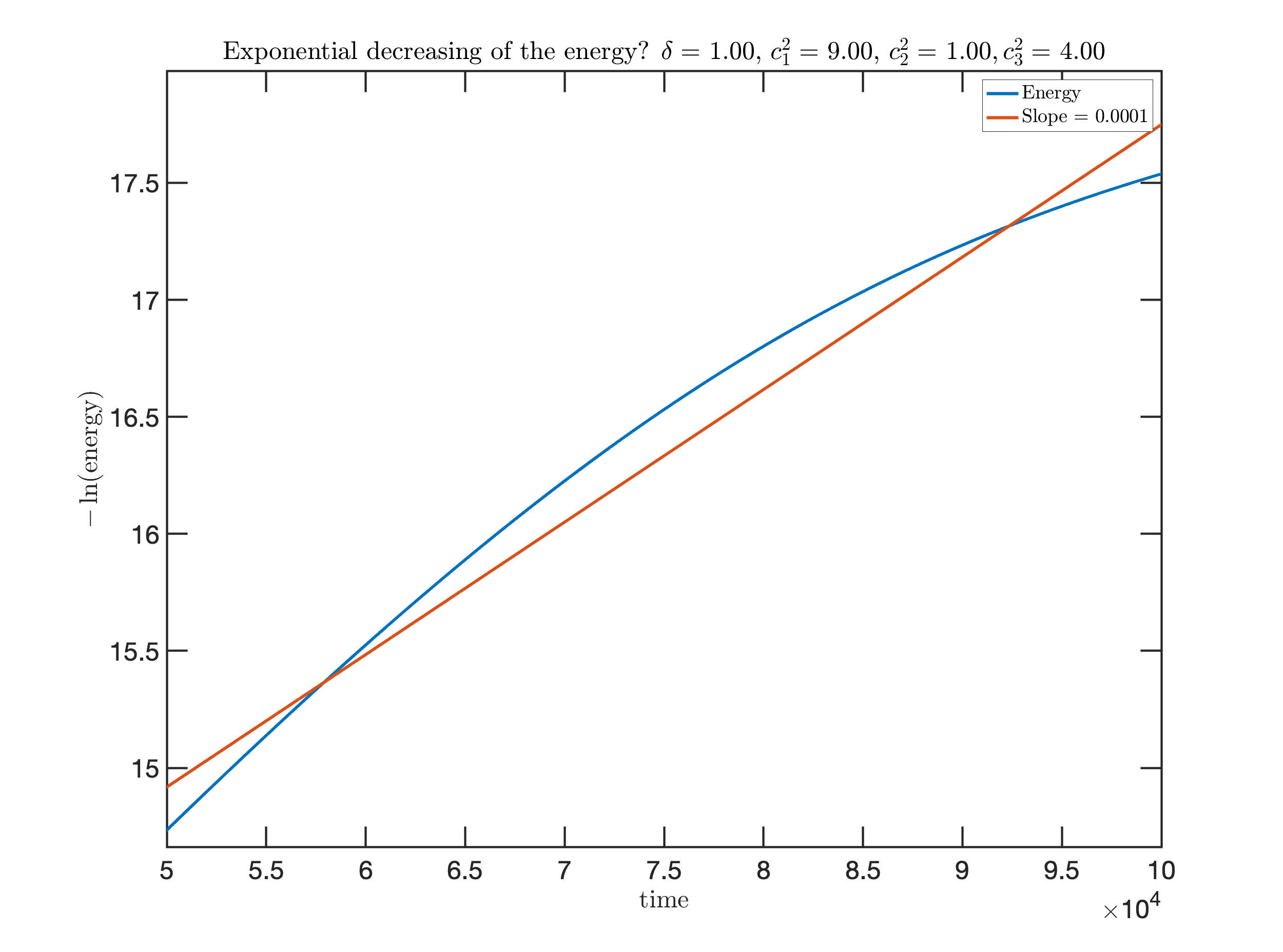}
    \subcaption{Exponential decay} 
  \end{minipage}
  \hfill
  \begin{minipage}[b]{0.45\textwidth}
    \includegraphics[width=\textwidth]{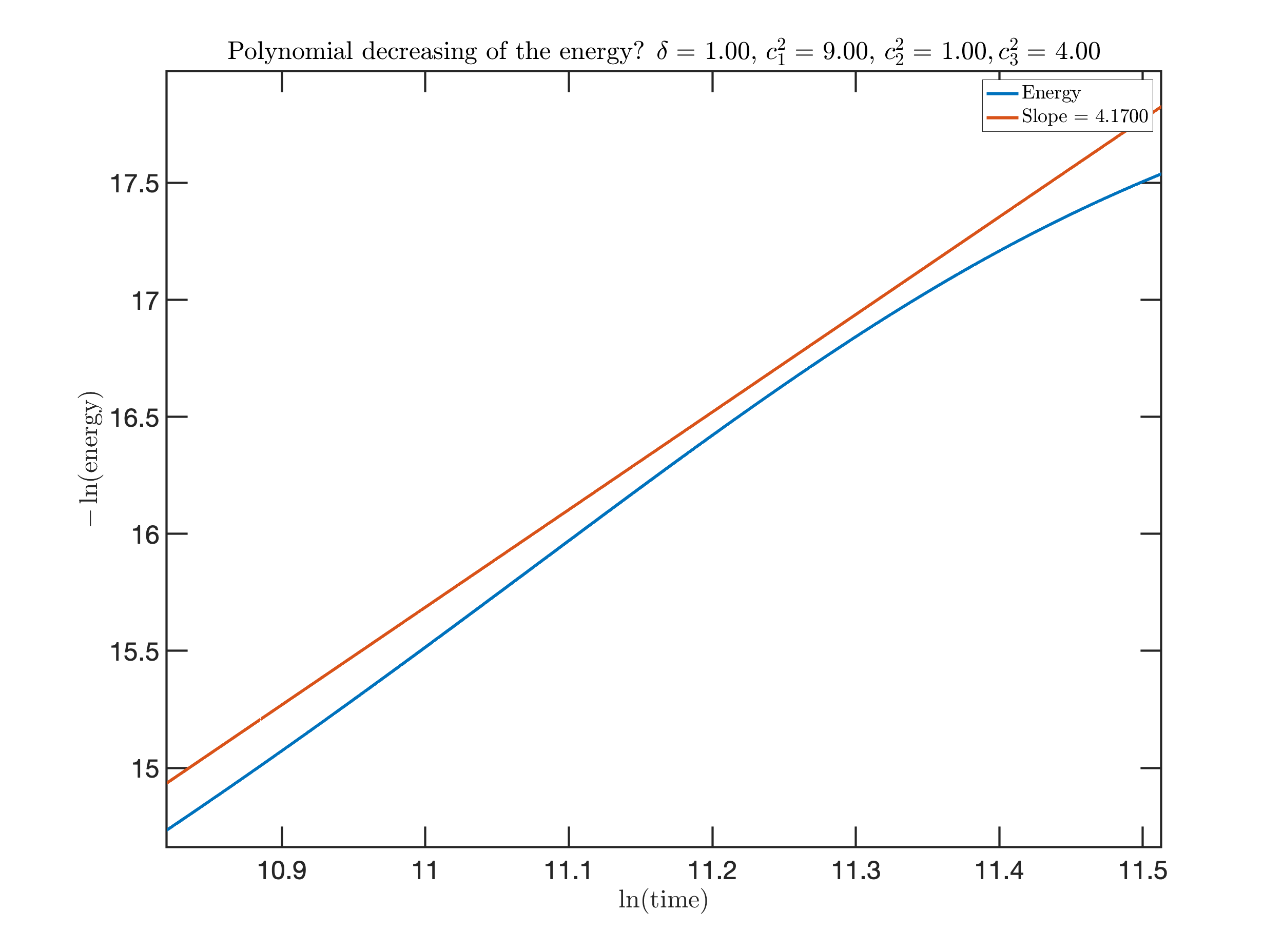}
    \subcaption{Polynomial decay} 
  \end{minipage}
  \hfill
  \begin{minipage}[b]{0.45\textwidth}
    \includegraphics[width=\textwidth]{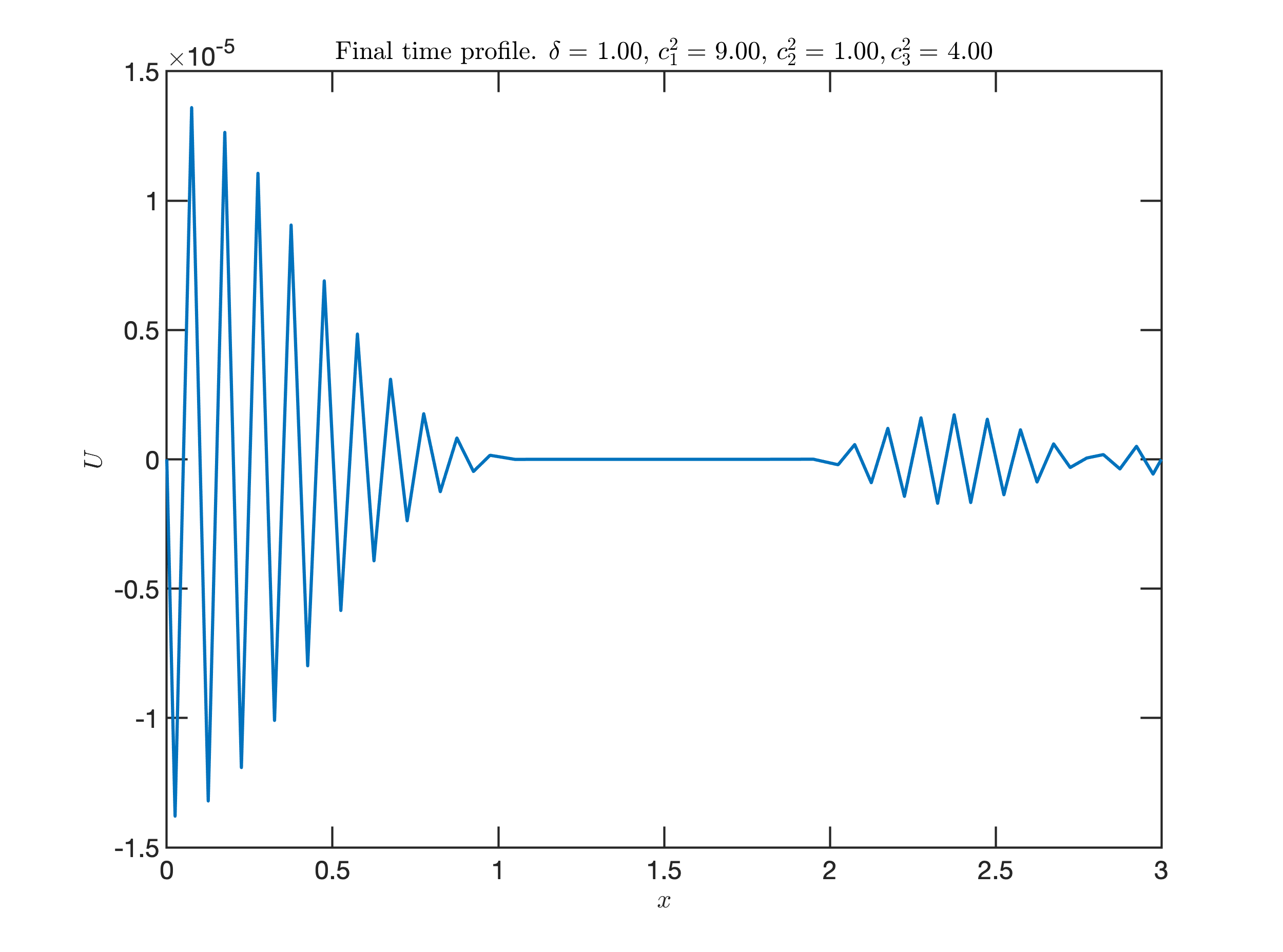}
    \subcaption{Final time profile} 
  \end{minipage}
  \caption{Long time behavior when $C_1^2=9, \  C_2^2 =1 , \  C_3^2=4$} \label{fig 5}
\end{figure}
\begin{figure}[!tbp]
  \centering
  \begin{minipage}[b]{0.45\textwidth}
    \includegraphics[width=\textwidth]{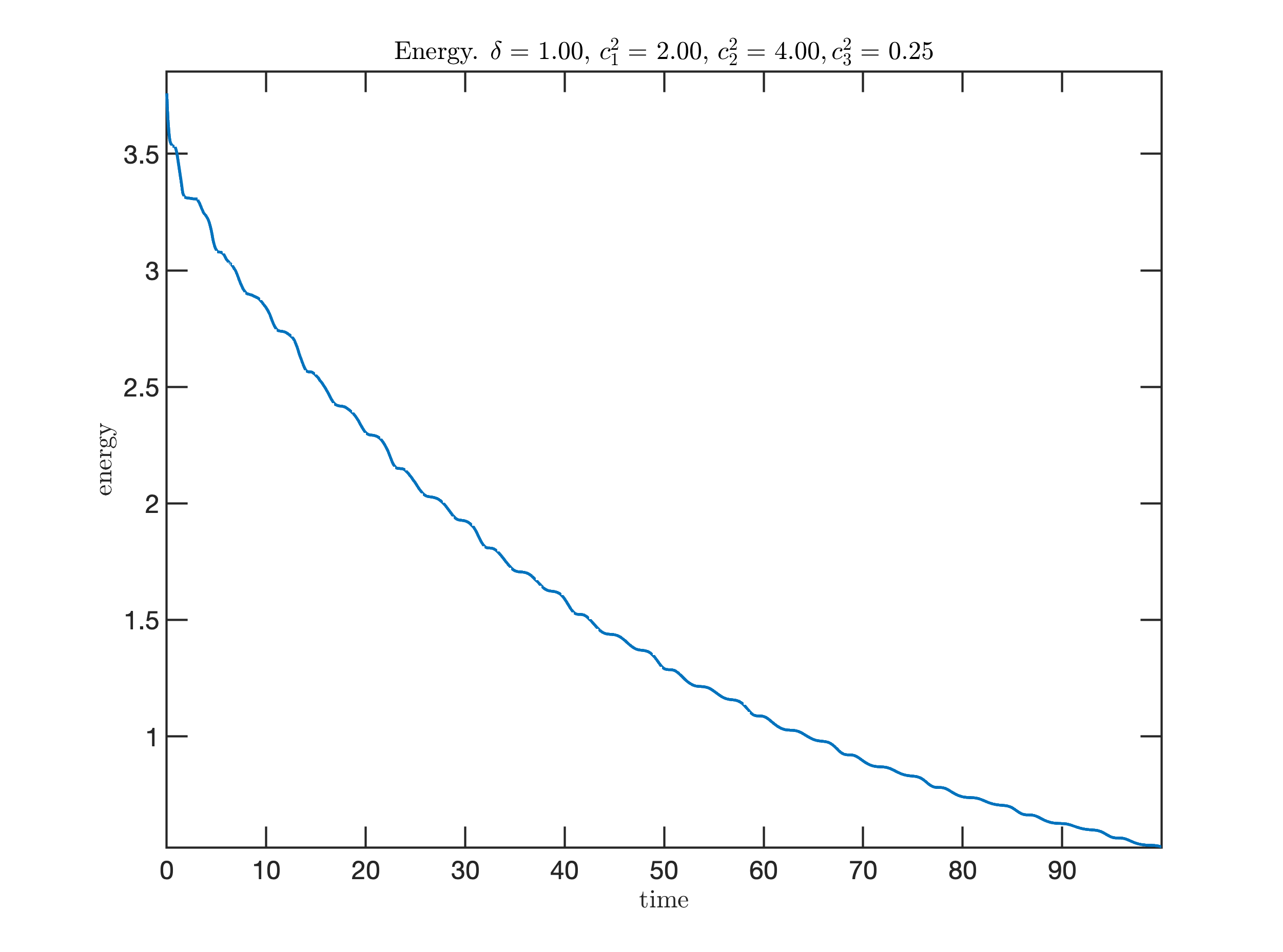}
    \subcaption{Energy ($t=0$ till $t=100$)} 
  \end{minipage}
  \hfill
  \begin{minipage}[b]{0.45\textwidth}
    \includegraphics[width=\textwidth]{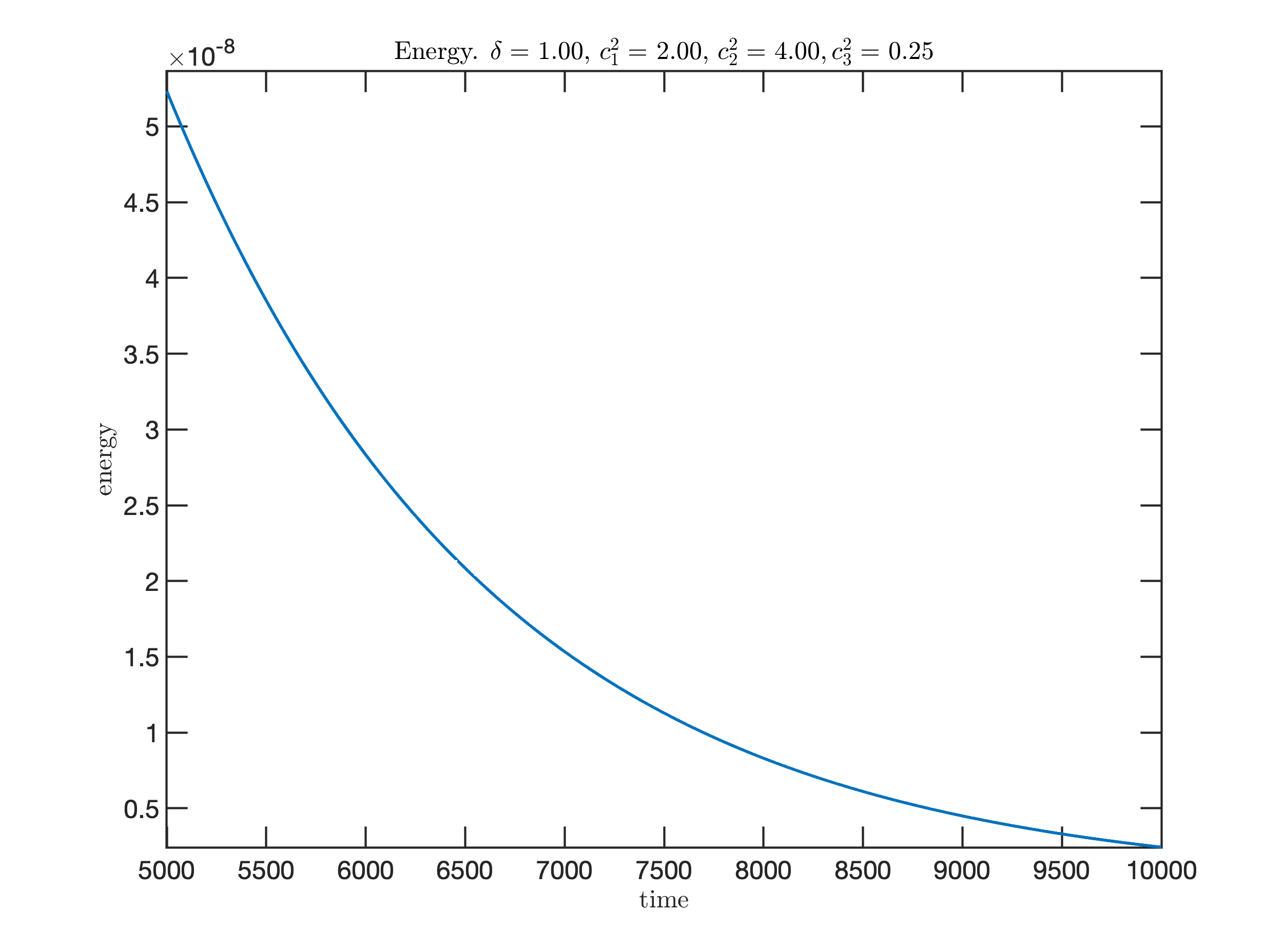}
    \subcaption{ Energy ($T/2$ till $T$)} 
  \end{minipage}
  \hfill
  \begin{minipage}[b]{0.45\textwidth}
    \includegraphics[width=\textwidth]{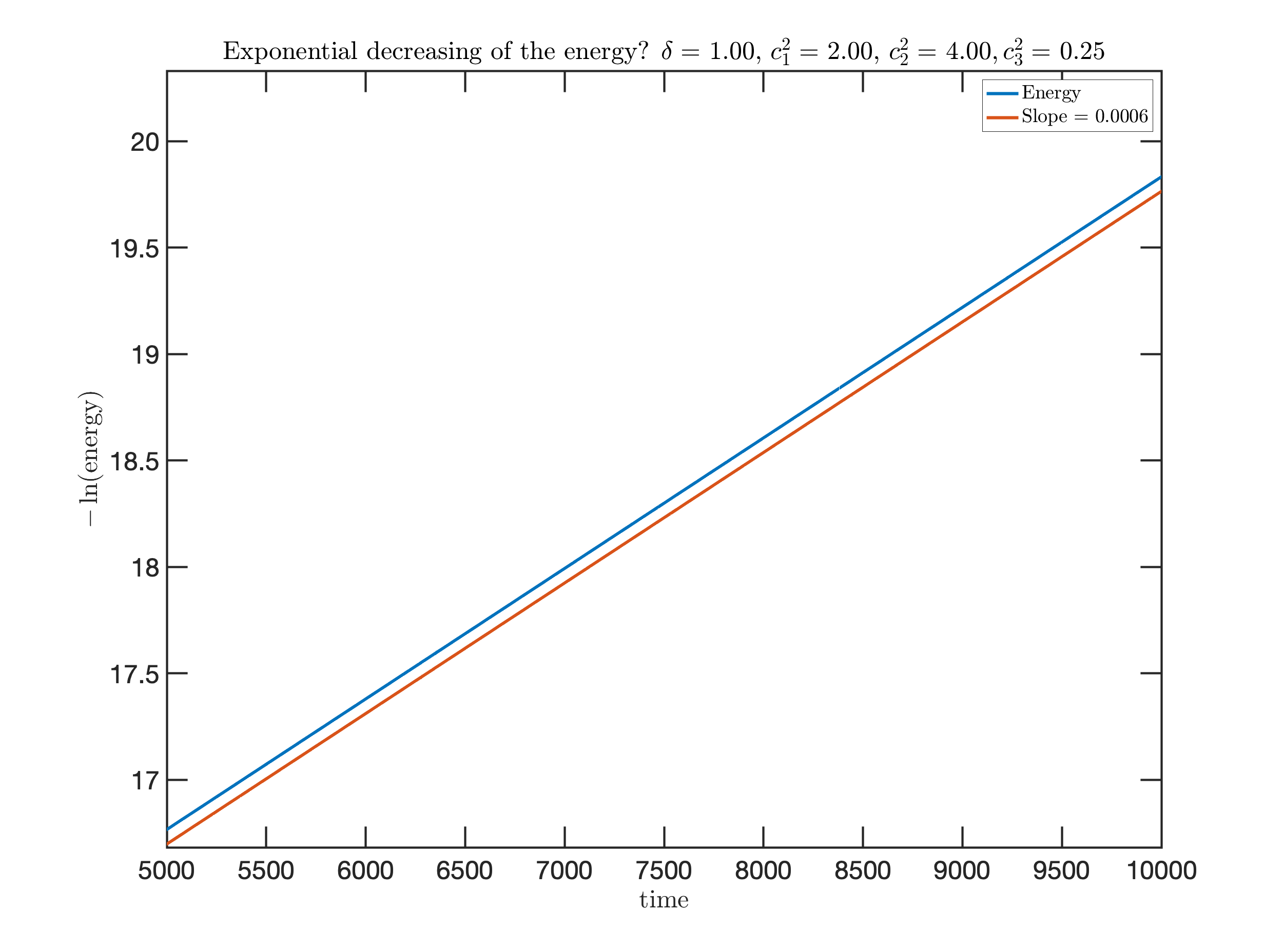}
    \subcaption{Exponential decay} 
  \end{minipage}
  \hfill
  \begin{minipage}[b]{0.45\textwidth}
    \includegraphics[width=\textwidth]{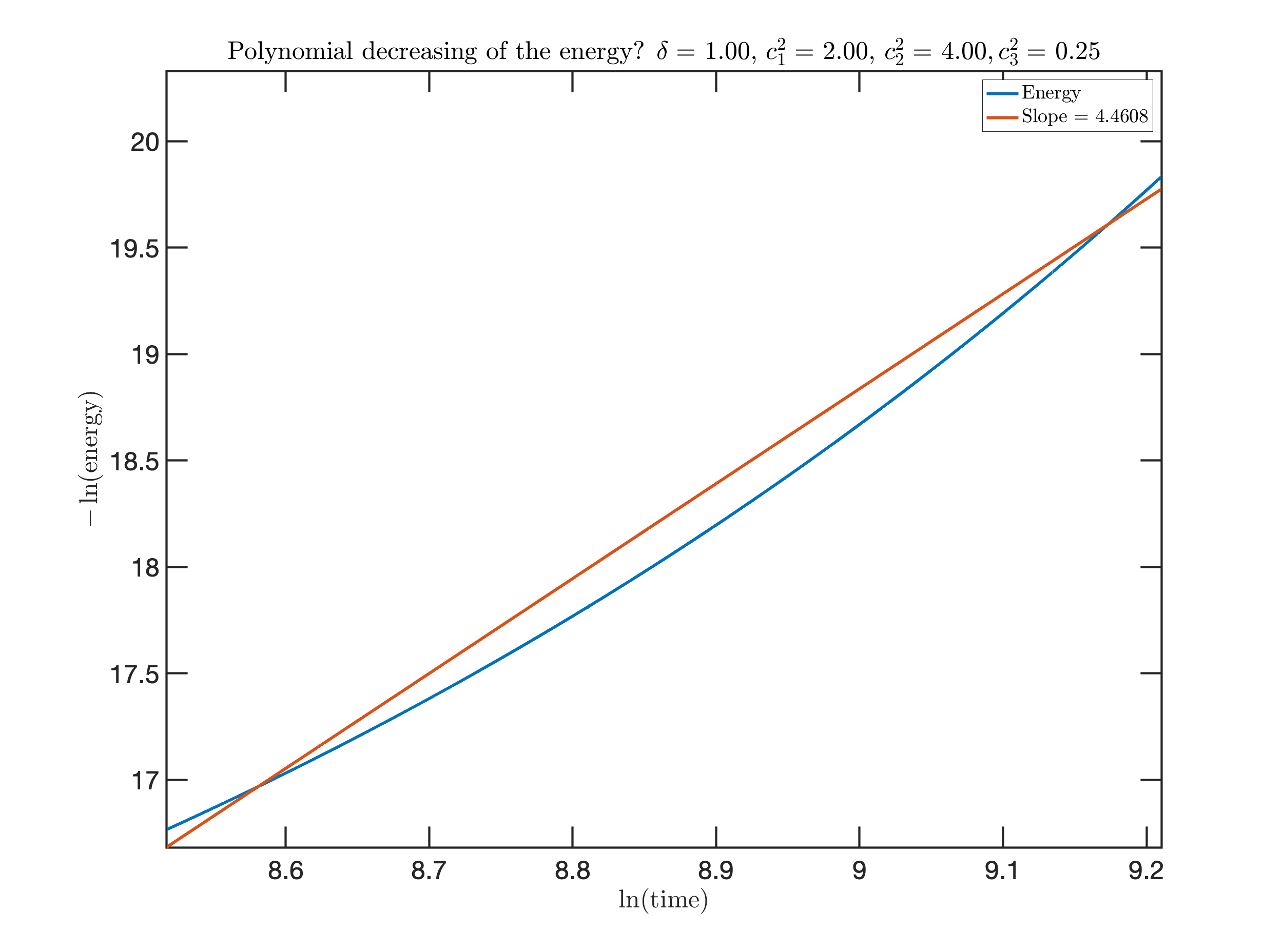}
    \subcaption{Polynomial decay} 
  \end{minipage}
  \hfill
  \begin{minipage}[b]{0.45\textwidth}
    \includegraphics[width=\textwidth]{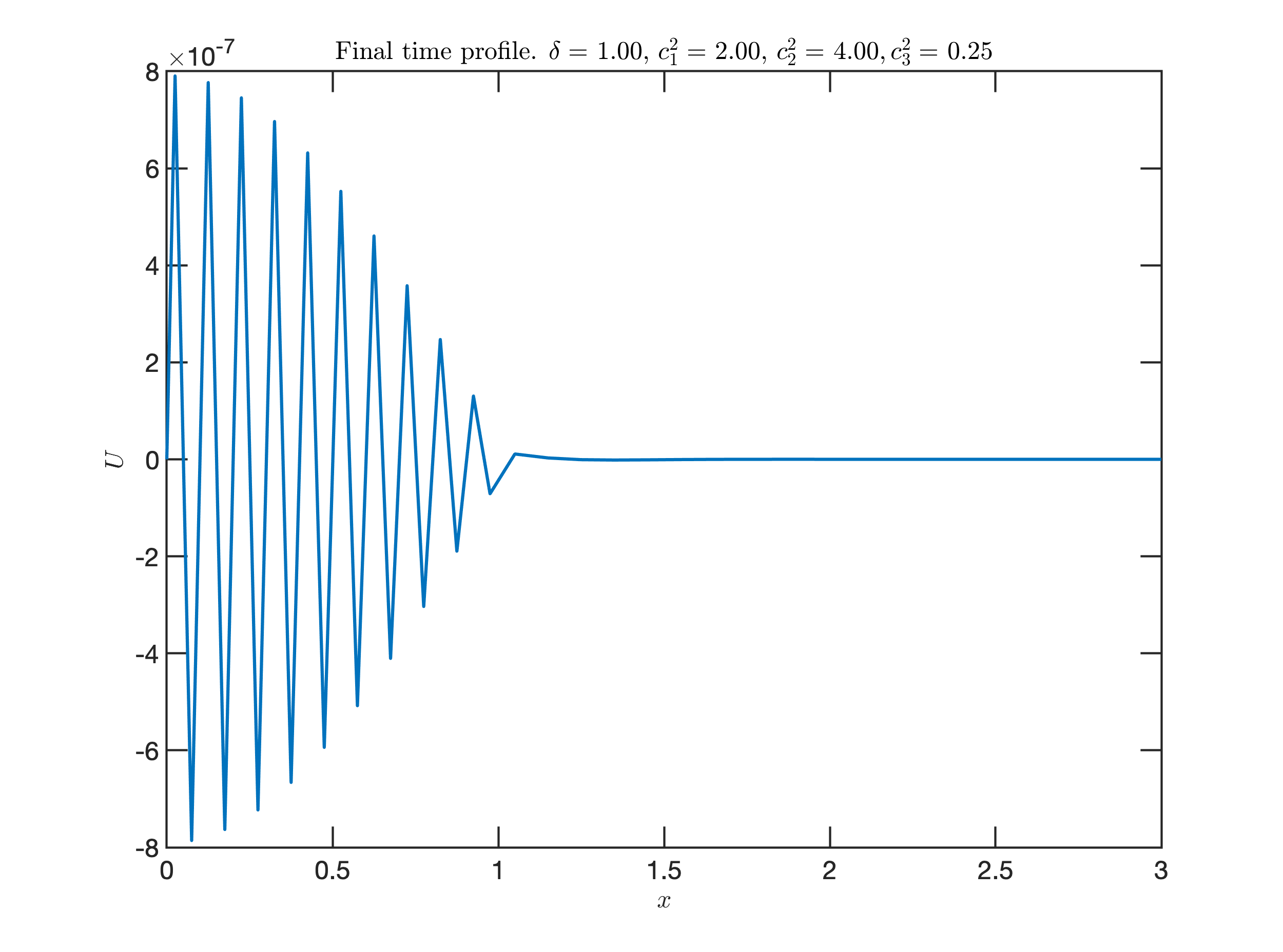}
    \subcaption{ Final time profile} 
  \end{minipage}
  \caption{Long time behavior when $C_1^2=2, \  C_2^2 =4 , \  C_3^2=0.25$} \label{fig 6}
\end{figure}
\begin{figure}[!tbp]
  \centering
  \begin{minipage}[b]{0.45\textwidth}
    \includegraphics[width=\textwidth]{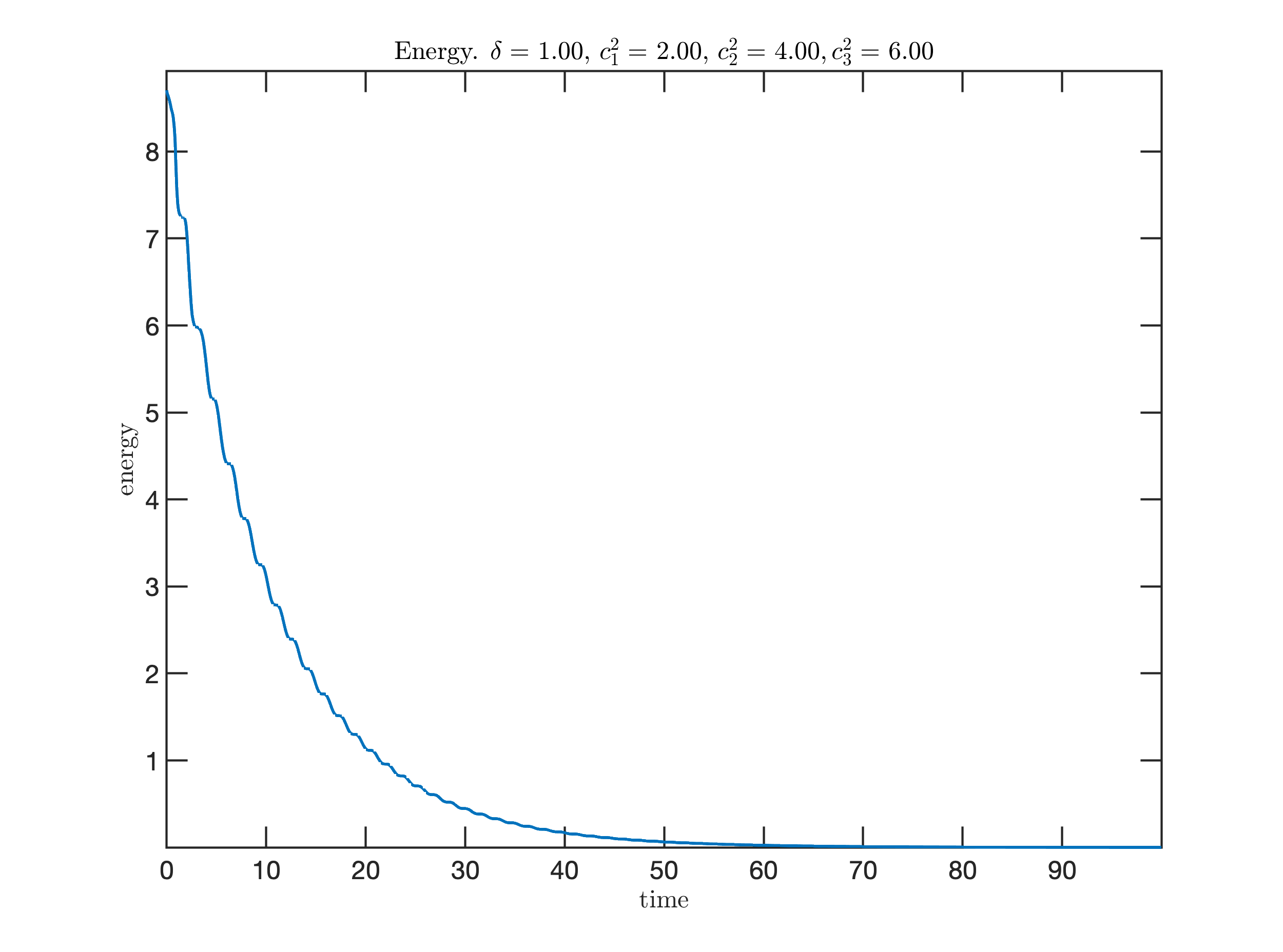}
    \subcaption{Energy ($t=0$ till $t=100$)} 
  \end{minipage}
  \hfill
  \begin{minipage}[b]{0.45\textwidth}
    \includegraphics[width=\textwidth]{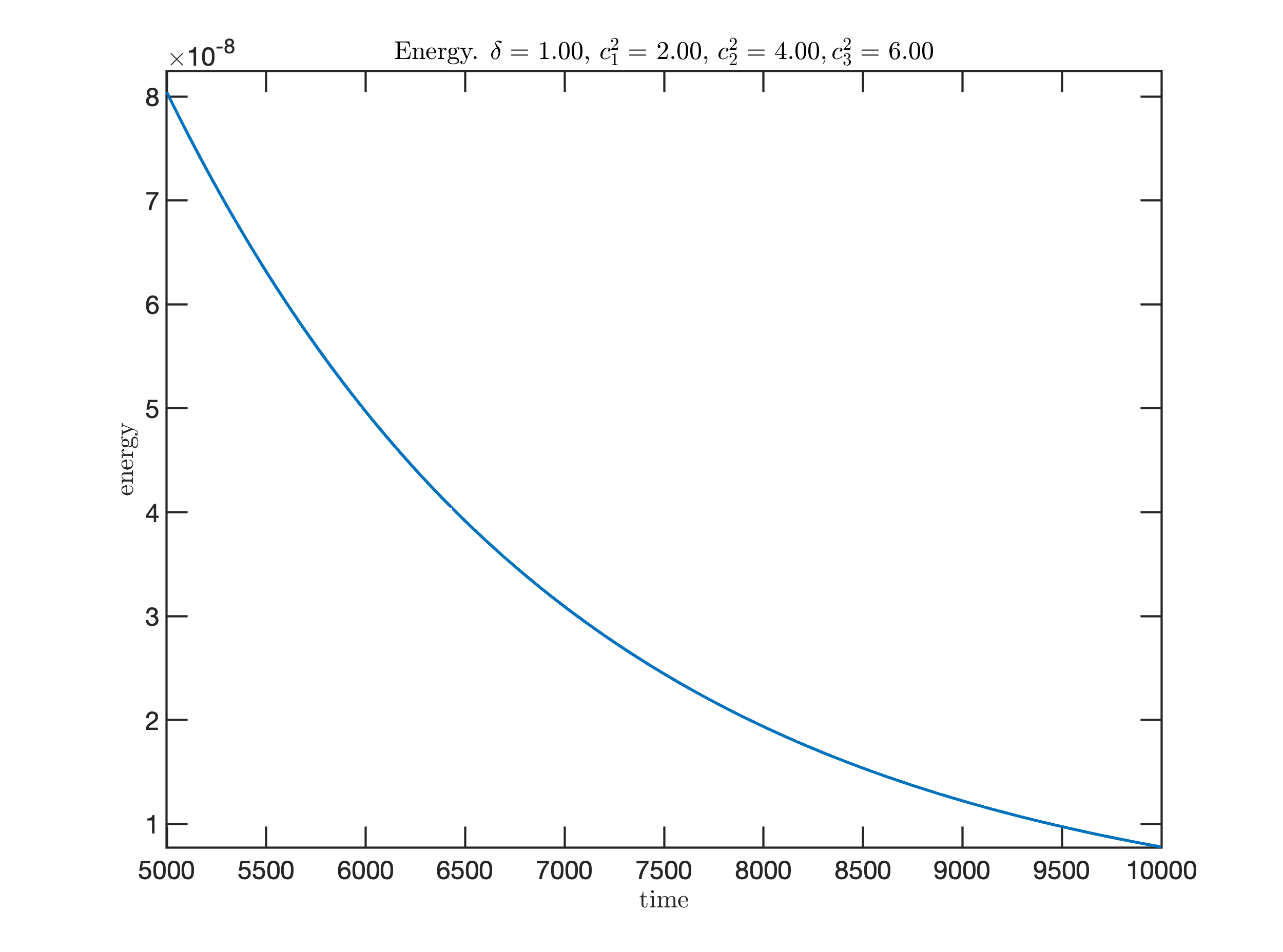}
    \subcaption{ Energy ($T/2$ till $T$)} 
  \end{minipage}
  \hfill
  \begin{minipage}[b]{0.45\textwidth}
    \includegraphics[width=\textwidth]{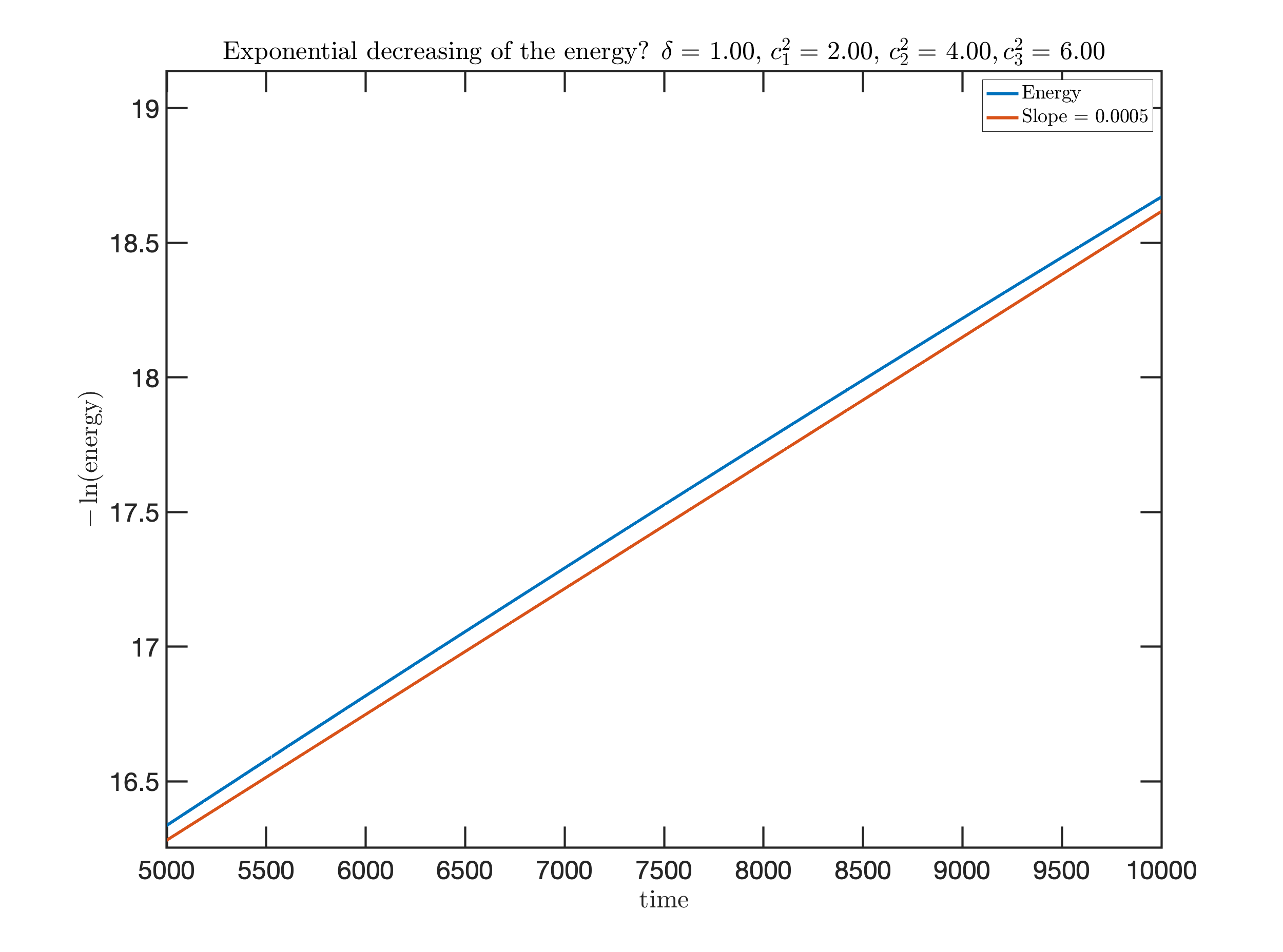}
    \subcaption{Exponential decay} 
  \end{minipage}
  \hfill
  \begin{minipage}[b]{0.45\textwidth}
    \includegraphics[width=\textwidth]{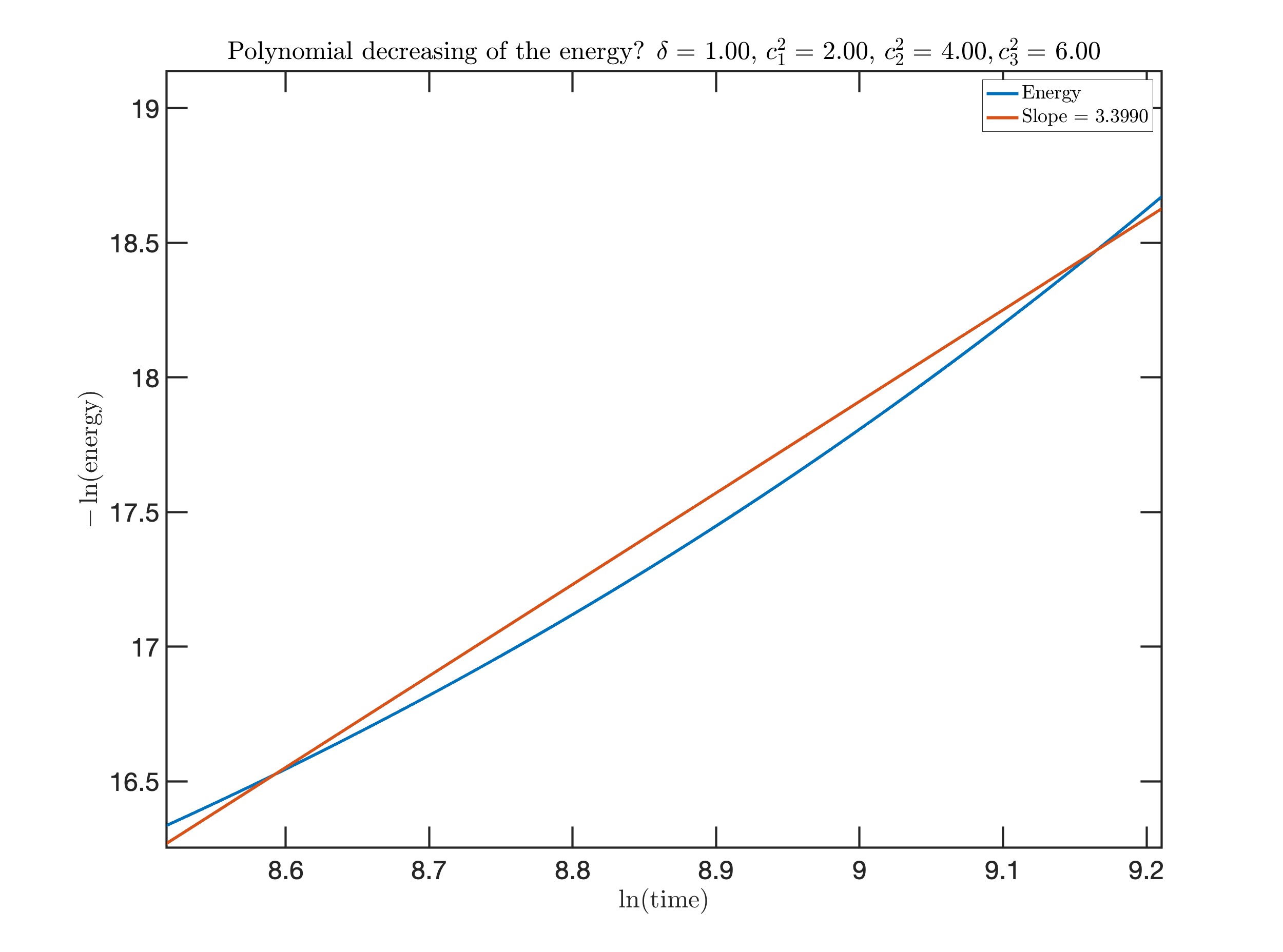}
    \subcaption{Polynomial decay} 
  \end{minipage}
  \hfill
  \begin{minipage}[b]{0.45\textwidth}
    \includegraphics[width=\textwidth]{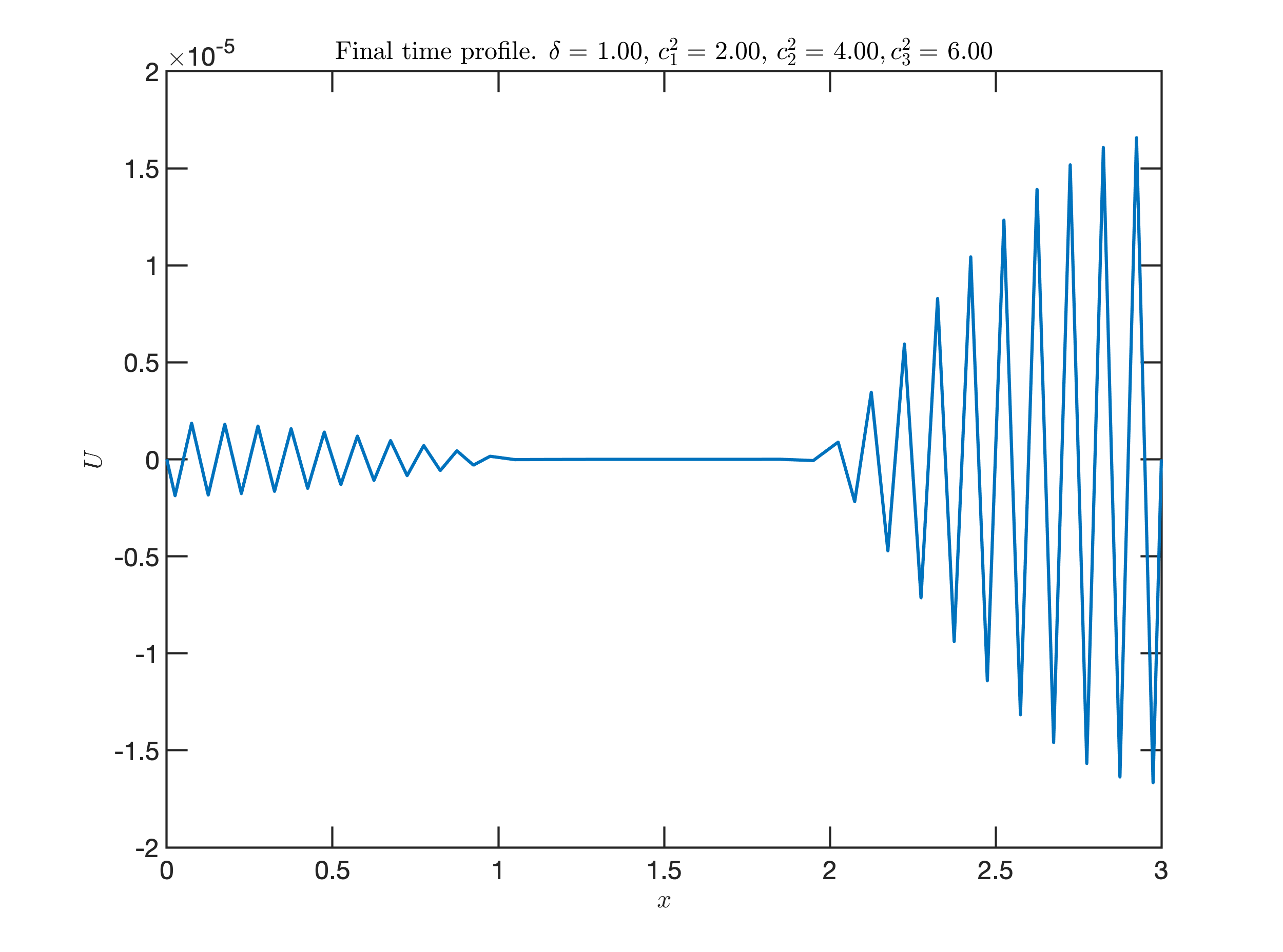}
    \subcaption{ Final time profile} 
  \end{minipage}
  \caption{Long time behavior when $C_1^2=2, \  C_2^2 =4 , \  C_3^2=6$} \label{fig 7}
\end{figure}
\begin{figure}[!tbp]
  \centering
  \begin{minipage}[b]{0.45\textwidth}
    \includegraphics[width=\textwidth]{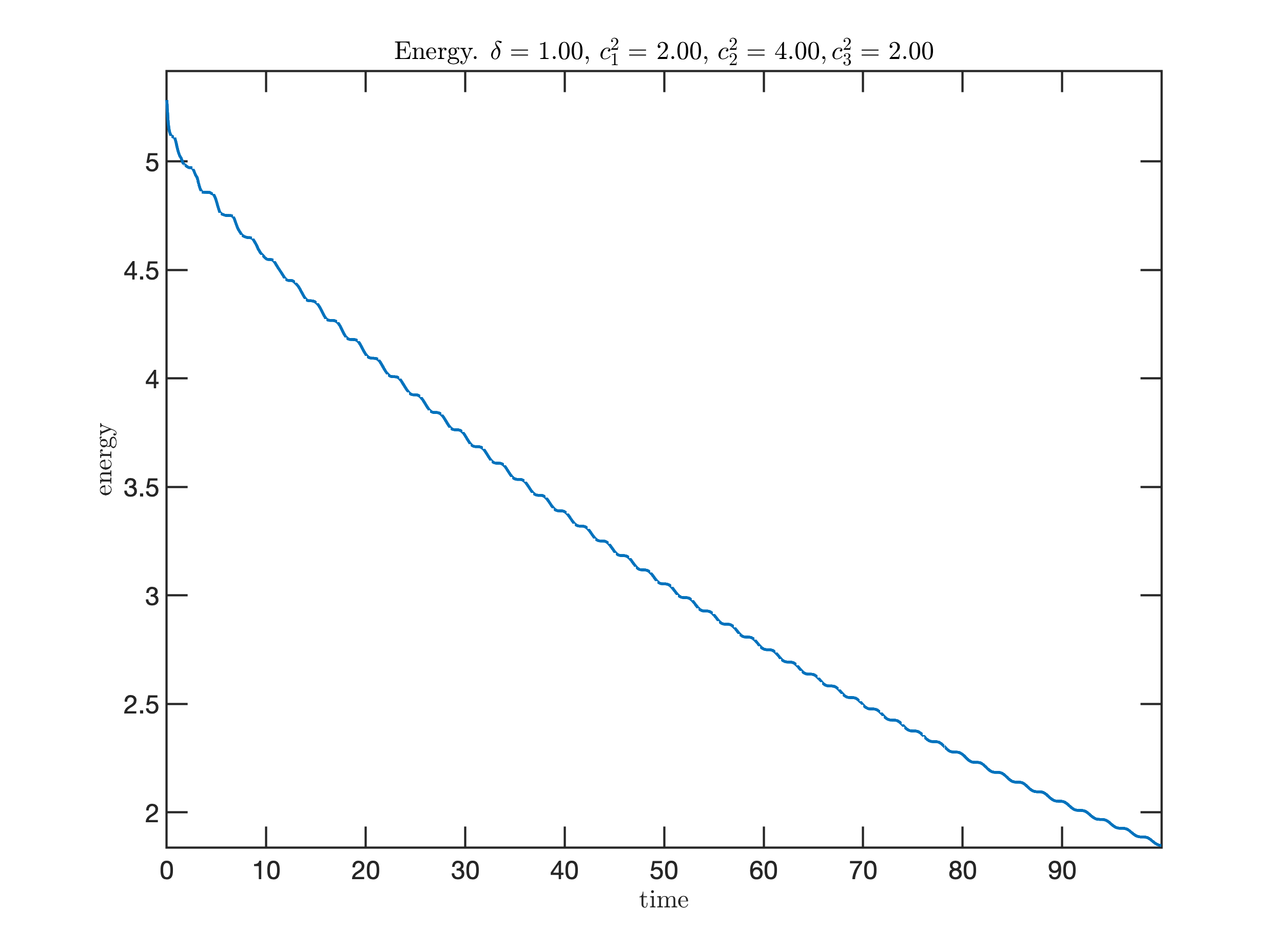}
    \subcaption{Energy ($t=0$ till $t=100$)} 
  \end{minipage}
  \hfill
  \begin{minipage}[b]{0.45\textwidth}
    \includegraphics[width=\textwidth]{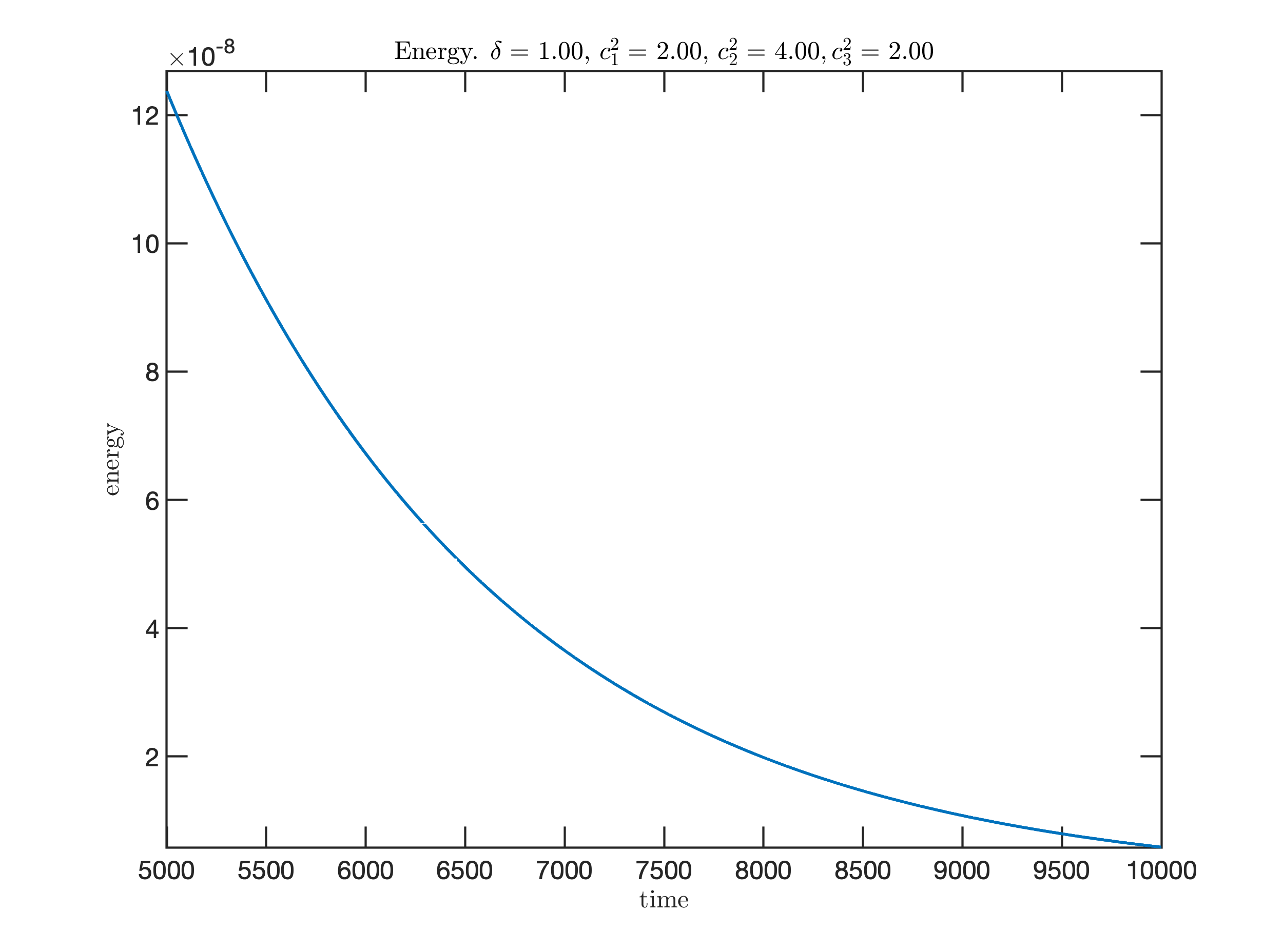}
    \subcaption{ Energy ($T/2$ till $T$)} 
  \end{minipage}
  \hfill
  \begin{minipage}[b]{0.45\textwidth}
    \includegraphics[width=\textwidth]{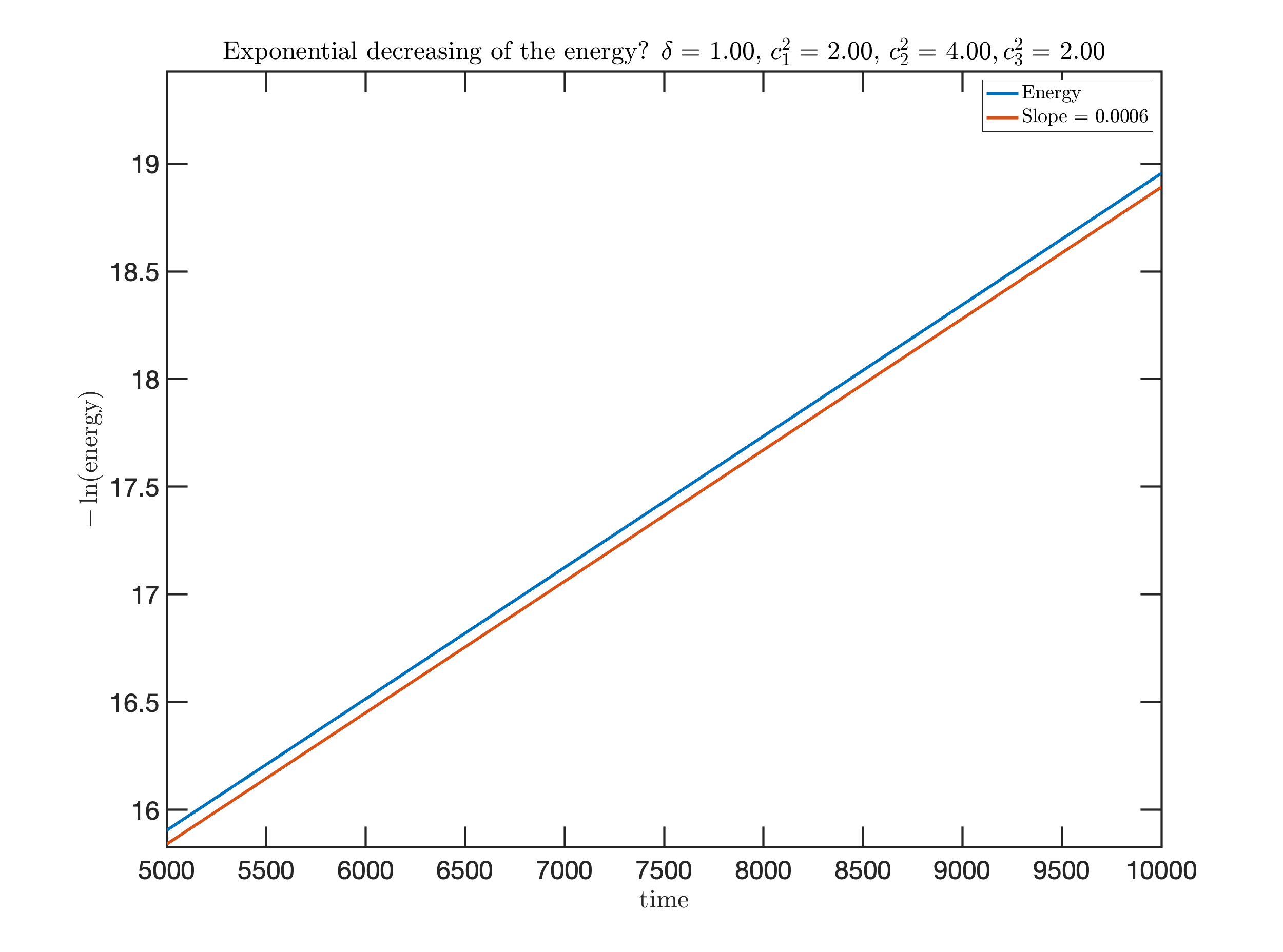}
    \subcaption{Exponential decay} 
  \end{minipage}
  \hfill
  \begin{minipage}[b]{0.45\textwidth}
    \includegraphics[width=\textwidth]{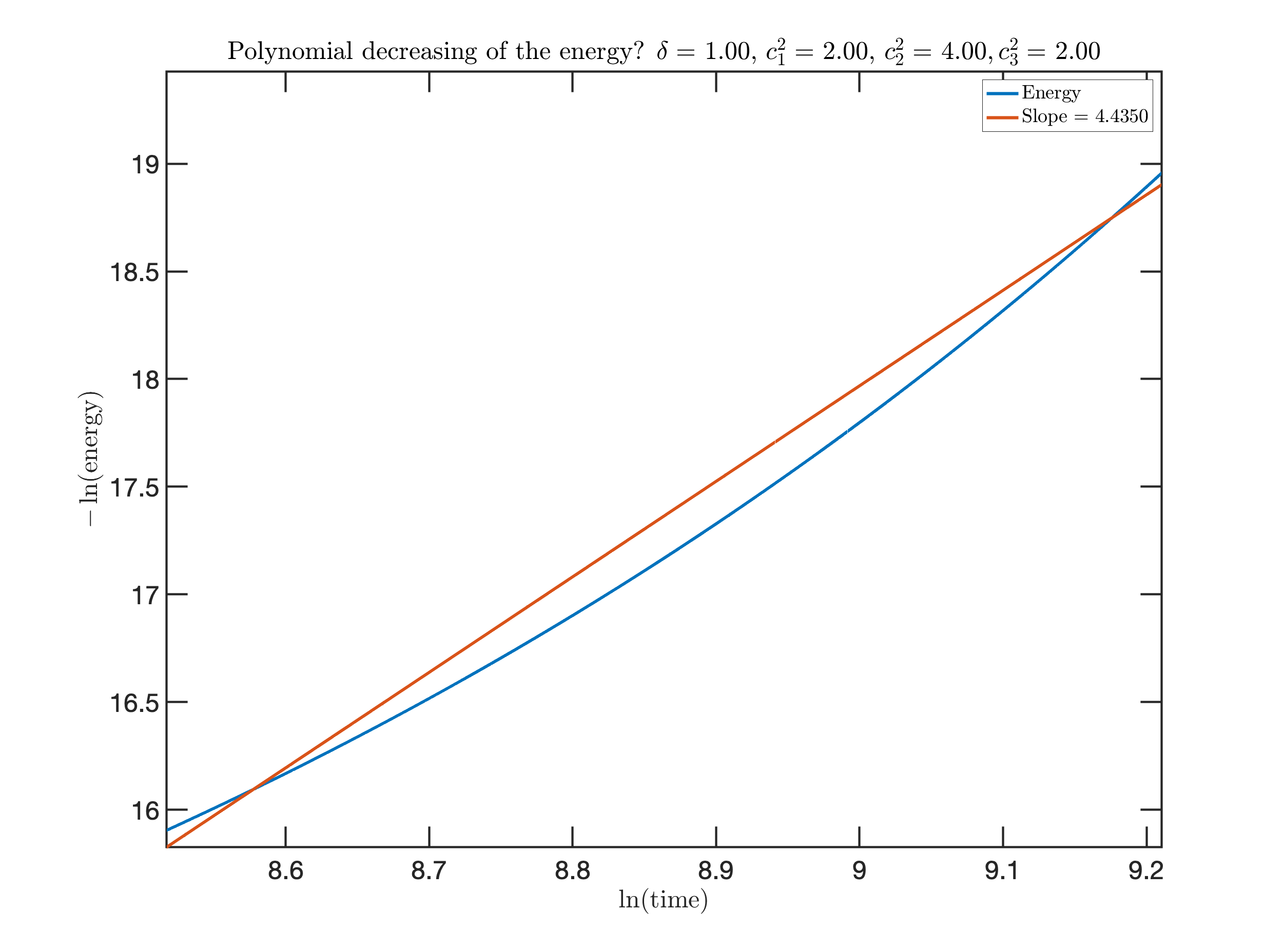}
    \subcaption{Polynomial decay} 
  \end{minipage}
  \hfill
  \begin{minipage}[b]{0.45\textwidth}
    \includegraphics[width=\textwidth]{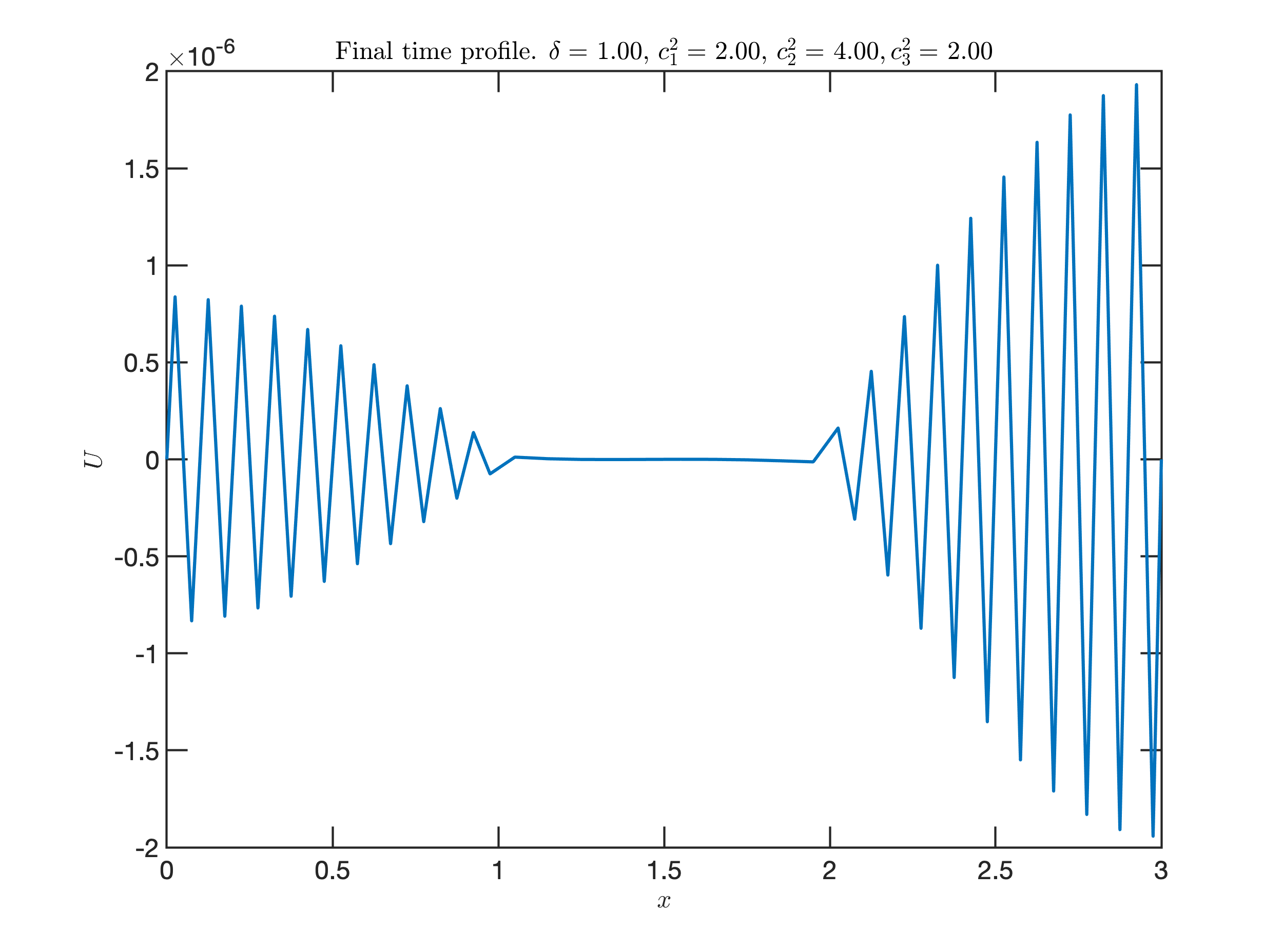}
    \subcaption{Final time profile} 
  \end{minipage}
  \caption{Long time behavior when $C_1^2=C_3^2=2, \  C_2^2 =4$ } \label{fig 8}
\end{figure}

\begin{figure}[!tbp]
  \centering
  \begin{minipage}[b]{0.45\textwidth}
    \includegraphics[width=\textwidth]{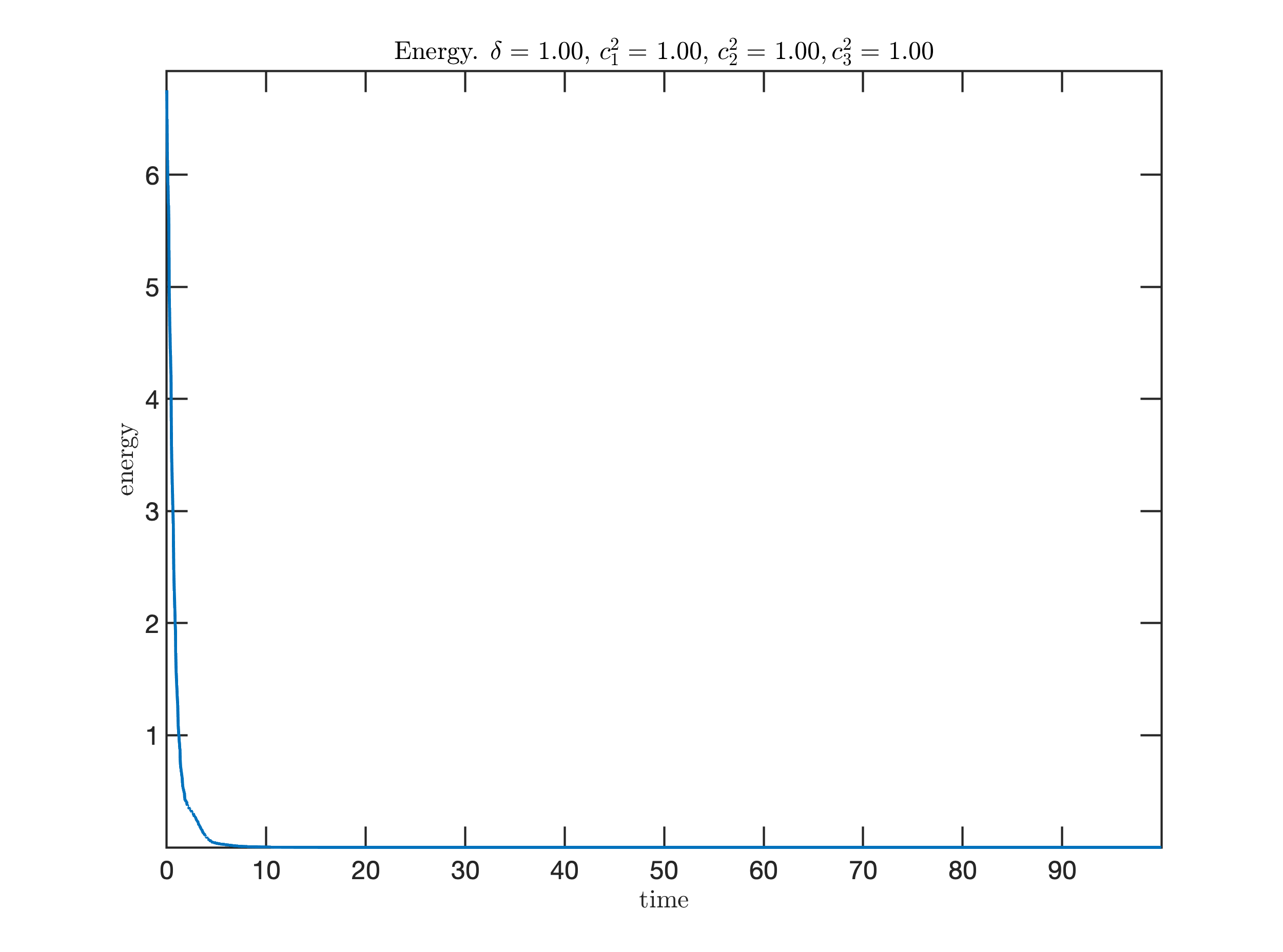}
    \subcaption{Energy ($t=0$ till $t=100$)} 
  \end{minipage}
  \hfill
  \begin{minipage}[b]{0.45\textwidth}
    \includegraphics[width=\textwidth]{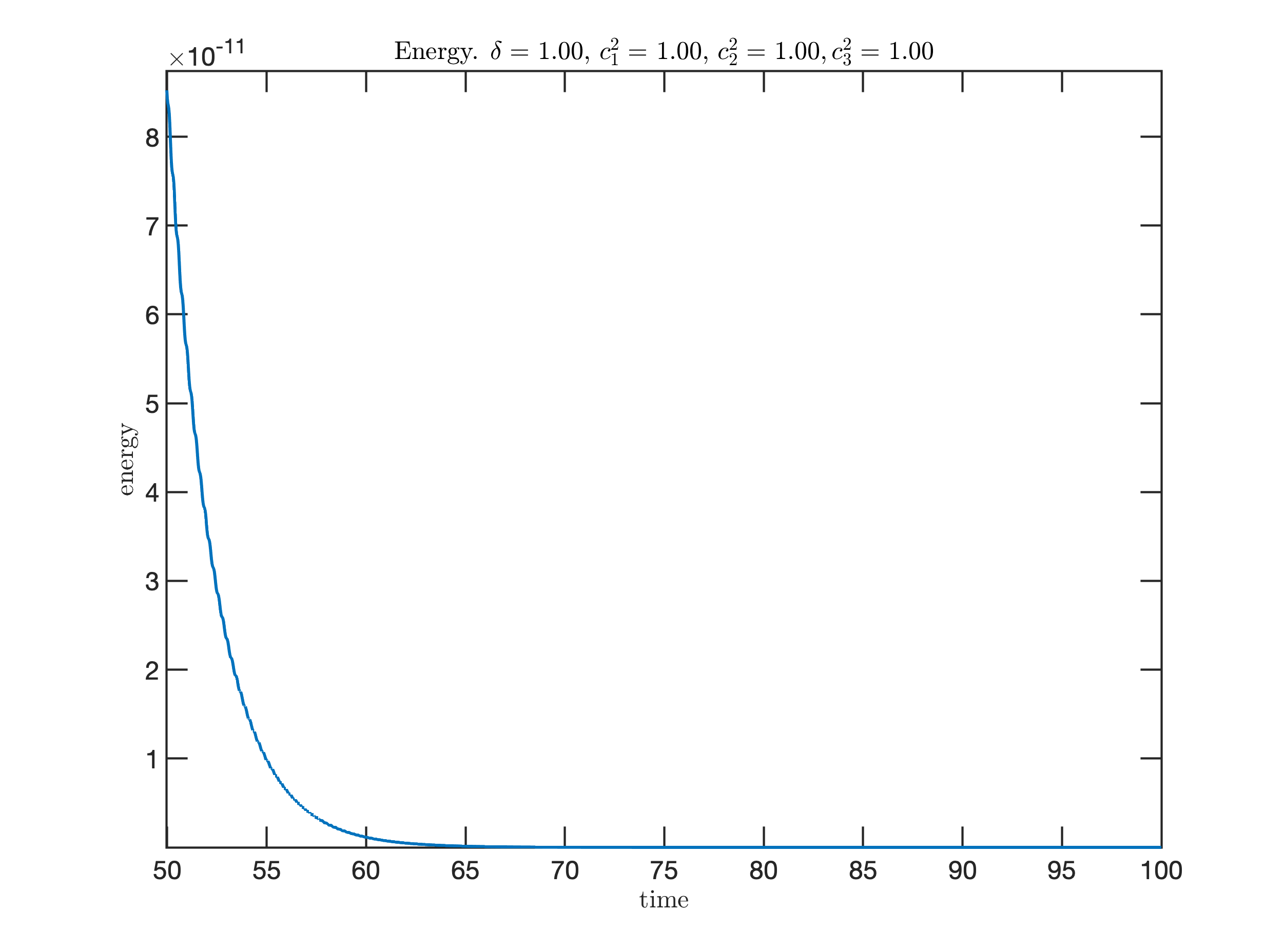}
    \subcaption{ Energy ($T/2$ till $T$)} 
  \end{minipage}
  \hfill
  \begin{minipage}[b]{0.45\textwidth}
    \includegraphics[width=\textwidth]{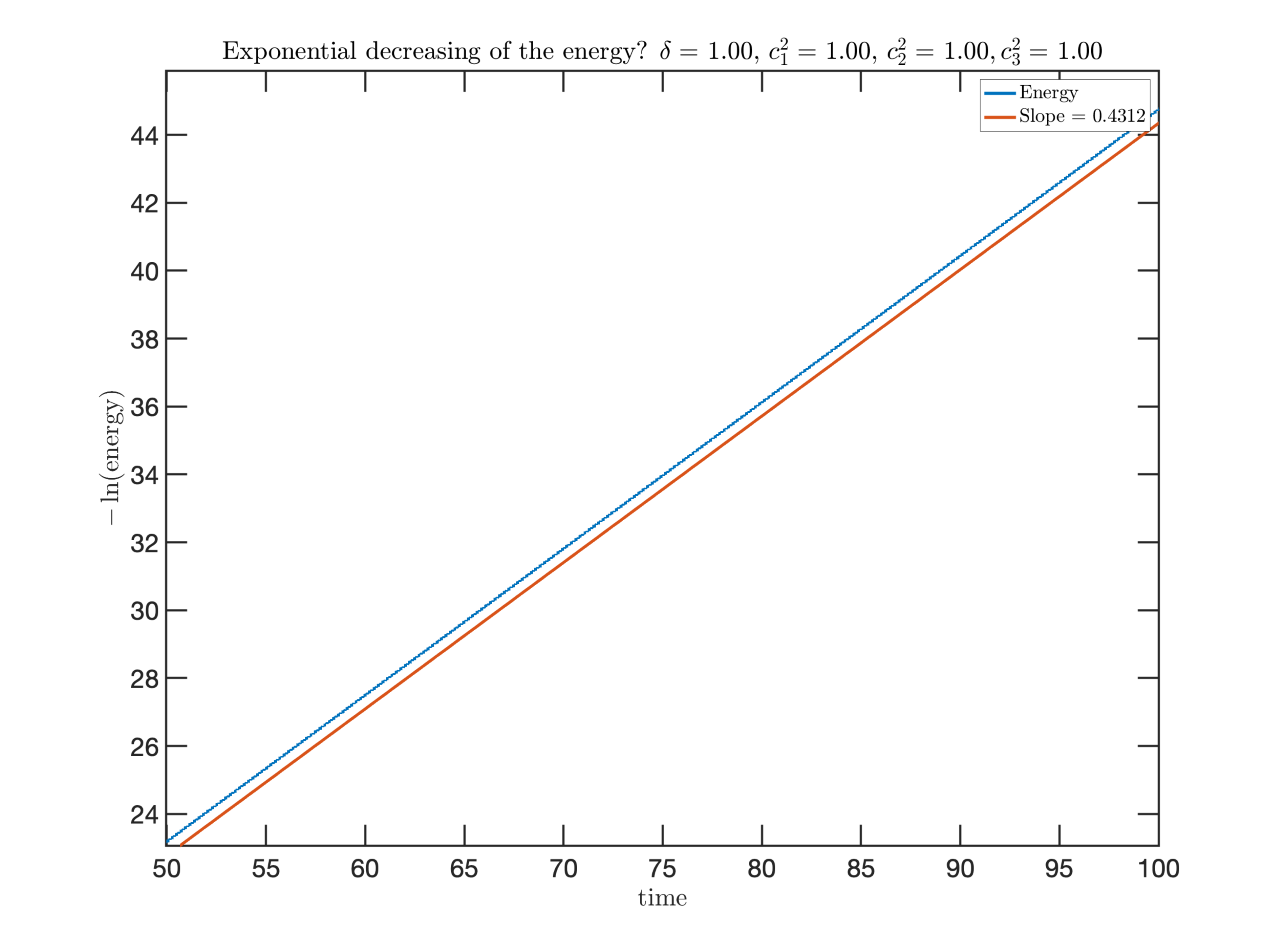}
    \subcaption{Exponential decay} 
  \end{minipage}
  \hfill
  \begin{minipage}[b]{0.45\textwidth}
    \includegraphics[width=\textwidth]{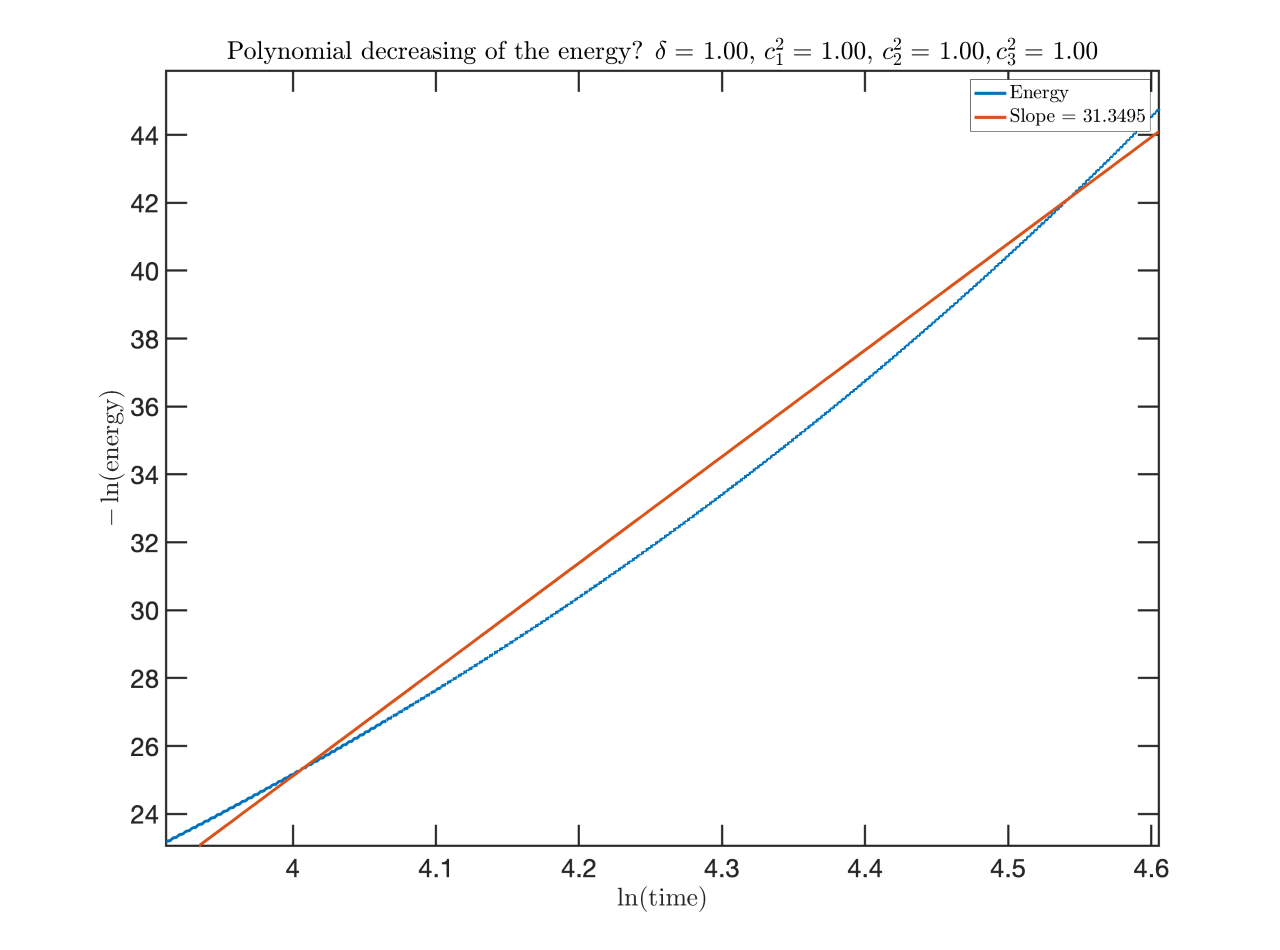}
    \subcaption{Polynomial decay} 
  \end{minipage}
  \hfill
  \begin{minipage}[b]{0.45\textwidth}
    \includegraphics[width=\textwidth]{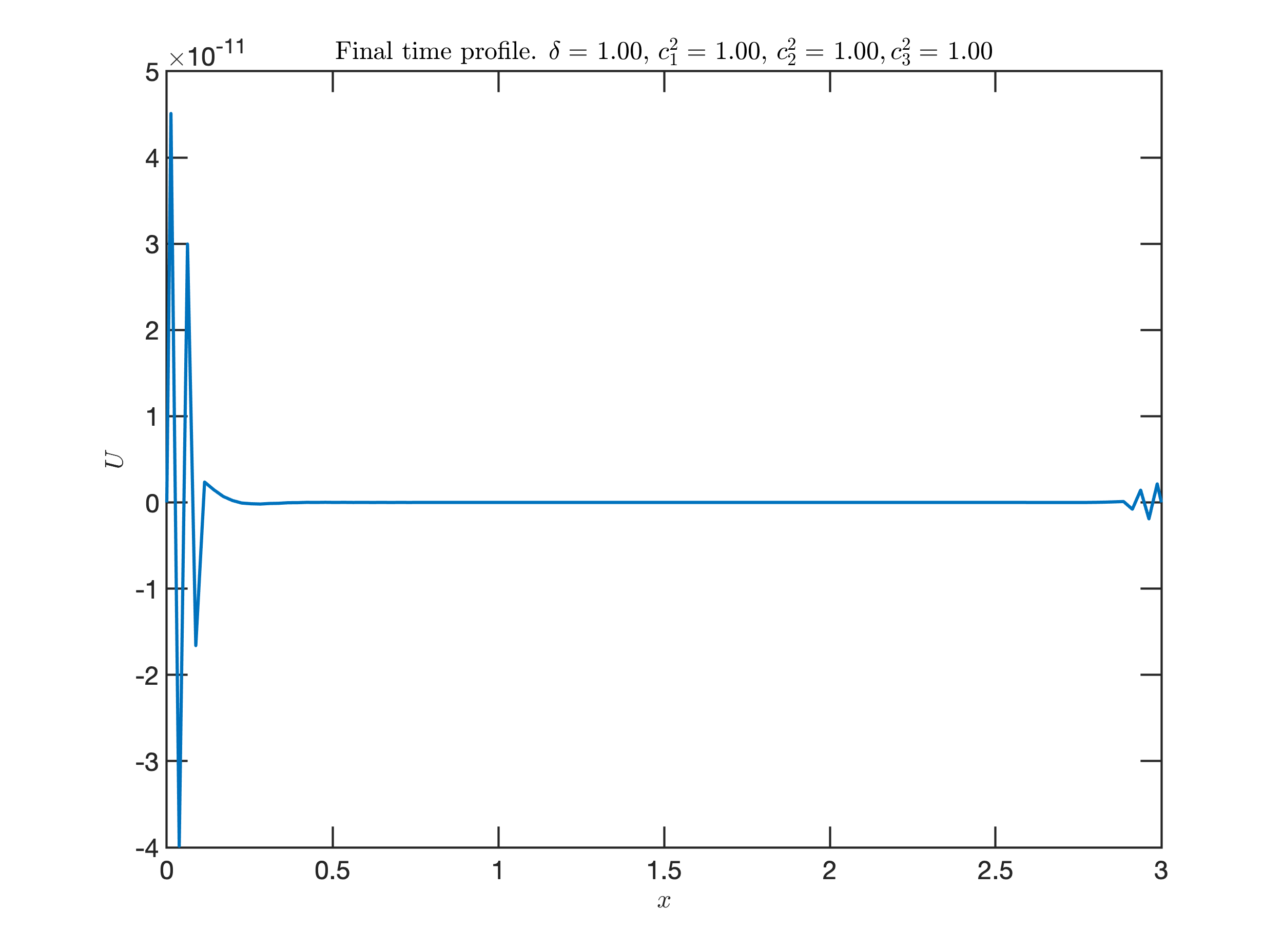}
    \subcaption{Final time profile} 
  \end{minipage}
  \caption{Exponential decay under same propagation speed and more viscoelastic material} \label{fig 10}
\end{figure}

\end{document}